\documentclass[12pt]{article}
\usepackage[final]{epsfig}
\usepackage{graphics,xcolor}
\usepackage{amsmath,amsthm}
\usepackage{amsfonts}
\usepackage{latexsym}
\usepackage{amssymb}
\usepackage{graphicx}
\usepackage{url}
\usepackage{epstopdf}
\usepackage{amscd}
\usepackage{tikz-cd}
\usepackage{hyperref,url}

\newtheorem{lemma}{Lemma}[section]
\newtheorem{proposition}[lemma]{Proposition}
\newtheorem{theorem}{Theorem}
\newtheorem{mainthm}{Main Theorem}
\newtheorem{corollary}[lemma]{Corollary}
\newtheorem{conjecture}[lemma]{Conjecture}

\theoremstyle{definition}
\newtheorem{example}[lemma]{Example}
\newtheorem{definition}[lemma]{Definition}
\newtheorem{remark}[lemma]{Remark}

\begin{document}

\newcommand{\eps}{{\varepsilon}}
\newcommand{\proofend}{$\Box$\bigskip}
\newcommand{\C}{{\mathbb C}}
\newcommand{\Q}{{\mathbb Q}}
\newcommand{\R}{{\mathbb R}}
\newcommand{\Z}{{\mathbb Z}}
\newcommand{\RP}{{\mathbb {RP}}}
\newcommand{\CP}{{\mathbb {CP}}}
\newcommand{\KP}{{\mathbb {KP}}}
\newcommand{\PP}{{\mathbb{P}}}
\newcommand{\KK}{{\mathbb{K}}}
\newcommand{\Tr}{\operatorname{tr}}
\newcommand{\HH}{{\mathbb H}}

\newcommand{\AP}{\operatorname{AP}}
\newcommand{\Id}{\operatorname{Id}}
\renewcommand{\emptyset}{\varnothing}
\renewcommand{\epsilon}{\varepsilon}
\newcommand{\SL}{\operatorname{SL}}
\renewcommand{\sl}{\operatorname{sl}}
\newcommand{\GL}{\operatorname{GL}}
\newcommand{\PGL}{\operatorname{PGL}}
\newcommand{\PSL}{\operatorname{PSL}}

\newcommand{\T}{\stackrel{\alpha}{\sim}}
\newcommand{\TT}{\stackrel{\beta}{\sim}}
\newcommand{\pP}{\mathbf{P}}
\newcommand{\pQ}{\mathbf{Q}}
\newcommand{\pR}{\mathbf{R}}
\newcommand{\pS}{\mathbf{S}}
\newcommand{\cP}{\mathcal{P}_n}
\newcommand{\cT}{\mathcal{T}_n}
\newcommand{\cF}{\mathcal{F}_n}
\newcommand{\cV}{\mathcal{V}_n}

\newcommand{\dist}{\operatorname{dist}}
\newcommand{\artanh}{\operatorname{artanh}}

\newcommand{\pc}[1]{\marginpar{\small\color{red} #1}}

\title{Cross-ratio dynamics on ideal polygons}

\author{Maxim Arnold\footnote{
Department of Mathematics, 
University of Texas, 
800 West Campbell Road,
Richardson, TX 75080;
maxim.arnold@utdallas.edu}
\and
Dmitry Fuchs\footnote{
Department of Mathematics, 
University of California, 
Davis, CA 95616;
 fuchs@math.ucdavis.edu}
\and
Ivan Izmestiev\footnote{
Department of Mathematics, 
University of Fribourg,
Chemin du Mus\'ee 23,
CH-1700 Fribourg;
ivan.izmestiev@unifr.ch 
}
\and
Serge Tabachnikov\footnote{
Department of Mathematics,
Pennsylvania State University,
University Park, PA 16802;
tabachni@math.psu.edu}}

\date{}
\maketitle

\begin{abstract}
Two ideal polygons, $(p_1,\ldots,p_n)$ and $(q_1,\ldots,q_n)$,
in the hyperbolic plane or in hyperbolic space are said to be $\alpha$-related if the cross-ratio $[p_i,p_{i+1},q_i,q_{i+1}] = \alpha$ for all $i$ (the vertices lie on the projective line, real or complex, respectively). 
For example, if $\alpha = -1$, the respective sides of the two polygons are orthogonal. 
This relation extends to twisted ideal polygons, that is, polygons with monodromy, and it descends to the moduli space of M\"obius-equivalent polygons. We prove that this relation, which is, generically, a 2-2 map, is completely integrable in the sense of Liouville. We describe integrals and invariant Poisson structures, and show that these relations, with different values of the constants $\alpha$, commute, in an appropriate sense. We investigate the case of small-gons, describe the exceptional ideal polygons, that possess infinitely many $\alpha$-related polygons, and study  the ideal polygons that are  $\alpha$-related to themselves (with a cyclic shift of the indices). 
\end{abstract}

\tableofcontents

\section{Introduction} \label{Intro}
\subsection{Motivation: iterations of evolutes} \label{motivsection}
The motivation for this work comes from our recent study of iterations of the evolutes and involutes of smooth curves and polygons \cite{AFITT,FT}. 

The evolute of a curve is the locus of the centers of its osculating circles, 
that is, the circles that pass through three ``consecutive" points of the 
curve. One of the definitions of the evolute of a polygon $\pP$ is that it is 
the polygon $\pQ$ formed by the centers of the circles that pass through 
consecutive triples of vertices of $\pP$. In other words, the vertices of 
$\pQ$ 
are 
the intersection points of the perpendicular bisectors of the adjacent sides 
of $\pP$, see Figure \ref{polyevol} left.  

\begin{figure}[hbtp]
\centering
\includegraphics[height=2in]{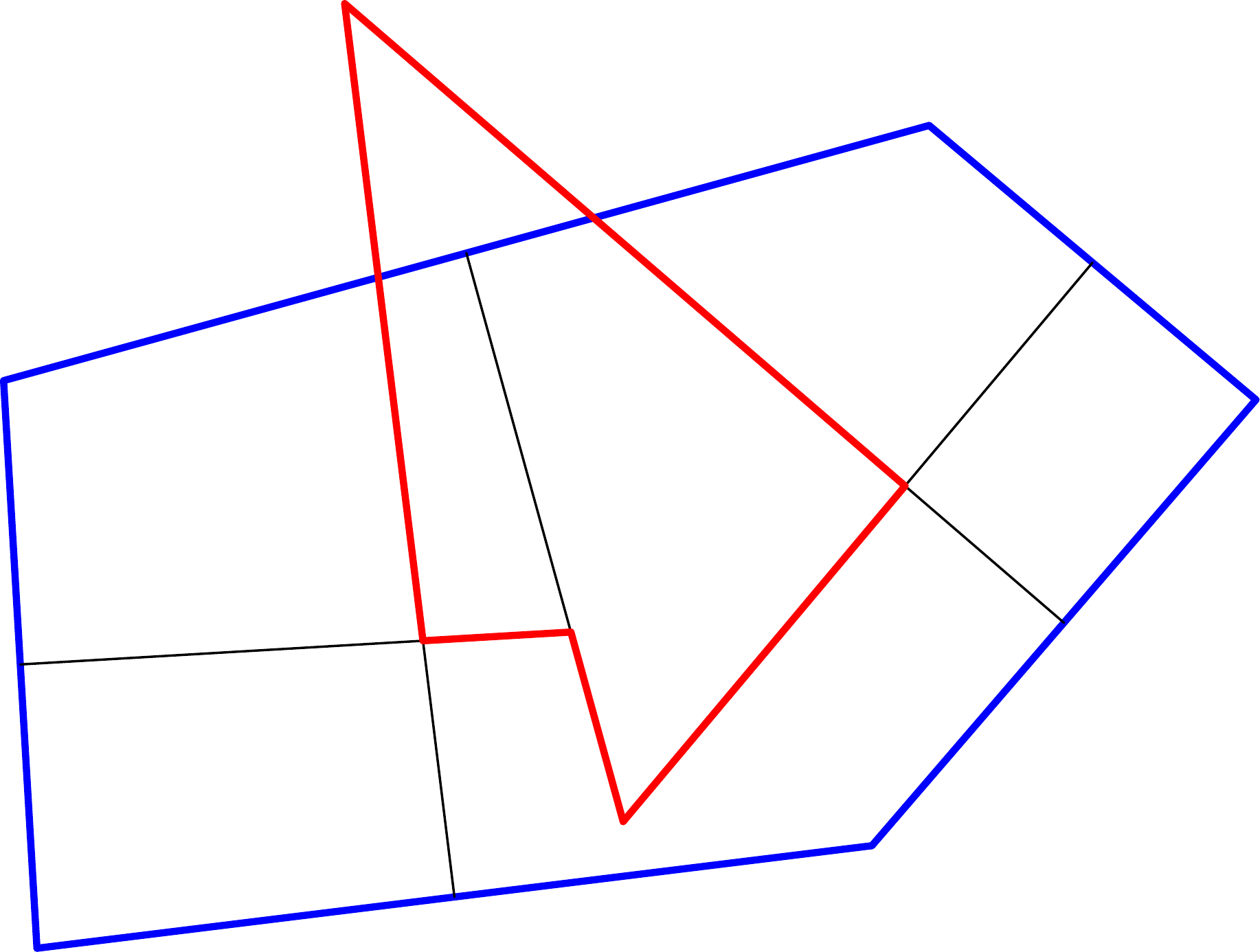}\qquad
\includegraphics[height=2in]{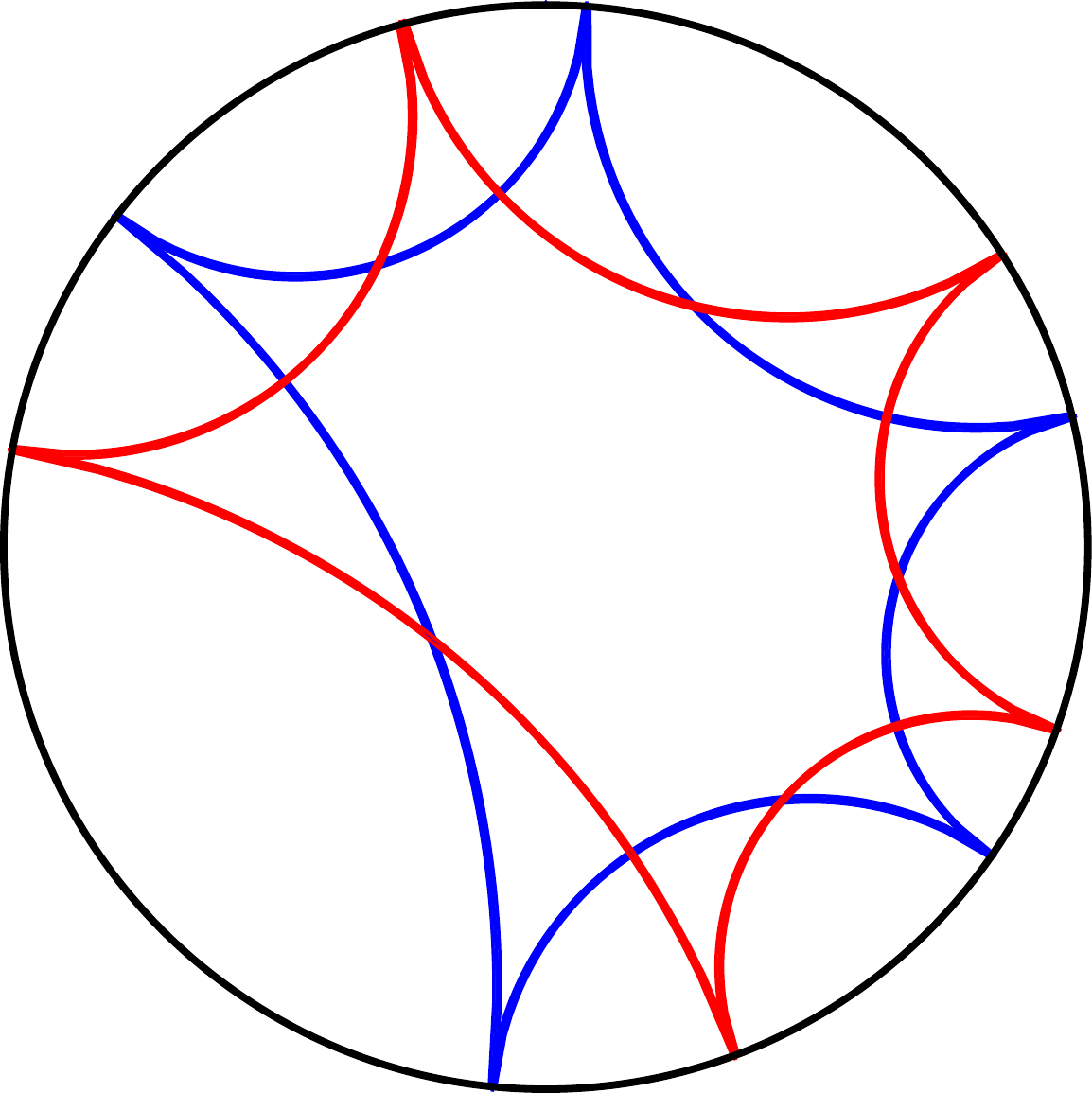}
\caption{Left: a convex pentagon and its evolute. Right: a pair of orthogonal ideal pentagons in the hyperbolic plane (in the Poincar\'e disk model).}
\label{polyevol}
\end{figure}

It is natural to investigate the dynamics of the evolute transformation on polygons in other geometries, in particular, in the hyperbolic plane. As a simplification, we consider the case of ideal polygons whose vertices lie on the circle at infinity. The notion of perpendicular bisector does not make sense anymore because the sides have infinite length, but one still can consider pairs of ideal polygons whose respective sides are perpendicular, see Figure \ref{polyevol} right. Call two such ideal $n$-gons orthogonal. 

In the projective model of the hyperbolic plane two lines are perpendicular if one passes through the pole of the other. Therefore two ideal $n$-gons are orthogonal if the extensions of the sides of one of them pass through the poles of the respective  sides of the other (and vice versa), that is, if each polygon circumscribes the polar of the other, see Figure \ref{fig:Castillon}. This provides a relation with the classical Cramer-Castillon problem \cite[\S 16.3.10.3]{Berger},\cite{Izm,Wanner}: inscribe a polygon in a circle whose sides pass through given points.

\begin{figure}[hbtp]
\centering
\includegraphics[height=2.2in]{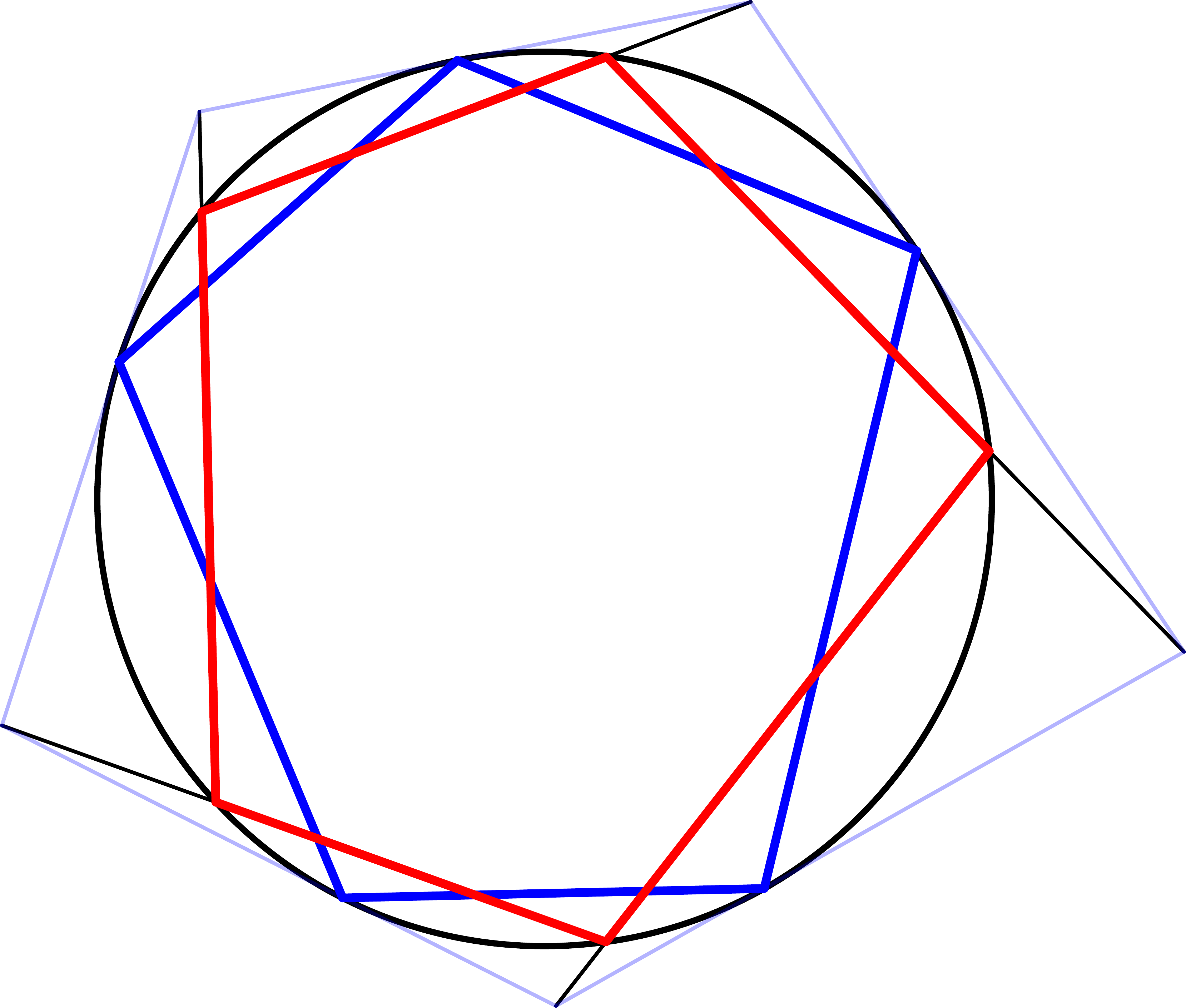}
\caption{Two ideal pentagons circumscribed about each other's polars.}
\label{fig:Castillon}
\end{figure}

The relation of being orthogonal generically is a 2-2 map on the space of ideal polygons. This paper concerns the geometry and dynamics of this map and its generalizations.

\subsection{Plan of the paper and main results} \label{resultsection}

An ideal $n$-gon $\pP$ is an ordered collection of points  $p_1,\ldots,p_n$ in the projective line $\PP^1$ over $\R$ or $\C$, the boundary of  the hyperbolic plane or hyperbolic space, respectively. In this paper, we use the following definition of cross-ratio:
$$
[z_1,z_2,z_3,z_4]=\frac{(z_1-z_3)(z_2-z_4)}{(z_1-z_4)(z_2-z_3)}.
$$
Define a family of relations $\pP \T \pQ$ by the formula
\begin{equation} \label{reldef}
[p_i,p_{i+1},q_i,q_{i+1}]=\alpha,\ i=1,\ldots,n;
\end{equation}
where the indices are understood cyclically, with $n+1 = 1$.\footnote{When dealing with closed polygons, we alsways understand the indices cyclically, unless explicitly specified otherwise.}
Two polygons are orthogonal if $\pP \T \pQ$ with $\alpha= -1$, that is, if $(p_i,p_{i+1},q_i,q_{i+1})$ is a harmonic quadruple for all $i$. Generically, $\T$ is a 2-2 relation; by ``generic" we mean a property that holds in a Zariski open subset. 

Along with closed polygons, we consider twisted ones. A twisted ideal $n$-gon is a bi-infinite sequence of points $\ldots, p_{-1}, p_0, p_1, p_2, \ldots$ in $\PP^1$, together with a M\"obius transformation $\Phi$, its {monodromy}, satisfying  $p_{i+n} = \Phi(p_i)$ for all $i$. For twisted $n$-gons $\pP$ and $\pQ$, we say that $\pP \T \pQ$ if, in addition to \eqref{reldef}, the polygons $\pP$ and $\pQ$ have the same monodromy.

The M\"obius group acts diagonally on the spaces of ideal $n$-gons and ideal twisted $n$-gons; the quotient spaces are the moduli spaces of ideal $n$-gons and ideal twisted $n$-gons, respectively. These moduli spaces have dimensions $n-3$ and $n$, respectively. 
As coordinates in the moduli spaces, we use $n$-periodic sequences of cross-ratios $c_i = [p_i, p_{i+1}, p_{i-1}, p_{i+2}]$.

The transformations and the relations that we study here are defined in open dense subsets of the spaces of ideal polygons (closed or twisted). That is, we assume that our polygons are non-degenerate in an appropriate sense (see Section \ref{sec:ClosedTwisted} for definitions). This is analogous to the  situation with  birational maps that are defined in Zariski open subsets. In particular, when considering infinite orbits of ideal polygons under  $\T$, we assume that all iterations are non-degenerate. 
\medskip

The contents and  main results of the paper are as follows. 

In Section \ref{spaces}, we associate with an ideal polygon a 1-parameter family of M\"obius transformations that we call its Lax matrices. A Lax matrix is the composition of the loxodromic transformations with a fixed parameter about the consecutive sides of the polygon, composed with the inverse of the monodromy in the twisted case. 

In Section \ref{monotw}, we determine the monodromy of a twisted polygon as a function of its coordinates $c_i$ and calculate its conjugacy invariant, the normalized trace (Theorem \ref{thm:trmono}). This material is closely related to the classical theory of continuants that goes back to Euler and, when $n$ is odd, with Coxeter's frieze patterns. 

The main result of Section \ref{sec:IntTwist} is that the relations $\T$ on the moduli space of twisted $n$-gons are completely integrable in the sense of Liouville   (Main Theorem \ref{thm:TwistedModuli}). This Main Theorem comprises a number of results (Theorems \ref{thm:LaxMatrix},\ref{thm:Integrals},\ref{thm:IntIndependence},\ref{thm:TraceComponents},\ref{thm:InvBracket},\ref{thm:FCommute}).
Its ingredients are as follows.

All the relations $\T$ share ${\lfloor \frac{n}{2} \rfloor}+1$ integrals, obtained from the Lax matrices of a twisted polygon; an equivalent set of integrals is provided by the homogeneous components of the normalized trace of the monodromy. These integrals are polynomials in $c_i$, generically independent. The moduli space of twisted $n$-gons carries a 1-parameter family of compatible Poisson structures, and the integrals are in involution with respect to all Poisson brackets in this family. 

For every $\alpha$, there is a Poisson structure in this pencil that is invariant under $\T$. If $n$ is odd, this structure has corank 1, and if $n$ is even, its corank equals 2. The space of the Hamiltonian vector fields of the integrals is independent on the choice of the Poisson bracket in the pencil.    

In addition, the relations $\T$ commute, in an appropriate sense (Theorem \ref{Bianchi}). We also give a criterion for a twisted $n$-gon $\pP$ to be exceptional in the sense that, for a given $\alpha$, there exist infinitely many $n$-gons $\pQ$ with $\pP \T \pQ$ (Theorem \ref{thm:InfAlpha}).

Section \ref{sec:IntClosed} concerns complete integrability of the relations $\T$ on the moduli space of closed $n$-gons.  
Its main result (Main Theorem \ref{thm:mainclosed}) is that a generic point of this moduli space belongs to a ${\lfloor \frac{n-3}{2} \rfloor}$-dimensional  submanifold, invariant under $\T$ and carrying an invariant affine structure, in which $\T$ is a parallel translation.   This foliation whose leaves are invariant manifolds and the affine structures on its leaves is the same for all values of $\alpha$.

An ingredient of this result is a description of the relations between the restrictions of the integrals to the moduli space of closed polygons (Theorem \ref{thm:RestrictIntegrals}) and an interpretation of the integrals in terms of multi-ratios (Theorem \ref{thm:IntegralsClosed}) and their geometric description in terms of alternating perimeters of ideal even-gons (Section \ref{sec:AltPer}). 

If $n$ is odd, the moduli space of closed $n$-gons carries a symplectic structure, known from the theory of  cluster algebras.  In Theorems \ref{thm:PoissonEmbedding} and \ref{Cluster}, we show that the inclusion of closed polygons to twisted ones,
endowed with a specific Poisson brackets from the pencil, is a Poisson map. For $n$ odd, we consider the limit of the relations $\T$ as $\alpha \to 0$:
this is a vector field on the moduli space of $n$-gons, Hamiltonian with respect to its symplectic structure  (Theorem \ref{thm:InfMap}).

    In Section \ref{twoaddsection} we consider the relations $\T$ on the space of (closed) ideal $n$-gons (before factorization by the M\"obius group). The space of $n$-gons carries a closed differential 2-form, invariant under $\T$, symplectic if $n$ is even and having a 1-dimensional kernel if $n$ is odd (Theorem \ref{2form}). This 2-form is M\"obius-invariant, but it does not descend to the moduli space of $n$-gons. 
    The relations $\T$ have two additional integrals whose Hamiltonian vector fields are infinitesimal generators of the M\"obius group action (Theorem \ref{upint}). The geometrical meaning of these additional integrals is that one can assign a line to an ideal polygon, that  we call its axis, and this line is invariant under the relations $\T$.
    
    We analyze the case of ``small-gons" in Section \ref{smallsection}. If two ideal quadrilaterals are in the relation $\T$, then they are isometric and they share their axes (Theorem \ref{thm:quad}). The moduli space of ideal pentagons is a surface with an area form, invariant under all the relations $\T$ and foliated by the level curves of their common integral (Theorem \ref{areapr}); the relations $\T$ on the 5-dimensional space of pentagons (before factorization by the M\"obius group) are integrable as well (Theorem \ref{transl}).
    
In Theorems \ref{thm:PentaInf} and \ref{Thm:exchexa}  we give a compete description of $\alpha$-exceptional ideal pentagons and hexagons, that admit infinitely many $\alpha$-related ideal pentagons and hexagons, respectively.   
    
    Finally, we study ``loxogons", the polygons that are in the relation $\T$ with themselves. More precisely, an $(n,k)$-loxogon is an ideal $n$-gon  such that  $[p_i,p_{i+1},p_{i+k},p_{i+k+1}]=\beta$ 	
for all $i=1,\ldots,n$ and some constant $\beta$.  A projectively regular ideal $n$-gon is an $(n,k)$-loxogon for all $k$; we address the question whether there exist non-regular loxogons.

    In Theorem \ref{triv} we answer this question for some pairs $(n,k)$: in some cases, one has rigidity (the only loxogons are projectively regular ones), and in other cases, one has examples of non-regular loxogons. The general question remains open.  Theorem \ref{infinites} is an infinitesimal rigidity result: for odd $n$ and any $k$,  there do not exist non-trivial deformations of a regular ideal $n$-gon in the class of $(n, k)$-loxogons.
    
The Appendix discusses certain collections of loxodromic transformations along the edges of an ideal tetrahedron,
which are related to our derivation of Lax matrices in Section \ref{spaces}.
Also, we connect this to the cubical diagram of 3D-consistency for the cross-ratio system from \cite{BS}, \cite[Section 6.6]{BSb}.

\subsection{Related work} \label{previous}

The cross-ratio dynamics on polygons in the projective line have been studied by many authors, and their integrability was established by different methods. We hope that our work sheds new light on these systems and adds to their understanding. Let us mention several relevant works; this list is by no means complete. 

In Bobenko and Pinkall's paper \cite{BP} maps $f \colon \Z^2 \to \C$ satisfying the property
\[
[f_{i,j}, f_{i+1,j}, f_{i,j+1}, f_{i+1,j+1}] = -1
\]
were called discrete holomorphic maps.
In these terms, the $\alpha = -1$ case of cross-ratio dynamics can be viewed as a periodic discrete holomorphic map $\Z^2/n\Z \to \C$.

The paper \cite{HMNP}  concerns algebraic-geometric integrability of periodic discrete holomorphic (conformal) maps; in particular, explicit solutions are given in terms of the Riemann theta function. 

Nijhoff and  Capel's paper \cite{NC} discusses the cross-ratio dynamics as space and time discretizations of the Schwarzian Korteweg-de Vries equation. 
 
The cross-ratio dynamical system is an example of an integrable system on quad-graphs \cite{BS}. In the ABS (Adler-Bobenko-Suris) classification of integrable, in the sense of consistency, systems on quad-graphs, this is the $Q1(0)$ case, see \cite{ABS}. We also refer to the book  \cite{BSb} where this subject is considered in the context of discrete differential geometry. 

Let us also mention relations with dressing chains of Veselov-Shabat \cite{SV}; see also the recent follow-up paper \cite{Evr}.

We would like to mention certain similarity of cross-ratio dynamical systems with the pentagram map, a discrete completely integrable system on the moduli space of projective equivalence classes of polygons in the projective plane that has recently attracted a considerable attention, see, e.g.,  \cite{Sch,OST1,OST2}. 

For example, the integrals of the pentagram map are obtained from homogeneous components of conjugacy invariants of the monodromy of a twisted polygon (which is a projective transformation of the plane). In this sense, our work is similar to the approach to the pentagram map developed in the above mentioned papers \cite{Sch,OST1,OST2}, whereas the algebraic-geometric approach to the cross-ratio dynamics in \cite{HMNP} is similar to the work of Soloviev \cite{Sol} on the pentagram map.

An important discovery, starting with the paper of Glick \cite{Gl}, was a close relation of the pentagram map with the theory of cluster algebras, see, e.g., the book \cite{{GSV}}. Is the cross-ratio system also related to the cluster dynamics?

\subsection{Acknowledgements} \label{ack}

We are grateful to V. Fock, M. Gekhtman, A. Izosimov, B. Khesin, S. Morier-Genoud, V. Ovsienko, M. Shapiro, and Yu. Suris for stimulating discussions.

Part of this material is based upon work supported by the National Science Foundation under grant DMS-1439786
while the authors were in residence at the Institute for Computational and Experimental Research in Mathematics in Providence, RI, during the Collaborate@ICERM  program in summer of 2017.
II was supported by the Swiss National Science Foundation grant 200021\_169391.
ST was  supported by NSF grant DMS-1510055.
Part of this material is based upon work supported by the National Science Foundation under Grant DMS-1440140
while ST was in residence at the Mathematical Sciences Research Institute in Berkeley, California, during the Fall 2018 semester. DF is grateful to MPIM Bonn for its invariable hospitality.

\section{Spaces and maps} \label{spaces}
\subsection{Cross-ratio and relative position of lines in hyperbolic geometry}

As was mentioned in Introduction, the circle at infinity of the hyperbolic plane $\HH^2$ is identified with the real projective line $\RP^1$. If $\pP$ is an ideal $n$-gon in $\HH^2$ with the vertices $p_1,\ldots,p_n$, we have $p_i \in \R \cup \infty$. 

We reiterate that we are concerned with the relations $\pP \T \pQ$ on ideal $n$-gons given by the equations
\begin{equation*}
[p_i,p_{i+1},q_i,q_{i+1}]=\alpha,\ i=1,\ldots,n.
\end{equation*}
Note that the relation $\pP \T \pQ$ is symmetric.

One can consider these equations over reals, that is, when $p_i,q_i \in \RP^1$, or over complex numbers, when $p_i,q_i \in \CP^1$;
the constant $\alpha$ is real or complex, accordingly.
The complex case corresponds to ideal polygons in the hyperbolic space $\HH^3$ where the sphere at infinity is the Riemann sphere $\CP^1$.
Most of our results hold in both cases and, when it does not lead to confusion,
we denote the projective line by $\PP^1$ and the projective linear group by $\PGL(2)$. 

Geometrically, two ideal polygons in $\HH^3$ satisfy $\pP \T \pQ$ if the {\it complex distance} between their respective sides is constant (depending on $\alpha$). 

\begin{figure}[hbtp]
\centering
\includegraphics[height=1.8in]{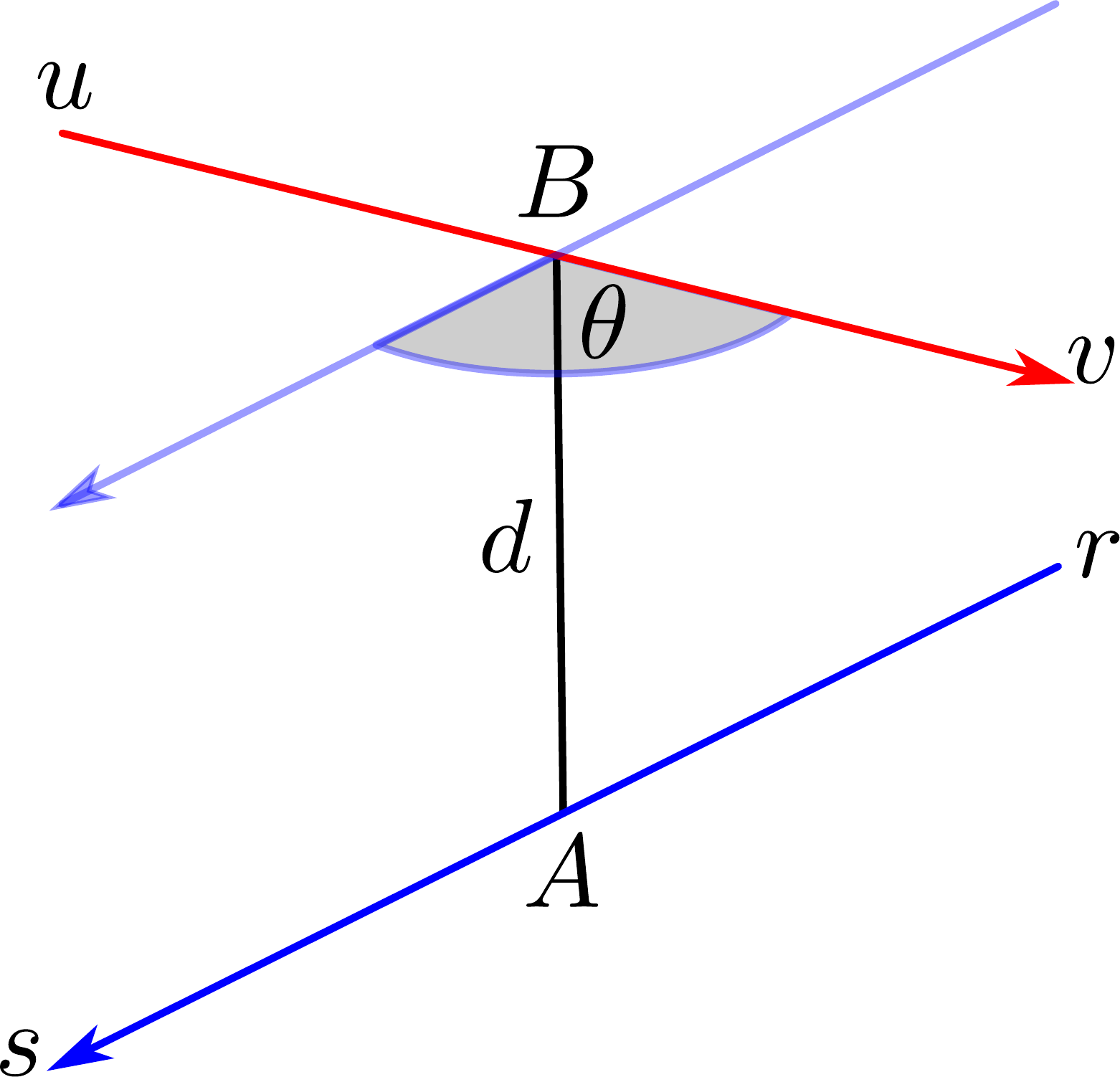}
\caption{Complex distance between lines.}
\label{skew}
\end{figure}

More precisely, let $\ell=uv$ and $m=rs$ be two oriented lines in $\HH^3$, with $u,v,r,s \in \CP^1$.
Let $A$ and $B$ be the intersection points of the lines with their common perpendicular.
Orient this common perpendicular, and let $d=\pm |AB|$ be the signed distance between the lines: the sign is positive if the segment $AB$ runs in the positive direction. Let  $\theta$ be the angle between the line though point $B$, coplanar with $rs$ and orthogonal to $AB$, and line $uv$, that is, the angle through which one must rotate the former line  to the latter one when looking in the direction of the common perpendicular; see Figure \ref{skew}. 
The complex distance is defined as follows: $\chi(\ell,m) = d + i \theta$, and its relation with cross-ratio is given by 
\begin{equation}
\label{eqn:DistCR}
\tanh^2 \left(\frac{\chi(\ell,m)}{2}\right) = [r,s,u,v],
\end{equation}
see  \cite{Ma}, pp. 354--355, for details.


\subsection{Closed and twisted $n$-gons}
\label{sec:ClosedTwisted}

We consider only {\it non-degenerate} polygons with distinct consecutive vertices and distinct consecutive sides, that is, we assume that $p_i\ne p_{i+1}$ and $p_i \neq p_{i+2}$ for all $i$. These conditions define a Zariski open subset of the set of  $n$-tuples of points on $\PP^1$.
The space of such non-degenerate ideal $n$-gons is denoted by $\widetilde\cP$, and we use $(p_1,\ldots,p_n)$ as coordinates therein.

The group $\PGL(2)$ acts on $\widetilde\cP$ diagonally; denote by $\cP$ the quotient space of this action.
We note that $\cP$ contains, as an open dense subset, the famous moduli space ${\mathcal M}_{0,n}$ of distinct $n$-tuples of points in the projective line modulo projective equivalence.
The relation $\T$ descends to $\cP$ and, slightly abusing notation, we continue to denote it by $\T$.
The dimensions of $\widetilde\cP$ and $\cP$ over the corresponding fields, $\R$ or $\C$, are equal to $n$ and $n-3$, respectively.

Consider the cross-ratios of quadruples of consecutive points:
\begin{equation}
\label{eqn:cCoord}
c_i = [p_i, p_{i+1}, p_{i-1}, p_{i+2}] = \frac{(p_{i-1} - p_i)(p_{i+1} - p_{i+2})}{(p_{i-1} - p_{i+1})(p_i - p_{i+2})}.
\end{equation}
Since by our assumption the points $p_{i-1}$, $p_i$, $p_{i+1}$ are distinct, the cross-ratio $c_i$ is well-defined.
Besides, since $p_{i+2}$ is distinct from $p_i$ and from $p_{i+1}$, we have $c_i \notin \{0, \infty\}$.
Due to the projective invariance of the cross-ratio, the functions $c_i$ descend to the space $\cP$.
It is easy to see that $c_1, \ldots, c_n$ determine a projective class of the polygon uniquely, providing 
an embedding of $\cP$ into $(\KK^*)^n$, where $\KK = \R$ or $\C$, according to the situation;
an explicit description of the image is given in Section \ref{calcmono}. 

\begin{remark}
If we allow consecutive sides to coincide then, in the situation $p_{i-1} = p_{i+1} \ne p_i = p_{i+2}$, the cross-ratio $c_i$ is not defined.
On the other hand, if $\pP \T \pQ$ and $p_{i-1} = p_{i+1}$ for some $i$, then
\[
[p_{i-1}, p_i, q_{i-1}, q_i] = \alpha = [p_i, p_{i+1}, q_i, q_{i+1}] \Rightarrow q_{i-1} = q_{i+1}.
\]
By removing the coincident pairs of sides, one obtains non-degenerate polygons $\pP' \T \pQ'$.
\end{remark}

We also consider a larger space of \emph{twisted $n$-gons} $\widetilde\cT$.
A twisted $n$-gon consists of a bi-infinite sequence of points $\ldots, p_{-1}, p_0, p_1, p_2, \ldots$ in $\PP^1$
and of a projective transformation $\Phi \in \PGL(2)$, called the {\it monodromy}, such that $p_{i+n} = \Phi(p_i)$ for all $i$. For $\pP, \pQ \in \widetilde\cT$, we write $\pP \T \pQ$ if, in addition to \eqref{reldef}, the polygons $\pP$ and $\pQ$ have the same monodromy. We consider only non-degenerate twisted polygons, subject to the constraints $p_i \neq p_{i+1}$ and $p_i \neq p_{i+2}$ for all $i$.

Twisted polygons are acted upon by $\PGL(2)$, and we denote the moduli space by $\cT$.
The following lemma is straighforward.
\begin{lemma}
\label{lem:MonodromyConj}
Let $\pP$ be a twisted polygon with monodromy $\Phi$.
Then for every $\Psi \in \PGL(2)$ the polygon $\Psi(\pP)$ has monodromy $\Psi \Phi \Psi^{-1}$.
\end{lemma}
It follows that the relation $\T$ descends to the space $\cT$.

The dimensions of $\widetilde\cT$ and $\cT$ are equal to $n+3$ and $n$, respectively.
We use the cross-ratios \eqref{eqn:cCoord} as coordinates in the space $\cT$; they identify  $\cT$ with an open dense subspace in $(\KK^*)^n$.

The cross-ratio coordinates $c_i$ are discrete analogs of projective curvature: they determine each next point $p_i$, given the preceding triple.  Thus an $n$-periodic sequence $c_i$ and an initial triple of vertices determine a twisted polygon uniquely. We shall study  the monodromy of a twisted polygon as a function of $c_i$ in detail in Section \ref{calcmono}. An analogous investigation of the monodromy of twisted polygons in the projective plane played a central role in the study of the pentagram map, see, e.g.,  \cite{Sch,OST1}.

\subsection{Loxodromic transformations and their matrices}
\label{sec:Loxodromics}
As we already mentioned, the complex projective line $\CP^1$ can be viewed as the sphere at infinity of the hyperbolic space;
the action of $\PGL(2, \C)$ on $\CP^1$ extends to the action on $\HH^3$ by orientation-preserving isometries.
The subgroup $\PGL(2, \R)$ fixes $\RP^1$; its elements correspond to isometries of $\HH^2 \subset \HH^3$.
Recall that $\PSL(2, \C) \cong \PGL(2, \C)$ but $\PSL(2, \R)$ is a subgroup of index $2$ in $\PGL(2, \R)$,
corresponding to the orientation-preserving isometries of $\HH^2$.

Hyperbolic isometries with two fixed points at the boundary at infinity are called \emph{loxodromic transformations}.
A loxodromic transformation in $\HH^3$ is similar to a Euclidean screw motion;
a loxodromic transformation in $\HH^2$ is similar to a translation or a glide reflection.
(In the hyperbolic plane loxodromic transformations are usually called isometries of hyperbolic type.)

Let $p, q \in \PP^1$ and $\lambda$ be a non-zero complex or real 
number.
A loxodromic transformation with the axis $pq$ and parameter $\lambda$ is a hyperbolic ``screw motion'' along $pq$
whose translational part in the direction from $p$ to $q$ is $\log |\lambda|$ and the rotation angle is $\arg \lambda$;
we denote it by $L_\lambda(p, q)$.
(In the real case, we have a translation along $pq$ if $\lambda > 0$ and a glide reflection if $\lambda < 0$.)
We have $L_\lambda(p,q) L_\mu(p,q) = L_{\lambda+\mu}(p,q)$, in particular, loxodromic transformations with the same axis commute.
Loxodromic transformations with different axes, but the same parameter, are conjugate. We refer, e.g., to \cite{Ma} for this material.

The next lemma makes it possible to express the relations $\T$, that is, equation (\ref{reldef}), in terms of loxodromic transformations.

\begin{lemma} \label{loxodr}
For every $z \in \PP^1$, one has
$$
[p, q, z, L_\lambda(p,q) (z)] = \frac{1}{\lambda}.
$$
\end{lemma}

\begin{proof}
A loxodromic transformation along any axis is conjugate to a loxodromic transformation along $0\infty$.
The cross-ratio is invariant under M\"obius transformations.
Therefore it suffices to prove the above identity for $p=0$ and $q=\infty$.
Due to $L_\lambda(0,\infty)(z) = \lambda z$, it boils down to
$$
[0,\infty,z,\lambda z] = [0,\infty,1,\lambda] = \frac{1}{\lambda},
$$
as claimed.
\end{proof}

The loxodromic transformation $L_\lambda(p,q)$ is represented by a linear transformation of the two-dimensional vector space
whose eigenvectors are representatives of $p$ and $q$, and whose eigenvalues are $1$ and $\lambda$, respectively.
In the standard affine chart on $\PP^1$, that associates to a number $p$ the vector $(p, 1)$ and to $\infty$ the vector $(1,0)$, the matrix of this linear transformation has the following form:
\begin{equation*}
\label{eqn:Acuv}
A_\lambda(p,q) = \frac{1}{p-q}
\begin{pmatrix}
p - q\lambda & pq(\lambda-1)\\
1 - \lambda & p\lambda - q
\end{pmatrix},
\end{equation*}
\[
A_\lambda(\infty, q) =
\begin{pmatrix}
1 & q(\lambda-1)\\ 0 & \lambda
\end{pmatrix},
\quad
A_\lambda(p, \infty) =
\begin{pmatrix}
\lambda & p(1-\lambda)\\
0 & 1
\end{pmatrix}.
\]
Note that $\det A_\lambda(p,q) = \lambda$.

The definition of $L_\lambda(p,q)$ takes into account the order of the points $p, q$.
Obviously, $L_\lambda(p,q) = L_{\lambda^{-1}}(q,p)$.
At the same time, the following lemma (verified by a direct calculation) shows that the corresponding matrices are different for $\lambda \ne \pm 1$.
\begin{lemma} \label{loxoprop}
One has
$$
A_{\lambda^{-1}} (p,q) = A_\lambda^{-1} (p,q) = \lambda^{-1} A_\lambda (q,p).
$$
\end{lemma}
Thus there is no continuous lift of the space of all loxodromic transformations to $\GL(2)$;
the domain of the map $L_\lambda(p,q) \mapsto A_\lambda(p,q)$ is the set of \emph{loxodromic transformations with oriented axes}.

\subsection{Lax matrix of an ideal  polygon}
\label{Laxmat}

Let $\pP$ be a twisted polygon with  monodromy $\Phi$.
Consider the  projective transformation of $\PP^1$ depending on a parameter $\lambda \ne 0$:
\[
L_\lambda(\pP) = \Phi^{-1} L_\lambda(p_n, p_{n+1}) \cdots L_\lambda(p_1, p_2),
\]
the composition of loxodromic transformations with parameter $\lambda$ along the edges, followed by the inverse of the monodromy.

\begin{definition}
A \emph{Lax matrix} $A_\lambda(\pP)$ of a twisted polygon $\pP$ is an element of $\GL(2)$ representing the above projective transformation.
In particular, we may write
\begin{equation}
\label{eqn:LaxMatrix}
A_\lambda(\pP) = M^{-1} A_\lambda(p_n, p_{n+1}) \cdots A_\lambda(p_1, p_2),
\end{equation}
where $M \in \GL(2)$ is a representative of the monodromy $\Phi$ of $\pP$.
\end{definition}
The Lax matrix of a twisted polygon is well-defined up to scaling. We shall justify the choice of terminology (Lax matrix) in Section \ref{sec:IntTwist}. 

The relation $\T$ is naturally expressed in terms of Lax matrices.

\begin{lemma}
\label{lem:FixedPoint}
If $\pP \T \pQ$, then $q_1$ is a fixed point of $L_{\alpha^{-1}}(\pP)$.
Conversely, if $q_1$ is a fixed point of $L_{\alpha^{-1}}(\pP)$ and the points
\begin{equation}
\label{eqn:FixedPoint}
q_{i+1} = L_{\alpha^{-1}}(p_i, p_{i+1})(q_i)
\end{equation}
form a non-degenerate twisted polygon $\pQ$, then $\pQ \T \pP$.
\end{lemma}

\begin{proof}
If $\pP \T \pQ$ then, by Lemma \ref{loxodr}, equation \eqref{eqn:FixedPoint} holds.
It follows that $L_{\alpha^{-1}}(p_n, p_{n+1}) \circ \cdots \circ L_{\alpha^{-1}}(p_1, p_2)(q_1) = q_{n+1}$.
By definition of  monodromy, we have $\Phi^{-1}(q_{n+1}) = q_1$, which implies the first part of the lemma.
For the second part, taking a fixed point $q_1$ and defining the other vertices of $\pQ$ by \eqref{eqn:FixedPoint} one obtains a polygon with monodromy $\Phi$.
\end{proof}

\begin{corollary}
\label{cor:2-2}
For every twisted polygon $\pP$ and every $\alpha \ne \{\infty, 0, 1\}$, the number of twisted polygons $\pQ$ such that $\pP \T \pQ$ is $1$, $2$, or $\infty$
in the complex case, and $0$, $1$, $2$, or $\infty$ in the real case.
\end{corollary}
\begin{proof}
These are all possible numbers of fixed points of a projective transformation of $\PP^1$.
\end{proof}

\begin{example}
There exists a closed ideal octagon whose Lax matrix with the parameter $-1$ is proportional to the identity, see Figure \ref{fig:Ortogon}.
This  octagon admits infinitely many orthogonal ideal octagons. See more in Section \ref{subsect:except}. 

\begin{figure}[htb]
\begin{center}
\includegraphics[width=.5\textwidth]{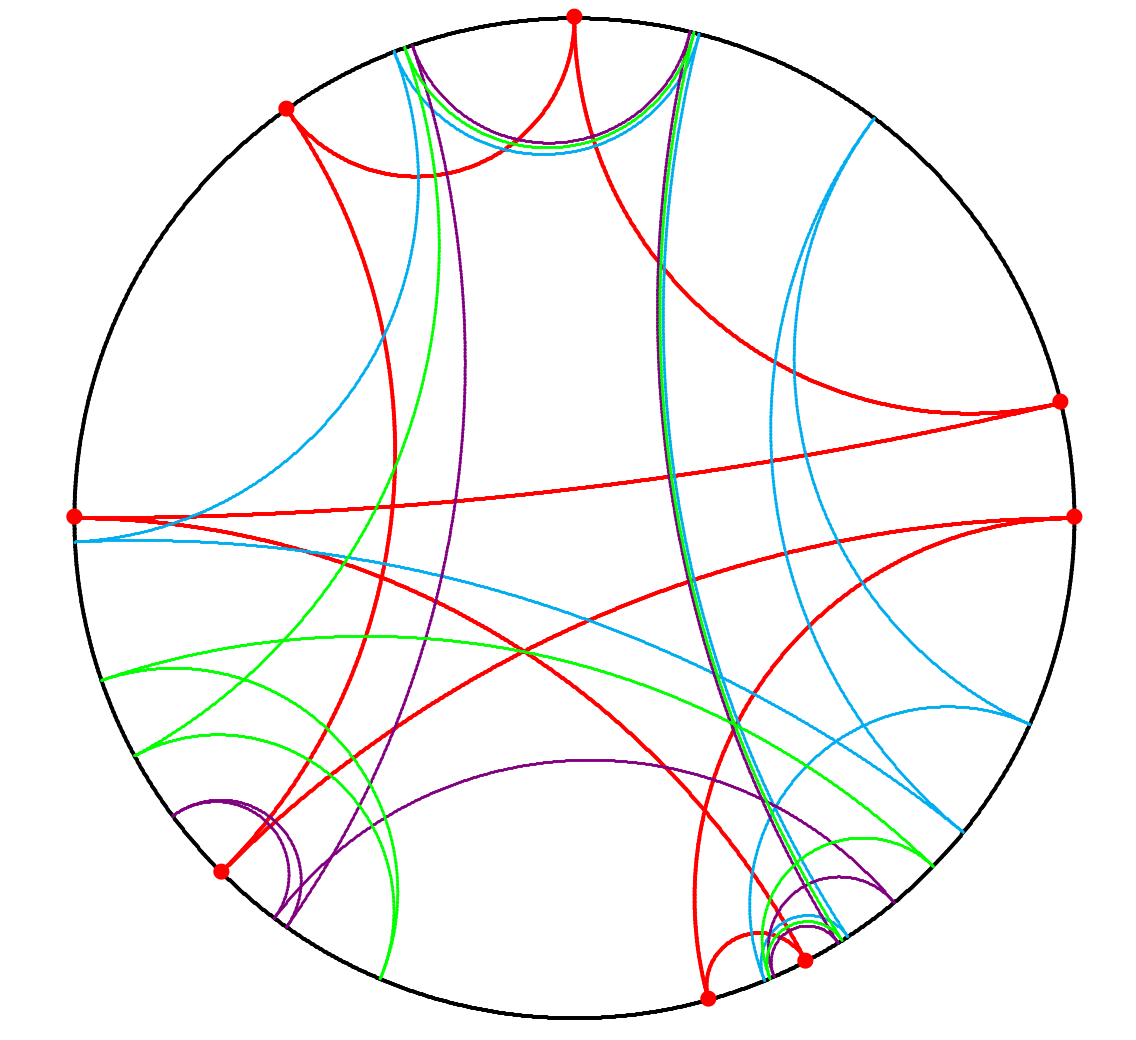}
\end{center}
\caption{The composition of the reflections in the consecutive sides of this ideal octagon is 
the identity. Three orthogonal ideal octagons are shown.}
\label{fig:Ortogon}
\end{figure}
\end{example}

If we work over $\C$, that is, with ideal polygons in $\HH^3$, then a generic  M\"obius transformation  has  two distinct fixed points. Therefore, for a generic $\pP$ and $\alpha$, one has two choices of $\pQ$ with $\pQ \T \pP$; generically, both of these choices are non-degenerate polygons. Choosing one of them makes the relation $\T$ into a map that we denote by $T_\alpha$; the other choice corresponds to the inverse map $T^{-1}_\alpha$. 

If we work over $\R$, that is, with ideal polygons in $\HH^2$, we need to assume that $L_{\alpha^{-1}}(\pP)$ has two fixed points. Then $\pQ$ exists and is in general non-degenerate. The map $L_{\alpha^{-1}}(\pQ)$ automatically has a fixed point (corresponding to $\pP$). Choosing a different fixed point of this M\"obius transformation makes it possible to continue, defining the map  $T_\alpha$.

The following are some basic properties of the Lax matrix needed in the sequel.

\begin{lemma}
The $\PGL$-action on the space of twisted polygons conjugates the Lax matrix:
for any $\Psi \in \PGL(2)$, we have
\[
L_\lambda(\Psi(\pP)) = \Psi L_\lambda(\pP) \Psi^{-1}.
\]
Accordingly, if $N \in \GL(2)$ is any representative of $\Psi$, then the matrices $A_\lambda(\Psi(\pP))$ and $NA_\lambda(\pP) N^{-1}$ are equal up to scaling.
\end{lemma}
\begin{proof}
We have $L_\lambda(\Psi(p_i), \Psi(p_{i+1})) = \Psi L_\lambda(p_i, p_{i+1}) \Psi^{-1}$.
Taking into account Lemma \ref{lem:MonodromyConj}, we obtain
\begin{multline*}
L_\lambda(\Psi(\pP)) = (\Psi\Phi\Psi^{-1}) (\Psi L_\lambda(p_n, p_{n+1}) \Psi^{-1}) \cdots (\Psi L_\lambda(p_1, p_2) \Psi^{-1})\\
= \Psi \Phi L_\lambda(p_n, p_{n+1}) \cdots L_\lambda(p_1, p_2) \Psi^{-1} = \Psi L_\lambda(\pP) \Psi^{-1}.
\end{multline*}
\end{proof}

From a closed or twisted polygon $\pP$, a new polygon $\pP^{+1}$ is obtained by the index shift: $p^{+1}_i = p_{i+1}$.

\begin{lemma} \label{cyclicshift}
The index shift conjugates the Lax matrix.
\end{lemma}
\begin{proof}
We have
\begin{multline*}
L_\lambda(\pP^{+1}) = \Phi^{-1} L_\lambda(p_{n+1}, p_{n+2}) \cdots L_\lambda(p_2, p_3)\\
= (\Phi^{-1} L_\lambda(p_{n+1}, p_{n+2}) \Phi) \Phi^{-1} L_\lambda(p_{n-1}, p_n \cdots L_\lambda(p_2, p_3)\\
= L_\lambda(p_1, p_2) \Phi^{-1} L_\lambda(p_{n-1}, p_n) \cdots L_\lambda(p_2, p_3)\\
= L_\lambda(p_1, p_2) L_\lambda(\pP) L_\lambda^{-1}(p_1, p_2),
\end{multline*}
as needed.
\end{proof}

\section{Monodromy of a twisted polygon} 
\label{monotw} 

\subsection{Continuants}
\label{sec:CoCont}

Following T. Muir \cite{Muir}, we call the three-diagonal determinants 
\[
\begin{vmatrix}
a_1 & b_1 & 0 & \ldots & 0 & 0\\
c_1 & a_2 & b_2 & \ldots & 0 & 0\\
0 & c_2 & a_3 & \ldots & 0 & 0\\
\vdots & \vdots & \vdots & \ddots & \vdots & \vdots\\
0 & 0 & 0 & \ldots &a_{n-1} & b_{n-1}\\
0 & 0 & 0 & \ldots &c_{n-1} & a_n
\end{vmatrix}
\]
{\it continuants} (sometimes this term is reserved for the particular case when $b_i=c_i=1$ for all $i$). Chapter 13 of the classic Muir book is devoted to the properties of these determinants. See also \cite{ConOvs17} for the curious history and applications of continuants.

We shall use the following rule for calculating continuants, that goes back to Euler (item 545 in \cite{Muir}):
{\it one term of the continuant is $a_1a_2\ldots a_n$, and the other terms are obtained from it by replacing any number of disjoint pairs $(a_i a_{i+1})$ by $(-b_ic_i)$.}

Given an $n$-tuple $c_1,\ldots,c_n$, consider the continuants
\[
D_{i,j} =
\begin{vmatrix}
1 & \sqrt{c_i} & 0 & \ldots & 0 & 0\\
\sqrt{c_i} & 1 & \sqrt{c_{i+1}} & \ldots & 0 & 0\\
0 & \sqrt{c_{i+1}} & 1 & \ldots & 0 & 0\\
\vdots & \vdots & \vdots & \ddots & \vdots & \vdots\\
0 & 0 & 0 & \ldots & 1 & \sqrt{c_j}\\
0 & 0 & 0 & \ldots & \sqrt{c_j} & 1
\end{vmatrix},
\]
where $1\le i\le j\le n$.
For every non-zero $c_k$, one has two choices for the value of $\sqrt{c_k}$;  Euler's rule implies that the result is independent of these choices and is a polynomial in the variables $c_k$. When we need to emphasize the dependence of $c_1,\ldots,c_n$, we also write $D_{i,j} (c)$.

Expanding $D_{i,j} $ in the last or in the first row yields the recurrences
\begin{equation}
\label{eqn:DRecurrence}
D_{i,j} = D_{i,j-1} - c_j D_{i,j-2}, \quad D_{i,j} = D_{i+1,j} - c_i D_{i+2,j}.
\end{equation}
As the initial values, one may take $D_{i,i-2} = D_{i,i-1} = 1$.

We also need a scaled version of the above continuant:
\[
D_{i,j}(\lambda) =
\begin{vmatrix}
1 & \sqrt{c_i\lambda} & 0 & \ldots & 0 & 0\\
\sqrt{c_i\lambda} & 1 & \sqrt{c_{i+1}\lambda} & \ldots & 0 & 0\\
0 & \sqrt{c_{i+1}\lambda} & 1 & \ldots & 0 & 0\\
\vdots & \vdots & \vdots & \ddots & \vdots & \vdots\\
0 & 0 & 0 & \ldots & 1 & \sqrt{c_j\lambda}\\
0 & 0 & 0 & \ldots & \sqrt{c_j\lambda} & 1
\end{vmatrix}.
\]
In particular, $D_{i,j} = D_{i,j}(1)$ (and  later we will also write this continuant as $D_{i,j}(\lambda c)$ to indicate the dependence on the variables $c_i$). 

Introduce the following notation: for $I \subset \{1,\ldots,n\}$, denote by $c_I$ the product of $c_i$ for all $i\in I$. A set $I\subset [i,j]$ is called {\it sparse} if it contains no pairs of consecutive indices, and it is called {\it cyclically sparse} if it is sparse when the indices are understood cyclically mod $n$, that is, $1$ follows after $n$.

\begin{lemma}
\label{lem:CoCont}
One has
\[
D_{i,j}(\lambda) = \sum_{I \subset [i,j] \text{ sparse}} (-1)^{|I|} \lambda^{|I|} c_I.
\]
\end{lemma}

\begin{proof}
This is a direct consequence of Euler's rule.
\end{proof}

\subsection{Calculating the monodromy}
\label{calcmono}

\begin{lemma}
\label{lem:MonMat2}
Let $\pP = (p_i)_{i \in \Z}$ be a twisted $n$-gon with the cross-ratios $c_i = [p_i, p_{i+1}, p_{i-1}, p_{i+2}]$.
Then the monodromy matrix of $\pP$ in the projective basis $p_0 = 1$, $p_1 = \infty$, $p_2 = 0$ is given by the product
\begin{equation}
\label{eqn:MonMat2}
M_n = \begin{pmatrix} 0 & c_1\\ -1 & 1 \end{pmatrix} \cdots \begin{pmatrix} 0 & c_n\\ -1 & 1 \end{pmatrix}.
\end{equation}
\end{lemma}
\begin{proof}
Induction on $n$. For $n=1$ the matrix $M_1$ acts as follows:
\[
\begin{pmatrix} 0 & c_1\\ -1 & 1 \end{pmatrix} :
\begin{cases}
p_0 &\sim \begin{pmatrix} 1\\ 1 \end{pmatrix} \mapsto \begin{pmatrix} c_1\\ 0 \end{pmatrix} \sim p_1\\
p_1 &\sim \begin{pmatrix} 1\\ 0 \end{pmatrix} \mapsto \begin{pmatrix} 0\\ -1 \end{pmatrix} \sim p_2\\
p_2 &\sim \begin{pmatrix} 0\\ 1 \end{pmatrix} \mapsto \begin{pmatrix}  c_1\\ 1 \end{pmatrix} \sim p_3
\end{cases}
\]
(the last line follows from $c_1 = [\infty, 0, 1, p_3] = p_3$).
For the induction step assume that $M_n$ acts by
\[
1 \mapsto p_n, \quad \infty \mapsto p_{n+1}, \quad 0 \mapsto p_{n+2}.
\]
By the invariance of the cross-ratio, one has
\[
c_{n+1} = [\infty, 0, 1, c_{n+1}] = [p_{n+1}, p_{n+2}, p_n, M_n(c_{n+1})].
\]
Hence, by definition of $c_{n+1}$, one has $M_n(c_{n+1}) = p_{n+3}$.

It follows that
\[
M_{n+1} = M_n \begin{pmatrix} 0 & c_{n+1}\\ -1 & 1 \end{pmatrix} :
\begin{cases}
1 &\mapsto \infty \mapsto p_{n+1}\\
\infty &\mapsto 0 \mapsto p_{n+2}\\
0 &\mapsto c_{n+1} \mapsto p_{n+3}.
\end{cases}
\]
\end{proof}

\begin{lemma}
\label{lem:MonMat1}
In the projective basis $p_0 = 1, p_1 = \infty, p_2 = 0$, the monodromy of a twisted polygon is represented by the matrix
\[
M_n=
\begin{pmatrix}
-c_1 D_{3,n-1} & c_1 D_{3,n}\\
-D_{2,n-1} & D_{2,n}
\end{pmatrix}.
\]
\end{lemma}
\begin{proof}
Arguing by induction on $n$, we check the equation for $n=1$, and then compute, using equation (\ref{eqn:DRecurrence}),
\begin{multline*}
\begin{pmatrix}
-c_1 D_{3,n-1} & c_1 D_{3,n}\\
-D_{2,n-1} & D_{2,n}
\end{pmatrix}
\begin{pmatrix}
0 & c_{n+1}\\ -1 & 1
\end{pmatrix}\\
=
\begin{pmatrix}
-c_1 D_{3,n} & c_1(-c_{n+1} D_{3,n-1} + D_{3,n})\\
-D_{2,n} & -c_{n+1} D_{2,n-1} + D_{2,n}
\end{pmatrix}\\
=
\begin{pmatrix}
-c_1 D_{3,n} & c_1 D_{3,n+1}\\
-D_{2,n} & D_{2,n+1}
\end{pmatrix}.
\end{multline*}
\end{proof}

As we mentioned in Section \ref{sec:ClosedTwisted}, the dimension of the 
moduli space of closed polygons is $3$ less that the dimension of the 
moduli space of twisted polygons, and hence the $n$ cross-ratios of 
quadruples of consecutive vertices of a closed ideal $n$-gon must satisfy 
three independent relations.
We can now derive these relations by equating the monodromy to the identity.

Consider the numbers $c_1, \ldots, c_n$ defined in \eqref{eqn:cCoord} as an $n$-periodic sequence.

\begin{lemma}
\label{lem:DRelations}
For every closed $n$-gon, the numbers $c_1, \ldots, c_n$  satisfy the relations $D_{i, n-3+i} = 0$ for all $i$.
Conversely, every set of numbers $(c_i)$ satisfying these equations corresponds to a closed $n$-gon.
Besides, any three consecutive equations imply the rest.
\end{lemma}
\begin{proof}
For a closed $n$-gon one has $M_n \sim \Id$ which, by Lemma \ref{lem:MonMat1}, implies $D_{3,n} = 0$.
A cyclic permutation of the factors on the right hand side of \eqref{eqn:MonMat2} conjugates the matrix $M_n$ (Lemma \ref{cyclicshift}).
Hence $M_n \sim \Id$ implies $D_{i, n-3+i} = 0$ for all $i$.

In the opposite direction, if $D_{i, n-3+i} = 0$ for all $i$, then the off-diagonal entries of the matrix in Lemma \ref{lem:MonMat1} vanish.
The diagonal entries are equal because of
\begin{multline*}
D_{2,n} + c_1 D_{3,n-1} = (D_{1,n} + c_1 D_{3,n}) + (D_{2,n-1} - D_{1,n-1})\\
= D_{1,n} - D_{1,n-1} = -c_1 D_{1,n-2} = 0.
\end{multline*}

Equation $D_{2,n-1} = 0$ is equivalent to the condition that the projective transformation defined by $M_n$ fixes the point $\infty \in \PP^1$,
that is, the monodromy $\Phi$ of the polygon $\pP$ fixes the point $p_1$.
By a cyclic permutation it follows that $D_{i, n-3+i} = 0$ is equivalent to $\Phi(p_{i-1}) = p_{i-1}$.
Since any three consecutive vertices of $\pP$ are pairwise distinct, three consecutive equations of the form $D_{i, n-3+i} = 0$ imply $\Phi = \Id$.
\end{proof}

\begin{remark}
The dependence between four consecutive equations can be made explicit by using  recurrences \eqref{eqn:DRecurrence}:
\[
c_{i-1}D_{i, n-3+i} + (c_i - 1)D_{i+1, n-2+i} - (c_i - 1)D_{i+2, n-1+i} - c_{i+1}D_{i+3, n+i} = 0.
\]
\end{remark}


The monodromy of the projective equivalence class of a twisted polygon is defined up to conjugation (Lemma \ref{lem:MonodromyConj}).
A conjugacy invariant of a projective transformation, represented by a matrix $M$, is its normalized trace, ${\Tr^2 M}/{\det M}$.

\begin{theorem}
\label{thm:trmono}
The normalized trace of the monodromy of a twisted $n$-gon is given by
\begin{equation}
\label{eqn:TraceC}
\frac{\Tr^2 M_n}{\det M_n} = \frac{1}{c_{[n]}} \left( \sum_{k=0}^{\lfloor \frac{n}{2} \rfloor} (-1)^k F_k(c) \right)^2,
\end{equation}
where
\begin{equation}
\label{eqn:Integrals1}
F_k(c) = \sum_{|I|=k} c_I, \ I \subset [n] \text{ cyclically sparse and, by definition,}\ F_0=1.
\end{equation}
\end{theorem}

\begin{proof}
From Lemmas \ref{lem:MonMat1} and \ref{lem:CoCont} it follows that
\begin{equation}
\label{eqn:TraceMon}
\Tr M_n = D_{2,n} - c_1 D_{3,n-1} = \sum_{I \circlearrowleft\text{sparse}} (-1)^{|I|} c_I.
\end{equation}
Indeed, there are two types of cyclically sparse subsets of $[n]$: those which contain $1$ and those which do not.
The latter are accounted for in the term $D_{2,n}$, the former in the term $-c_1 D_{3,n-1}$.

Lemma \ref{lem:MonMat2} implies that $\det M_n = c_{[n]}$, and the theorem follows.
\end{proof}

\begin{corollary}
\label{cor:MonoPara}
The monodromy of a twisted $n$-gon is parabolic or the identity if and only if the cross-ratios satisfy the identity
\begin{equation}
\label{eqn:MonoPara}
\left( \sum_{k=0}^{\lfloor\frac{n}{2}\rfloor} (-1)^k F_k(c) \right)^2 = 4 c_{[n]}.
\end{equation}
\end{corollary}
\begin{proof}
A projective transformation $\Phi \in \PGL(2)$ is of parabolic type or the identity if and only if $\frac{\Tr^2 M}{\det M} = 4$
for a representative of $\Phi$.
The rest follows from Theorem \ref{thm:trmono}.
\end{proof}

\begin{corollary}
\label{cor:SchoeneGleichung}
The cross-ratios of a closed polygon satisfy the relation \eqref{eqn:MonoPara}.
\end{corollary}

The following two lemmas will be useful later.

\begin{lemma}
\label{lem:MoreDRel}
The cross-ratios of a closed polygon satisfy the following relations:
\[
D_{i,n-1+i} = D_{i,n-2+i} = \frac12 \sum_{I \circlearrowleft \text{sparse}} (-1)^{|I|} c_I
\]
for all $i$.
\end{lemma}
\begin{proof}
Due to equations \eqref{eqn:DRecurrence} and Lemma \ref{lem:DRelations} one has
\begin{gather*}
D_{i,i+n-1} = D_{i,i+n-2} - c_{i+n-1}D_{i,i+n-3} = D_{i,i+n-2},\\
D_{i-1,i+n-2} = D_{i,i+n-2} - c_{i-1}D_{i+1,i+n-2} = D_{i,i+n-2}.
\end{gather*}
It follows that $D_{i,n-1+i} = D_{i,n-2+i}$ and their common value is independent of~$i$.

At the same time one has
\begin{gather*}
\sum_{I \circlearrowleft \text{sparse}} (-1)^{|I|} c_I = D_{i,i+n-2} - c_{i-1}D_{i+1,i+n-3},\\
D_{i-1,i+n-3} = D_{i,i+n-3} - c_{i-1}D_{i+1,i+n-3} = -c_{i-1}D_{i+1,i+n-3},
\end{gather*}
which implies
\[
\sum_{I \circlearrowleft \text{sparse}} (-1)^{|I|} c_I = D_{i,i+n-2} + D_{i-1,i+n-3} = 2D_{i,i+n-2},
\]
and the lemma is proved.
\end{proof}

\begin{lemma}
\label{lem:LinRelClosed}
The cross-ratios of a closed polygon satisfy the equation
\[
\sum_{k=0}^{\lfloor \frac{n}{2} \rfloor} (-1)^k (n-2k) \sum_{|I|=k} c_I =0, \quad I \text{ cyclically sparse}.
\]
\end{lemma}
\begin{proof}
This equation is  the sum of the equations $D_{i, n-3+i} = 0$.

By Lemma \ref{lem:CoCont}, the polynomial $D_{i, n-3+i}$ is the sum of monomials $(-1)^{|I|}c_I$ over all sparse $I \subset [i, n-3+i]$. It suffices to establish that, for every $I$ with $|I|=k$, the number of indices $i$, such that $I \subset [i, n-3+i]$, equals $n-2k$.
This is true because there are exactly $2k$ pairs of consecutive indices $(i-2, i-1)$ having a non-empty intersection with $I$.
\end{proof}


\begin{example}
For $n=5$, Corollary \ref{cor:MonoPara} says that the relations $c_{i-1} + c_i + c_{i+1} = 1 + c_{i-1}c_{i+1}$ (indices taken modulo $5$) imply the identity
\[
\left( 1 - \sum_{i=1}^5 c_i + \sum_{i=1}^5 c_{i-1}c_{i+1} \right)^2 = 4\prod_{i=1}^5 c_i.
\]
Since the sum of the initial relations yields $3\sum c_i = 5 + \sum c_{i-1}c_{i+1}$, they also imply the identity
\[
(c_1 + c_2 + c_3 + c_4 + c_5 - 2)^2 = c_1c_2c_3c_4c_5.
\]
In the context of \emph{pentagramma mirificum} (and in the variables $\gamma_i = c_i - 1$) it was proved by Gauss, see \cite{Cox71}.
\end{example}

\subsection{Frieze patterns}
\label{fripat}

The moduli space of projective equivalence classes of  $n$-gons in the projective line is intimately related to Coxeter frieze patterns, the subject of a considerable current interest; see \cite{Cox71} for the original paper and \cite{Mor} for a modern comprehensive survey. 

In this section, we relate the previous material with  results from the theory of friezes (that was one of our motivations  in the first place).  We consider the case of closed $n$-gons with $n$ odd.

Let $V_i$ be a doubly infinite sequence of vectors in $\KK^2$ (as before, $\KK$ is either $\R$ or $\C$).
Organize their pairwise determinants in a table:

\begin{center}
\begin{tabular}{ccccccc}
$\ldots$ && $[V_0, V_1]$ && $[V_1, V_2]$ && $\ldots$\\
& $[V_{-1}, V_1]$ && $[V_0, V_2]$ && $[V_1, V_3]$ & \\ 
$\ldots$ && $[V_{-1}, V_2]$ && $[V_0, V_3]$ && $\ldots$\\
& $[V_{-2}, V_2]$ && $[V_{-1}, V_3]$ && $[V_0, V_4]$ &\\
$\ldots$ && $\ldots$ && $\ldots$ && $\ldots$
\end{tabular}
\end{center}

This table can be extended upwards by a row of zeros $[V_i, V_i] = 0$ and further by $[V_j, V_i]$ with $j > i$.
The reflection in the line $i=j$ inverts the signs.

If we assume the sequence to be antiperiodic: $V_{i+n} = - V_i$, then each diagonal of the table is antiperiodic as well,
the row $j = i+n$ consists of zeros, and the subsequent rows repeat the strip $i < j < i+n$ with a sign change.
Combining the translation $(i,j) \mapsto (i, j+n)$ with the reflection in the line $i=j$, we obtain an additional symmetry:
\[
[V_i, V_j] = -[V_j, V_i] = [V_j, V_{i+n}].
\]
This is a glide reflection with respect to the line $i+j = n$ by distance $n/2$.
As a fundamental domain of its action on the strip, one can choose the triangle $i \ge 0, j \le n$.

Let $p_i$ be an $n$-periodic sequence of points on the projective line $\PP^1$ with odd $n$.
Then one can lift points $p_i$ to vectors $V_i$, normalized in such a way that $[V_i, V_{i+1}] = 1$ for all $i$ and the sequence $V_i$ is antiperiodic.

\begin{remark} \label{liftev}
{\rm
If $n$ is even, a closed projective $n$-gon may not have such a lift. The subspace of $n$-gons that admit a lift satisfying the normalization condition $[V_i, V_{i+1}] = 1$ has codimension one (these polygons have zero alternating perimeter, see Section \ref{sec:AltPer}), and polygons in this subspace admit 1-parameter families of lifts; see \cite{MOST} for details. 
}
\end{remark}

The Ptolemy-Pl\"ucker relation
\[
[V_{i-1}, V_j] [V_i, V_{j+1}] - [V_i, V_j] [V_{i-1}, V_{j+1}] = [V_{i-1}, V_i] [V_j, V_{j+1}]
\]
implies that each elementary diamond
\[
 \begin{matrix}
 &N&
 \\
W&&E
 \\
&S&
\end{matrix}
\]
satisfies  the unimodular relation $EW - NS = 1$.

This unimodular relation, along with the boundary conditions (row of zeros, followed by a row of ones) is the definition of frieze patterns. The name was coined by H.S.M. Coxeter because of the glide reflection  symmetry. 

\begin{example}
\label{ex:fr5}
A general frieze pattern of width 2 looks like this:
$$
 \begin{array}{ccccccccccc}
\cdots&&1&& 1&&1&&\cdots
 \\[4pt]
&x_1&&\frac{x_2+1}{x_1}&&\frac{x_1+1}{x_2}&&x_2&&
 \\[4pt]
\cdots&&x_2&&\frac{x_1+x_2+1}{x_1x_2}&&x_1&&\cdots
 \\[4pt]
&1&&1&&1&&1&&
\end{array}
$$
(the rows of 0s are omitted); it is related to Gauss's ``Pentagramma Mirificum", his  study of self-dual spherical pentagons, published posthumously, see \cite{Ga}.
\end{example}

The unimodular relation makes it possible to reconstruct the table from its two lines: the line of $[V_i, V_{i+1}] = 1$ and the next line of $[V_{i-1}, V_{i+1}] = a_i$.
These two equations are equivalent to the second-order linear recursive relation
\begin{equation}
\label{dHill}
V_{i+1} = a_i V_{i} - V_{i-1}.
\end{equation}
By the above arguments, an $n$-periodic sequence $a_i$ defines an antiperiodic sequence $V_i$ if and only if
the $(n-1)$-st line of the table consists of ones.

Equation (\ref{dHill}) is a discrete analog of Hill's differential equation $V''(t)=p(t) V(t)$, with the sequence $a_i$ playing the role of the potential $p(t)$. Hence the theory of friezes is a discretization of the theory of Hill's equation whose solutions are anti-periodic, see \cite{OT2}.

The frieze contains information about the cross-ratio coordinates $c_i$ of the respective closed ideal polygon $(p_i)$.

\begin{lemma}
\label{atoc}
One has
$$
c_i = \frac{1}{a_ia_{i+1}}\ \ {\rm and}\ \ a_i = \frac{\sqrt{c_{[n]}}}{c_ic_{i+2} \cdots c_{i-1}} = \frac{c_{i+1} c_{i+3} \cdots c_{i-2}}{\sqrt{c_{[n]}}}.
$$
\end{lemma}

\begin{proof}
One has
\[
c_i = [p_i, p_{i+1}, p_{i-1}, p_{i+2}] = \frac{[V_i, V_{i-1}] [V_{i+1}, V_{i+2}]}{[V_i, V_{i+2}] [V_{i+1}, V_{i-1}]} = \frac{1}{a_ia_{i+1}},
\]
which is the first equality. To obtain the second one, solve for $a_i$, using the fact that $n$ is odd. 
\end{proof}

In terms of the entries of the first non-trivial row, the entries of a frieze pattern are given by the continuants
\begin{equation*} \label{continuant}
K_{i,j} =
\left|\begin{array}{cccccc}
a_i&1&0&0&\dots&0\\
1&a_{i+1}&1&0&\dots&0\\
\dots&\dots&\dots&\dots&\dots&\dots\\
\dots&\dots&\dots&1&a_{j-1}&1\\
\dots&\dots&\dots&0&1&a_{j}
\end{array}\right|.
\end{equation*}
In particular, $K_{i,n-3+i}=1$ and $K_{i,n-2+i}=0$ (the last two rows of the frieze).
These continuants are related to the ones introduced in Section \ref{sec:CoCont} as follows.

\begin{lemma}
\label{relcont}
One has
$$
D_{i,j-1}= \frac{1}{a_i\ldots a_j} K_{i,j}.
$$
\end{lemma}

\begin{proof}
Substitute $c_i=1/a_ia_{i+1}$ to the formula for $D_{i,j-1}$ and multiply the $i$-th row and the $i$-th column by $\sqrt{a_i}$.
\end{proof}

This is consistent with Lemma \ref{lem:DRelations}: the relation $D_{i,n-3+i}=0$ for closed polygons is equivalent to $K_{i,n-2+i}=0$.

Likewise, consider the discrete Hill equation (\ref{dHill}) with $n$-periodic coefficients $a_i$. Starting with the initial conditions $V_0=(1,0)^T, V_1=(0,1)^T$, one constructs a twisted polygon in $\KK^2$ that projects to a twisted $n$-gon in $\PP^1$:
$$
\Big(
\begin{array}{cccccc}
1, & 0, & -1, & -a_2, & -a_2a_3+1, & -a_2a_3a_4 + a_4+a_2,\\
0, & 1, & a_1, & a_1 a_2 -1, & a_1a_2a_3-a_3-a_1, & a_1a_2a_3a_4 - a_3 a_4 - a_1 a_4 - a_1 a_2 + 1,
\end{array}
$$
$$
\begin{array}{ccc}
 -a_2a_3a_4a_5+a_4a_5 + a_2 a_5 + a_2 a_3 -1, & \ldots\\ 
 a_1a_2a_3a_4a_5 - a_3a_4a_5-a_1a_4a_5-a_1a_2a_5-a_1a_2a_3 + a_5+a_3+a_1, & \ldots
\end{array}
\Big)
$$
The entries are the continuants $K_{1,m}$ and $-K_{2,m}$.

Since each next vector is obtained from the previous two by the recurrence (\ref{dHill}), the monodromy matrix of this twisted polygon
is 
$$
M'_n=
\begin{pmatrix}
0 & -1\\ 1 & a_1
\end{pmatrix}
\begin{pmatrix}
0 & -1\\ 1 & a_2
\end{pmatrix}
\ldots
\begin{pmatrix}
0 & -1\\ 1 & a_n
\end{pmatrix}.
$$

\begin{lemma}
\label{monoeq}
Up to scaling and conjugation, the matrix $M_n$ from Lemma \ref{lem:MonMat2} and the matrix $M'_n$ coincide.
\end{lemma}

\begin{proof}
Since $c_{i-1}=1/a_{i-1} a_i$, we decompose
$$
\begin{pmatrix}
0 & c_{i-1}\\ -1 & 1
\end{pmatrix}
=
\begin{pmatrix}
\frac{1}{a_{i-1}} & 0\\ 0 & -1
\end{pmatrix}
\begin{pmatrix}
0 & \frac{1}{a_{i}}\\ 1 & -1
\end{pmatrix},
$$
regroup, and multiply
$$
\begin{pmatrix}
0 & \frac{1}{a_{i}}\\ 1 & -1
\end{pmatrix}
\begin{pmatrix}
\frac{1}{a_{i}} & 0\\ 0 & -1
\end{pmatrix}
=
\begin{pmatrix}
0 & -\frac{1}{a_{i}}\\ \frac{1}{a_{i}} & 1
\end{pmatrix}
\sim
\begin{pmatrix}
0 & -1\\ 1 & a_i
\end{pmatrix}.
$$
This implies the result.
\end{proof}

%

\section{Integrability on the moduli space of twisted polygons}
\label{sec:IntTwist}

\subsection{Main Theorem 1}
\label{sec:mainthm}

The main result of this section is that the relation $\T$ on the moduli space $\cT$ of twisted ideal polygons is Liouville  integrable. We shall formulate this theorem here; its proof occupies the rest of the section.

Consider the following Poisson structure on $(\KK^*)^n$:
\begin{equation}
\label{eqn:PoissonAlpha}
\begin{split}
\{c_i, c_{i+1}\}_\alpha &= c_ic_{i+1}(c_i + c_{i+1} - \alpha).\\
\{c_i, c_{i+2}\}_\alpha &= c_ic_{i+1}c_{i+2}.
\end{split}
\end{equation}
The values that are not mentioned explicitly are either zero or follow by the skew-symmetry from those mentioned.
For example, $\{c_{i}, c_{i-2}\}_\alpha = - c_{i-2}c_{i-1}c_i$.

Recall that
\[
F_k = \sum_{|I|=k} c_I, \quad I \text{ cyclically sparse}.
\]
Introduce two  functions, one of which is defined only for $n$ even:
\begin{gather*}
E_\alpha = \frac{1}{c_{[n]}} \left( \sum_{k=0}^{\lfloor \frac{n}{2} \rfloor} (-1)^k F_k \alpha^{-k} \right)^2,\\
\frac{c_{even}}{c_{odd}} = \frac{c_2 c_4 \cdots c_n}{c_1 c_3 \cdots c_{n-1}} \text{ for }n\text{ even}.
\end{gather*}
For an even $n$ we have $F_{\frac{n}{2}} = c_{even} + c_{odd}$. The geometric meaning of the
quotient ${c_{even}}/{c_{odd}}$ (the alternating perimeter) is described in Section \ref{sec:AltPer}.

\begin{mainthm}
\label{thm:TwistedModuli}
The relation $\T$ on the moduli space $\cT$ of twisted ideal polygons is completely integrable in the following sense:
\begin{enumerate}
\item
The Poisson structure \eqref{eqn:PoissonAlpha} is invariant under $\T$.
The functions $c_{[n]}$, $F_1, \ldots, F_{\lfloor \frac{n}{2} \rfloor}$ are independent and invariant under $\T$.
\item
For $n$ odd, the Poisson structure \eqref{eqn:PoissonAlpha} has corank $1$, and the function $E_\alpha$ is its Casimir.
The integrals $\frac{F_k^2}{c_{[n]}}$, $k = 1, \ldots {\frac{n-1}2}$ pairwise commute and form, together with $E_\alpha$, an independent system of functions.
\item
For $n$ even, the Poisson structure \eqref{eqn:PoissonAlpha} has corank $2$, with Casimirs $E_\alpha$ and $\frac{c_{even}}{c_{odd}}$.
The integrals $\frac{F_k^2}{c_{[n]}}$, $k = 1, \ldots {\frac{n-2}2}$ pairwise commute and form, together with $E_\alpha$ and $\frac{c_{even}}{c_{odd}}$, an independent system of functions.
\item
The space of Hamiltonian vector fields of the integrals does not depend on the choice of the Poisson structure in the family (\ref{eqn:PoissonAlpha}).
\end{enumerate}
\end{mainthm}

This complete integrability theorem has strong dynamical consequences implied by the Arnold-Liouville theorem, see, e.g., \cite{Arn}.

Namely, the moduli space of twisted ideal polygons $\cT$ is foliated by the level surfaces of the Casimir functions; this symplectic foliation has codimension 1 if $n$ is odd, and codimension 2 if $n$ is even.  The symplectic leaves are foliated by the level surfaces of the remaining integrals, the leaves are Lagrangian submanifolds of the symplectic leaves and they are invariant under $\T$. If $n=2k+1$ then these Lagrangian leaves have dimension $k$, and if $n=2k$ then they have dimension $k-1$.

The commuting Hamiltonian vector fields of the integrals give a locally free action of the Abelian Lie algebra $\R^k$, if $n=2k+1$, and $\R^{k-1}$, if $n=2k$. This gives each leaf an affine structure, and  the map $\T$ is a parallel translation in this affine structure. In particular, one has a Poncelet-style result: if a point of a Lagrangian leaf is periodic, then all points of this leaf are periodic with the same period. 

\subsection{Conjugacy of the Lax matrices}
\label{conjLax}

The following theorem explains the name Lax matrix for $A_\lambda(\pP)$.

\begin{theorem}
\label{thm:LaxMatrix}
If $\pP \T \pQ$, then for all $\lambda \ne \alpha^{-1}$ the matrices $A_\lambda(\pP)$ and $A_\lambda(\pQ)$ are conjugate.
Namely, we have
\[
A_\lambda(\pQ) = A_\mu^{-1}(p_1, q_1) A_\lambda(\pP) A_\mu(p_1, q_1) \  {\rm with} \ \mu = \frac{1-\alpha}{1-\alpha\lambda},
\]
provided that in \eqref{eqn:LaxMatrix} the same representative $M$ of the monodromy is chosen for $\pP$ and for $\pQ$.
\end{theorem}
\begin{proof}
The proof relies on Lemma \ref{lem:LaxConjug} below.
Substitute the formula from this lemma into the definition of $A_\lambda(\pQ)$.
After some cancellations, we get
\begin{multline*}
A_\lambda(\pQ) = M^{-1} A_\mu^{-1}(p_{n+1}, q_{n+1}) A_\lambda(p_n, p_{n+1}) \cdots A_\lambda(p_1, p_2) A_\mu(p_1, q_1)\\
= M^{-1} A_\mu^{-1}(p_{n+1}, q_{n+1}) M A_\lambda(\pP) A_\mu(p_1, q_1) = A_\mu^{-1}(p_1, q_1) A_\lambda(\pP) A_\mu(p_1, q_1),
\end{multline*}
as claimed.
\end{proof}

\begin{lemma}
\label{lem:LaxConjug}
Let $[p_i, p_{i+1}, q_i, q_{i+1}] = \alpha \in \KK^* \setminus \{1\}$ and let $\lambda \in \KK^* \setminus \{\alpha^{-1}\}$.
Put $\mu = \frac{1-\alpha}{1-\alpha\lambda}$. Then we have
\[
A_\lambda(q_i, q_{i+1}) = A_\mu^{-1}(p_{i+1}, q_{i+1}) A_\lambda(p_i, p_{i+1}) A_\mu(p_i, q_i).
\]
\end{lemma}
This is slightly stronger than the corresponding identity for loxodromic transformations $L_\lambda$ and $L_\mu$
because the latter implies the identity for matrices only up to scaling.

Lemma \ref{lem:LaxConjug} can be proved by a direct computation.
In the Appendix we give a less direct but more conceptual proof,
based on special collections of loxodromic transformations along the edges of an ideal tetrahedron.

In the case of closed polygons, all of the above results are applicable, and some of the proofs are simplified.

Theorem \ref{thm:LaxMatrix} implies the following proposition.

\begin{proposition} \label{true2-2}
If $\pP \T \pQ$, and $\pP$ is $\alpha$-related exactly to two different ideal polygons,
then $\pQ$ is also $\alpha$-related to exactly two different (possibly degenerate) ideal polygons.
That is, a $2-2$ dynamics cannot end with a polygon $\alpha$-related to one or infinitely many other polygons, but can only end with a degenerate polygon.
\end{proposition}

\begin{proof}
We are given that $L_{\alpha^{-1}} (\pP)$ is a loxodromic transformation, and we want to show that so is $L_{\alpha^{-1}} (\pQ)$.
A M\"obius transformation is a loxodromic transformation if and only if a matrix that represents it has distinct eigenvalues (whose ratios are the reciprocal derivatives of the transformation at its fixed points). 

Thus $A_{\alpha^{-1}} (\pP)$ has distinct eigenvalues, say, $t_1$ and $t_2$.
According to Theorem \ref{thm:LaxMatrix}, $A_{\lambda} (\pP)$ is conjugate to $A_{\lambda} (\pQ)$ for all $\lambda \neq \alpha^{-1}$. The eigenvalues of $A_{\lambda} (\pQ)$ depend continuously on $\lambda$, so taking limit $\lambda \to \alpha^{-1}$, we conclude that the eigenvalues of  $A_{\alpha^{-1}} (\pQ)$ are also equal to $t_1$ and $t_2$. Hence $L_{\alpha^{-1}} (\pQ)$ is a loxodromic transformation.
\end{proof}

\subsection{Bianchi permutability} \label{permsection}
In this section we show that, properly understood, the 2-2 correspondences $\T$ and $\TT$ commute.

\begin{theorem} \label{Bianchi}
Let $\pP,\pQ$, and $\pR$ be three twisted polygons such that $\pP \T \pQ$ and  $\pP \TT \pR$.
There exists a twisted polygon $\pS$ such that $\pQ \TT \pS$ and $\pR \T \pS$.
\end{theorem}
\begin{proof}
By Lemma \ref{lem:FixedPoint} we have
\[
L_{\alpha^{-1}}(\pP)(q_1) = q_1, \quad L_{\beta^{-1}}(\pP)(r_1) = r_1.
\]
On the other hand, Theorem \ref{thm:LaxMatrix} implies
\begin{align*}
L_{\beta^{-1}}(\pQ) &= L_\mu^{-1}(p_1, q_1) L_{\beta^{-1}}(\pP) L_\mu(p_1, q_1), &\mu = \frac{1-\alpha}{1-\alpha\beta^{-1}},\\
L_{\alpha^{-1}}(\pR) &= L_{\mu'}^{-1}(p_1, r_1) L_{\alpha^{-1}}(\pP) L_{\mu'}(p_1, r_1), &\mu' = \frac{1-\beta}{1-\beta\alpha^{-1}}.
\end{align*}
It follows that the points
\[
s_1 = L_\mu^{-1}(p_1, q_1)(r_1), \quad s'_1 = L_{\mu'}^{-1}(p_1, q_1)(r_1)
\]
are fixed points of the transformations $L_{\beta^{-1}}(\pQ)$ and $L_{\alpha^{-1}}(\pR)$, respectively.
Thus they give rise to $n$-gons $\pS$ and $\pS'$ such that $\pQ \TT \pS$ and $\pR \T \pS'$.

In order to prove that $\pS = \pS'$, it suffices to show that $s_1 = s'_1$:
the coincidence of the other vertices follows from the conjugation identity of Lemma \ref{lem:LaxConjug}.
By Lemma \ref{loxodr} we have
\[
[p_1, q_1, r_1, s_1] = \mu, \quad [p_1, r_1, q_1, s'_1] = \mu'.
\]
Since $\mu + \mu' = 1$, the points $s_1$ and $s'_1$ coincide.
\end{proof}

See Appendix \ref{sec:BianchiAgain} for an alternative proof of Bianchi permutability.

\subsection{Integrals} \label{intsubsect}
Theorem \ref{thm:LaxMatrix} implies that for $\pP \T \pQ$ we have ${\Tr} A_\lambda(\pP) = {\Tr} A_\lambda(\pQ)$.
The trace is a polynomial in $\lambda$, and its coefficients are integrals, shared by the maps $\T$ for all values of $\alpha$.
Under a projective transformation of $\pP$ the matrix $A_\lambda(\pP)$ goes to a conjugate one and preserves its trace.
Hence these integrals will be projectively invariant and expressible in terms of our chosen coordinates on the moduli space of polygons
\[
c_i = [p_i, p_{i+1}, p_{i-1}, p_{i+2}] = \frac{(p_{i-1} - p_i)(p_{i+1} - p_{i+2})}{(p_{i-1} - p_{i+1})(p_i - p_{i+2})}.
\]
In order to remove dependence on the choice of a matrix $M$ representing the monodromy,
instead of the trace we  consider the normalized trace, ${\Tr^2 M}/{\det M}$.

\begin{theorem}
\label{thm:Integrals}
For a closed or twisted polygon $\pP$, we have
\begin{equation}
\label{eqn:TracePolynomial}
\frac{\Tr^2 A_\lambda(\pP)}{\det A_\lambda(\pP)} = \frac{1}{c_{[n]}\lambda^n} \left( \sum_{k=0}^{\lfloor \frac{n}2 \rfloor} (-1)^k F_k(c) \lambda^k \right)^2,
\end{equation}
where, as in (\ref{eqn:Integrals1}),
$$
F_k(c) = \sum_{|I|=k} c_I, \quad I \subset [n] \text{ cyclically sparse}.
$$
\end{theorem}

Comparing Theorems \ref{thm:trmono} and \ref{thm:Integrals}, we see that the normalized trace of the monodromy of a twisted polygon $\pP$  equals the normalized trace of the Lax matrix $A_1(\pP)$. This observation will be expanded in Theorem \ref{thm:TraceComponents} below.

\begin{proof}[Proof of Theorem \ref{thm:Integrals} for closed polygons.]
Let us use a shorthand notation for differences:
$
p_i^{(t)} = p_i - p_{i+t}.
$
In particular, in this notation, we have
\begin{equation}
\label{eqn:cDiff}
c_i = \frac{p_{i-1}^{(1)} p_{i+1}^{(1)}}{p_{i-1}^{(2)} p_i^{(2)}}.
\end{equation}
Observe that
\begin{equation}
\label{eqn:ALambdaSum}
A_\lambda(p_i, p_{i+1}) = \frac{1}{p_i^{(1)}} (B_i - \lambda C_i),
\end{equation}
where
\[
B_i = \begin{pmatrix} p_i\\ 1 \end{pmatrix} \begin{pmatrix} 1 & -p_{i+1} \end{pmatrix}, \quad
C_i = \begin{pmatrix} p_{i+1}\\ 1 \end{pmatrix} \begin{pmatrix} 1 & -p_i \end{pmatrix}.
\]
Expanding the product of \eqref{eqn:ALambdaSum}, we obtain
\begin{equation}
\label{eqn:TraceFirst}
A_\lambda(\pP) = \frac1{p_1^{(1)} \cdots p_n^{(1)}} \sum_{k=0}^n (-1)^k \lambda^k \sum_{|I|=k}\Pi_I,
\end{equation}
where, for any subset $I \subset [n]$, we use the notation
\[
\Pi_I = X_n X_{n-1} \cdots X_1, \quad X_i = \begin{cases} B_i, &\text{ if }i \notin I,\\ C_i, &\text{ if }i \in I \end{cases}.
\]


Note that $C_i C_{i-1} = 0$ for all $i$.
Therefore we can assume that $I = \{i_1, \ldots, i_k\}$ with $i_{s+1} > i_s + 1$: the subset $I$ is sparse.
We are only interested in the trace of the matrix $\Pi_I$, and the trace is invariant under cyclic permutations of the indices.
In particular, $\Tr(\Pi_I) = 0$ if $I$ contains both $1$ and $n$, so that $I$ can be assumed cyclically sparse.

Permute the factors of $\Pi_I$ cyclically in the following way:
\[
\Pi'_I = (B_{i_1-1} \cdots B_{i_k+1} C_{i_k}) \cdots (B_{i_2-1} \cdots B_{i_1+1} C_{i_1}).
\]
(The indices go in the decreasing order and are taken modulo $n$.)
Compute
\[
B_{j-1} \cdots B_{i+1} C_i = p_{j-2}^{(2)} \cdots p_{i+1}^{(2)} p_{i+1}^{(1)}
\begin{pmatrix} p_{j-1}\\ 1 \end{pmatrix} \begin{pmatrix} 1 & -p_i \end{pmatrix}.
\]
Substituting this into the above formula we obtain
\begin{multline*}
\Pi'_I = (p_{i_1-2}^{(2)} \cdots p_{i_k+1}^{(2)} p_{i_k+1}^{(1)}) \cdots (p_{i_2-2}^{(2)} \cdots p_{i_1+1}^{(2)} p_{i_1+1}^{(1)}) \cdot\\
\cdot \begin{pmatrix} p_{i_1-1}\\ 1 \end{pmatrix} \begin{pmatrix} 1 & -p_{i_k} \end{pmatrix} \cdots
\begin{pmatrix} p_{i_2-1}\\ 1 \end{pmatrix} \begin{pmatrix} 1 & -p_{i_1} \end{pmatrix}\\
= \frac{p_1^{(2)} \cdots p_n^{(2)}}{p_{i_1-1}^{(2)} p_{i_1}^{(2)} \cdots p_{i_k-1}^{(2)} p_{i_k}^{(2)}}
p_{i_1+1}^{(1)} \cdots p_{i_k+1}^{(1)} p_{i_2-1}^{(1)} \cdots p_{i_k-1}^{(1)}
\begin{pmatrix} p_{i_1-1}\\ 1 \end{pmatrix} \begin{pmatrix} 1 & -p_{i_1} \end{pmatrix}
\end{multline*}
The trace of the ``column times row'' matrix on the very right is $p_{i_1-1}^{(1)}$.
It follows that
\[
\Tr(\Pi_I) = \Tr(\Pi'_I) = p_1^{(2)} \cdots p_n^{(2)} c_I.
\]
From \eqref{eqn:TraceFirst} we obtain
\[
\Tr(A_\lambda(\pP)) = \frac{p_1^{(2)} \cdots p_n^{(2)}}{p_1^{(1)} \cdots p_n^{(1)}} \sum_{k=0}^{\lfloor\frac{n}{2}\rfloor}
(-1)^k \lambda^k \sum_{|I|=k} c_I, \quad I \text{ cyclically sparse}.
\]
To identify the factor before the sum, compute, using \eqref{eqn:cDiff}:
\[
c_{[n]} = \frac{p_0^{(1)} p_1^{(1)} \left(p_2^{(1)} \cdots p_{n-1}^{(1)}\right)^2 p_n^{(1)} p_{n+1}^{(1)}}{p_0^{(2)}\left(p_1^{(2)} \cdots p_{n-1}^{(2)}\right)^2p_n^{(2)}}
= \left(\frac{p_1^{(1)} \cdots p_n^{(1)}}{p_1^{(2)} \cdots p_n^{(2)}} \right)^2.
\]
It remains to recall that $\det A_\lambda(p,q) = \lambda$, so  $\det A_\lambda(\pP) = \lambda^n$, and the theorem is proved.
\end{proof}

\begin{proof}[Proof of Theorem \ref{thm:Integrals} for twisted polygons.]
Similarly to \eqref{eqn:TraceFirst}, we have
\[
\Tr (A_\lambda(\pP)) = \sum_{k=0}^n (-1)^k \lambda^k \sum_{|I|=k} \Tr \left( \frac{M^{-1} X_n \cdots X_1}{p_1^{(1)} \cdots p_n^{(1)}} \right),
\]
where $X_i = C_i$ or $B_i$ according as $i \in I$ or not, and $B_i$, $C_i$ are as in \eqref{eqn:ALambdaSum}.
We claim that
\[
\Tr \left(\frac{M^{-1} X_n \cdots X_1}{p_1^{(1)} \cdots p_n^{(1)}}\right) = \Tr \left(\frac{M^{-1} X_{i+n} \cdots X_{i+1}}{p_{i+1}^{(1)} \cdots p_{i+n}^{(1)}}\right)
\]
for every $i$, where we put $X_{i+n} = B_{i+n}$ iff  $X_i = B_i$.
Indeed, due to
\[
M^{-1} A_\lambda(p_i, p_{i+1}) = A_\lambda(p_{i-n}, p_{i+1-n}) M^{-1},
\]
we have
\[
\frac{M^{-1} B_i}{p_i^{(1)}}= \frac{B_{i-n}M^{-1}}{p_{i-n}^{(1)}}, \quad \frac{M^{-1} C_i}{p_i^{(1)}} = \frac{C_{i-n}M^{-1}}{p_{i-n}^{(1)}},
\]
which implies
\[
\frac{M^{-1} X_n \cdots X_1}{p_1^{(1)} \cdots p_n^{(1)}} = \frac{X_0 M^{-1} X_{n-1} \cdots X_1}{p_0^{(1)} \cdots p_{n-1}^{(1)}}.
\]
Moving $X_0$ to the end of the product, conjugates the matrix and therefore does not change the trace.

If $I = \{i_1 < \cdots < i_k\}$, then shift the indices as follows:
\[
\Pi'_I = (B_{i_1+n-1} \cdots B_{i_k+1} C_{i_k}) \cdots (B_{i_2-1} \cdots B_{i_1+1} C_{i_1}).
\]
The difference from the case of closed polygons is that the indices are not taken modulo $n$.
It follows that
\[
\Pi'_I = \frac{p_{i_1-1}^{(2)} \cdots p_{i_1+n-2}^{(2)}}{p_{i_1-1}^{(2)} p_{i_1}^{(2)} \cdots p_{i_k-1}^{(2)} p_{i_k}^{(2)}}
p_{i_1+1}^{(1)} \cdots p_{i_k+1}^{(1)} p_{i_2-1}^{(1)} \cdots p_{i_k-1}^{(1)}
\begin{pmatrix} p_{i_1+n-1}\\ 1 \end{pmatrix} \begin{pmatrix} 1 & -p_{i_1} \end{pmatrix}.
\]

For $M = \begin{pmatrix} a & b\\ c & d \end{pmatrix}$ we have
\[
M \begin{pmatrix} z\\ 1 \end{pmatrix} = (cz+d) \begin{pmatrix} \frac{az+b}{cz+d}\\ 1 \end{pmatrix}.
\]
Therefore the following holds for every $i$:
\[
M^{-1} \begin{pmatrix} p_{i+n}\\ 1 \end{pmatrix} = \frac{1}{cp_i + d} \begin{pmatrix} p_i\\ 1 \end{pmatrix}.
\]
Combining this with the previous computations we obtain
\begin{multline*}
\Tr \left( \frac{M^{-1} X_n \cdots X_1}{p_1^{(1)} \cdots p_n^{(1)}} \right)
= \Tr \left( \frac{M^{-1} \Pi'_I}{p_{i_1}^{(1)} \cdots p_{i_1+n-1}^{(1)}} \right)\\
= \frac{1}{cp_{i_1-1}+d} \cdot \frac{p_{i_1-1}^{(2)} \cdots p_{i_1+n-2}^{(2)}}{p_{i_1}^{(1)} \cdots p_{i_1+n-1}^{(1)}} \cdot
\frac{p_{i_1-1}^{(1)} p_{i_1+1}^{(1)} \cdots p_{i_k-1}^{(1)} p_{i_k+1}^{(1)}}{p_{i_1-1}^{(2)} p_{i_1}^{(2)} \cdots p_{i_k-1}^{(2)} p_{i_k}^{(2)}}\\
= \frac{1}{cp_{i_1-1}+d} \cdot \frac{p_{i_1-1}^{(2)} \cdots p_{i_1+n-2}^{(2)}}{p_{i_1}^{(1)} \cdots p_{i_1+n-1}^{(1)}} c_{I},
\end{multline*}
which we rewrite as
\[
\Tr \left( \frac{M^{-1} X_n \cdots X_1}{p_1^{(1)} \cdots p_n^{(1)}} \right) = x_{i_1} c_{I}, \quad
x_i = \frac{1}{cp_{i-1}+d} \frac{p_{i-1}^{(2)} \cdots p_{i+n-2}^{(2)}}{p_i^{(1)} \cdots p_{i+n-1}^{(1)}}.
\]
Let us show that $x_i$ is independent of $i$.
Indeed, with the help of
\begin{multline*}
\frac{p_{i+n}^{(t)}}{p_i^{(t)}} = \frac{\frac{ap_i + b}{cp_i + d} - \frac{ap_{i+t} + b}{cp_{i+t} + d}}{p_i - p_{i+t}}
= \frac{(ad-bc)p_i - (ad-bc)p_{i+t}}{(cp_i+d)(cp_{i+t}+d)(p_i-p_{i+t})}\\ = \frac{\det M}{(cp_i+d)(cp_{i+t}+d)},
\end{multline*}
we compute
\[
\frac{x_{i+1}}{x_i} = \frac{cp_{i-1}+d}{cp_i+d} \cdot \frac{p_{i+n-1}^{(2)}}{p_{i-1}^{(2)}} \cdot \frac{p_i^{(1)}}{p_{i+n}^{(1)}} = 1.
\]
Taking into account
\[
c_{[n]} = \frac{p_0^{(1)} p_1^{(1)} \left(p_2^{(1)} \cdots p_{n-1}^{(1)}\right)^2 p_n^{(1)} p_{n+1}^{(1)}}{p_0^{(2)}\left(p_1^{(2)} \cdots p_{n-1}^{(2)}\right)^2p_n^{(2)}},
\]
we obtain
\begin{multline*}
x_1^2 = \frac{1}{(cp_0+d)^2} \frac{(p_0^{(2)} \cdots p_{n-1}^{(2)})^2}{(p_1^{(1)} \cdots p_n^{(1)})^2}
= \frac{1}{(cp_0+d)^2 c_{[n]}} \cdot \frac{p_{n+1}^{(1)}}{p_1^{(1)}} \cdot \frac{p_0^{(1)}}{p_n^{(1)}} \cdot \frac{p_0^{(2)}}{p_n^{(2)}}
= \frac{\det M^{-1}}{c_{[n]}}.
\end{multline*}

It follows that
\[
\Tr^2 A_\lambda(\pP) = \frac{\det M^{-1}}{c_{[n]}} \left( \sum_{k=0}^{\lfloor\frac{n}{2}\rfloor} (-1)^k \lambda^k F_k(\pP) \right)^2.
\]
Combined with
\[
\det A_\lambda(\pP) = \lambda^n \det M^{-1},
\]
this implies the theorem.
\end{proof}

\begin{corollary}
\label{cor:FIntegrals}
The polynomials $c_{[n]}$ and $F_k(c)$, $k = 1, \ldots, \left\lfloor \frac{n}{2} \right\rfloor$ are integrals of the relation $\T$.
\end{corollary}
\begin{proof}
From Theorem \ref{thm:LaxMatrix} we know that the coefficients of the (Laurent) polynomial on the right hand side of \eqref{eqn:TracePolynomial}
are integrals of $\T$.
The coefficient of this polynomial at $\lambda^{-n}$ is $c_{[n]}^{-1}$, therefore $c_{[n]}$ is an integral.
The coefficient at $\lambda^{k-n}$ has the form $c_{[n]}^{-1}((-1)^k 2F_k + \text{terms in }F_{<k})$;
therefore, by induction, all $F_k$ are integrals.
\end{proof}


\subsection{Independence of the integrals}
\label{indeptwist}

A collection of functions  is said to be independent on a subset of their common domain if their differentials are linearly independent at every point of  this set.

\begin{theorem}
\label{thm:IntIndependence}
The functions $c_{[n]}$ and $F_k(c)$, $k = 1, \ldots, \left\lfloor \frac{n}{2} \right\rfloor$
are independent on a Zariski open subset $\cT$.
\end{theorem}

\begin{proof}
The argument is a simplified version of the proof of a similar statement for the monodromy integrals of the pentagram map in \cite{Sch}, see also \cite{OST2}.

Consider an $(\lfloor{n/2}\rfloor+1) \times n$ matrix $M$ whose columns are the gradient vectors  $\nabla F_1,\nabla F_2,\ldots,\nabla F_{\lfloor{n/2}\rfloor},\nabla c_{[n]}$. The $j$th row of $M$ consists of the first partial derivatives with respect to $c_j$ of the functions $F_1,\ldots, F_{\lfloor{n/2}\rfloor}, c_{[n]}$.

The integrals are polynomial functions, therefore it suffices to show that $M$ has the maximal rank $\lfloor{n/2}\rfloor+1$ at some point: the rank is maximal in a Zariski open set.

We estimate rk $M$ at a special point
$$
c_1=\eps, c_2=\eps^2,\ldots, c_n = \eps^n
$$
by identifying its non-degenerate square submatrix. 

To prove that this square matrix is non-degenerate, we divide  each column by the highest power of $\eps$ that divides all its entries and then let $\eps \to 0$. We show that the limiting matrix is non-degenerate, and since the rank may only drop at a special value (zero) of the parameter $\eps$, we conclude that rk $M =  \lfloor{n/2}\rfloor+1$ almost everywhere. 

To implement this argument, delete the even rows $2,4,6,\ldots,2(\lfloor{(n-1)/2}\rfloor)$ from $M$ and call the resulting square matrix $M'$. The entries of $M'$ (except for the last column) are the first partial derivatives of the functions $F_k$ with respect to the variables $c_j$ with $j$ odd.

Start with the last column of $M'$.  Its $j$th entry is $c_{[n]}/c_j$, and the lowest exponent of $\eps$, namely, $n(n-1)/2$, is attained in the last row. Divide the entries by $\eps^{n(n-1)/2}$ and let $\eps \to 0$. The last column becomes $(0,\ldots,0,1)^T$.

Next, consider the column  before the last. The term on or above the diagonal that has the lowest exponent of $\eps$ is 
$c_{\{1,3,\ldots,2(\lfloor{n/2}\rfloor)-3\}}$, and it occurs on the diagonal. Dividing by the respective power of $\eps$ and taking limit $\eps\to 0$, this column becomes $(0,\ldots,0,1,*)^T$.

Continuing in the same way, each time the lowest power of $\eps$ among the terms on and above the diagonal occurs on the diagonal. Dividing by the respective power of $\eps$ and taking limit $\eps\to 0$,
we turn $M'$ into a lower-triangular matrix with ones on the diagonal. Therefore  $M'$ is non-degenerate, as needed.
\end{proof}

\subsection{Extended dynamics and monodromy integrals}
\label{sec:ExtDynTwist}
Let us describe the relation $\T$ in terms of auxiliary coordinates.

\begin{lemma}
\label{lem:AuxVar}
Let $\pP$ and $\pQ$ be two twisted or closed polygons such that $\pP \T \pQ$. Set
\begin{align*}
c_i &= [p_i, p_{i+1}, p_{i-1}, p_{i+2}],& x_i &= [p_i, p_{i+1}, p_{i-1}, q_i],\\
d_i &= [q_i, q_{i+1}, q_{i-1}, q_{i+2}],& y_i &= [q_i, q_{i+1}, q_{i-1}, p_i].
\end{align*}
Then we have
\begin{equation}
\label{eqn:ExtDyn}
c_i = \alpha x_i (1 - x_{i+1}), \quad d_i = \alpha y_i (1 - y_{i+1}), \quad x_i + y_i = 1.
\end{equation}

Conversely, for any $x \in (\KK \setminus\{0,1\})^n$ and any $\alpha \in \KK \setminus \{0, 1\}$,
the numbers $c_i$ and $d_i$ given by formulas \eqref{eqn:ExtDyn} define a pair $\pP \T \pQ$, unique up to a simultaneous projective transformation.
\end{lemma}
\begin{proof}
The multiplicativity of the cross-ratio and the relation $\pP \T \pQ$ imply
\begin{multline*}
[p_i, q_i, p_{i-1}, p_{i+1}] = [p_i, q_i, p_{i-1}, q_{i-1}] [p_i, q_i, q_{i-1}, q_{i+1}] [p_i, q_i, q_{i+1}, p_{i+1}]\\
= (1 - \alpha) [p_i, q_i, q_{i-1}, q_{i+1}] \frac{1}{1-\alpha} = [p_i, q_i, q_{i-1}, q_{i+1}].
\end{multline*}
Using the action of the permutations of the variables on the cross-ratio, this equation rewrites as
\[
1 - \frac1{1-x_i} = 1 - \frac1{y_i},
\]
which implies $x_i + y_i = 1$.

In a similar way we have
\[
c_i = [p_i, p_{i+1}, p_{i-1}, q_i] [p_i, p_{i+1}, q_i, q_{i+1}] [p_i, p_{i+1}, q_{i+1}, p_{i+2}] = x_i \alpha (1 - x_{i+1}),
\]
and similarly, replacing $c$ with $d$, $x$ with $y$, and exchanging $p$ and $q$.

Let us prove the second statement of the lemma. For given $x$ and $\alpha$, compute the $n$-tuple $c_1, \ldots, c_n$
and extend it to a periodic doubly infinite sequence.
This sequence defines a projective class of a twisted $n$-gon.
Let $\pP$ be a representative of this class.
Then the $n$-tuple $x_1, \ldots, x_n$ determines an $n$-periodic doubly infinite sequence of points $q_i$.

We claim that $\pQ = (q_i)$ is a twisted polygon with the same monodromy as $\pP$.
Indeed, from $\Phi(p_i) = p_{i+n}$ and $x_i = [p_i, p_{i+1}, p_{i-1}, q_i]$, it follows that $\Phi(q_i) = q_{i+n}$.
Thus we have a pair of twisted polygons $\pP$ and $\pQ$ with the same monodromy (in the special case when the monodromy is the identity, the polygons are closed).
The equation $c_i = \alpha x_i(1-x_{i+1})$ implies $[p_i, p_{i+1}, q_i, q_{i+1}] = \alpha$, thus we have $\pP \T \pQ$.
\end{proof}

Observe that scaling $(c_i)$ and $(d_i)$ simultaneously by $t$  yields two $t\alpha$-related sequences of cross-ratios.
This serves an alternative source of integrals and provides an alternative proof of Corollary \ref{cor:FIntegrals}.

\begin{theorem}
\label{thm:TraceComponents} Let $M$ be a monodromy matrix of a twisted polygon. 
The homogeneous components of the normalized trace ${\Tr^2 M}/{\det M}$ are integrals of the relation $\T$.
\end{theorem}

\begin{proof}
For a given pair $\pP \T \pQ$, let $c$ and $d$ be the corresponding sequences of cross-ratios. Denote the normalized trace of the monodromy by $F(c)$.
Since $\pP$ and $\pQ$ have the same monodromy, we have $F(c) = F(d)$. 

At the same time, Lemma \ref{lem:AuxVar} implies that there exist twisted polygons $\pP_t \stackrel{t\alpha}{\sim} \pQ_t$
with the cross-ratio sequences $tc$ and $td$ respectively.
It follows that $F(tc) = F(td)$ for all $t$, implying that the homogeneous components of $F(c)$ are equal to the respective homogeneous components of $F(d)$.
\end{proof}

Since the function $F(c)$ is given by the right hand side of \eqref{eqn:TraceC},
Theorem \ref{thm:TraceComponents} implies that $c_{[n]}$ and $F_k(c)$ are integrals.

Let us reformulate the statement of Theorem \ref{thm:TraceComponents}. Given a projective equivalence class of a twisted polygon $\pP$ with the cross-ratios $c$, let $\pP_t$ be the projective equivalence class of the twisted polygon whose cross-ratios equal $tc$, and let $M_t$ denote the monodromy of $\pP_t$. Since the polynomials $F_k$ are of degree $k$,  one has
\begin{equation}
\label{scaledmono}
\frac{\Tr M_t}{\sqrt{\det M_t}} = \sum_{k=0}^{\lfloor \frac{n}{2} \rfloor} (-1)^k t^{k-\frac{n}{2}} \frac{F_k}{\sqrt{c_{[n]}}}.
\end{equation}

\subsection{Extended dynamics and exceptional polygons}
\label{sec:ExtDynExcPol}
Recall Lemma \ref{lem:FixedPoint} and Corollary \ref{cor:2-2}: for an ideal polygon $\pP$ the number of polygons $\pQ$ such that $\pP \T \pQ$
is equal to $1$, $2$, or $\infty$, according to whether the Lax matrix $L_{\alpha^{-1}}(\pP)$ is parabolic, non-parabolic, or the identity.
(We consider only the complex case.)
Due to the M\"obius invariance, the type of the Lax matrix depends only on the cross-ratios $c$ of $\pP$.
In this section we describe the exceptional polygons (those having one or infinitely many $\alpha$-related polygons) in terms of $c$.

Let $X = (\C \setminus \{0, 1\})^n$ and $C = (\C \setminus \{0\})^n$.
Let $\pi_\alpha : X \to C$ be the map given by the formulas of Lemma \ref{lem:AuxVar} that express $c$ in terms of $x$:
\[
\pi_\alpha(x) = c, \text{ where } c_i = \alpha x_i(1-x_{i+1}).
\]
The preimage of a point $c \in C$ under $\pi_{\alpha}^{-1}$ is found from the fixed points of the composition of M\"obius maps
\begin{equation}
\label{eqn:TAlpha}
x_i = \frac{c_i}{\alpha(1 - x_{i+1})}.
\end{equation}
Thus we can represent $C$ as a disjoint union
\[
C = C_1(\alpha) \sqcup C_2(\alpha) \sqcup C_\infty(\alpha),
\]
where $C_i(\alpha) = \{c \in C \mid |\pi_{\alpha}^{-1}(c)| = i\}$.

\begin{lemma}
One has $c \in C_i(\alpha)$ if and only if an ideal polygon with cross-ratios $c$ is $\alpha$-related to $i$ ideal polygons (possibly degenerate).
\end{lemma}
\begin{proof}
We use Lemma \ref{lem:AuxVar}.
The complex number $x_i$ determines the position of the point $q_i$ with respect to the points $p_{i-1}, p_i, p_{i+1}$.
Thus for every polygon $\pP$ with cross-ratios $c$
the polygons $\alpha$-related to $\pP$ are in a one-to-one correspondence with the elements of $\pi_{\alpha}^{-1}(c)$.
\end{proof}


\begin{theorem}
\label{thm:InfAlpha}
Let $\pP$ be a twisted or closed $n$-gon with cross-ratio coordinates $c_1, \ldots, c_n$, and let $\alpha \in \C \setminus \{0,1\}$.
Then the following holds.

There are infinitely many $n$-gons $\alpha$-related to $\pP$ if and only if the cross-ratios $\alpha^{-1}c_1, \ldots, \alpha^{-1}c_n$
determine a projective equivalence class of closed polygons, that is,
the numbers $\alpha^{-1}c_1, \ldots, \alpha^{-1}c_n$ satisfy the equations
\begin{equation}
\label{eqn:DAlpha}
D_{i, n-3+i}(\alpha^{-1}c) = 0
\end{equation}
with $D_{i,j}$ as in Section \ref{sec:CoCont}.

There is exactly one $n$-gon $\alpha$-related to $\pP$ if and only if the cross-ratios $\alpha^{-1}c_1, \ldots, \alpha^{-1}c_n$
determine a projective equivalence class of a twisted polygon with parabolic monodromy, that is if and only if
\begin{equation}
\label{eqn:TraceAlpha}
\frac{1}{c_{[n]}} \left( \sum\limits_{k=0}^{\lfloor \frac{n}2\rfloor} (-1)^k \alpha^{\frac{n}2 - k} F_k(c) \right)^2 = 4
\end{equation}
and not all of the equations $D_{i, n-3+i}(\alpha^{-1}c) = 0$ are satisfied.
\end{theorem}
Note that equations \eqref{eqn:DAlpha} imply equation \eqref{eqn:TraceAlpha} as was shown in Section \ref{calcmono}.
\begin{proof}
The composition of maps \eqref{eqn:TAlpha} corresponds to the product of matrices
\[
M(\alpha^{-1}c) = \begin{pmatrix} 0 & \alpha^{-1}c_1\\ -1 & 1 \end{pmatrix} \cdots \begin{pmatrix} 0 & \alpha^{-1}c_n\\ -1 & 1 \end{pmatrix},
\]
which is the monodromy of a polygon with cross-ratios $\alpha^{-1}c$.
Thus the cardinality of the preimage $\pi^{-1}_\alpha(c)$ is $\infty$ or $1$ if this monodromy is, respectively, the identity or parabolic.
Lemma \ref{lem:DRelations} and Corollary \ref{cor:SchoeneGleichung} give necessary and sufficient conditions of triviality and parabolicity of the monodromy.
\end{proof}

\begin{lemma}
\label{lem:CoordInfFiber}
The infinite fiber $\pi_\alpha^{-1}(c)$ for $c \in C_\infty(\alpha)$ consists of points with the coordinates $x_i = [p_i, p_{i+1}, p_{i-1}, z]$,
where $\pP$ is any polygon with cross-ratios $\alpha^{-1}c$, and $z$ is any point in $\CP^1$.
\end{lemma}
\begin{proof}
A point in $\pi_\alpha^{-1}(c)$ is uniquely determined by any one of its coordinates.
Further, for $x_i$ given by the formula in the lemma we have
\begin{multline*}
\alpha x_i (1-x_{i+1}) = \alpha [p_i, p_{i+1}, p_{i-1}, z] [p_{i+1}, p_i, p_{i+2}, z]\\
= \alpha [p_i, p_{i+1}, p_{i-1}, z] [p_i, p_{i+1}, z, p_{i+2}]
= \alpha [p_i, p_{i+1}, p_{i-1}, p_{i+2}] = c_i,
\end{multline*}
as needed.
\end{proof}

We apply Theorem \ref{thm:InfAlpha} in Section \ref{smallsection} to give examples of closed $n$-gons $\alpha$-related
to infinitely many other closed $n$-gons.

\subsection{Second set of auxiliary variables}
\label{sec:AuxVar2}
In the space $X = (\C \setminus \{0,1\})^n$ perform the following  change of variables:
\[
u_i = \frac{1}{x_i} - 1, \quad x_i = \frac{1}{1+u_i}.
\]
Then $u \in (\C \setminus \{0,-1\})^n$, and the projection $\pi_\alpha$ writes in the new coordinates as
\[
c_i = \frac{\alpha u_{i+1}}{(1+u_i)(1+u_{i+1})} = \frac{\alpha}{(1+u_i)(1+\frac{1}{u_{i+1}})}.
\]
The coordinates $u_i$ will play an important role in the description of Poisson structures on $\cT$ and $\cP$.
In this section we will use them to describe the union of infinite fibers $X_\infty$.

\begin{lemma}
\label{lem:DijU}
Under the above substitutions we have
\[
D_{i,j}(\alpha^{-1}c) = \frac{1 + u_i + u_i u_{i+1} + \cdots + u_{[i,j+1]}}{(1+u_i)(1+u_{i+1}) \cdots (1+u_{j+1})},
\]
\[
\frac{\Tr^2 M(\alpha^{-1}c)}{\det M(\alpha^{-1}c)} = \frac{(1+u_{[n]})^2}{u_{[n]}}.
\]
\end{lemma}
\begin{proof}
The first formula is proved by induction on $j$ using the recurrence \eqref{eqn:DRecurrence}.
The second formula follows from the first one and the equation $\Tr M = D_{1,n-1} - \alpha^{-1}c D_{2,n-2}$.
\end{proof}

\begin{proposition}
In the coordinates $u_i$ the union of the infinite fibers $\pi_\alpha^{-1}(C_\infty(\alpha))$ is the solution set of the following system of equations:
\begin{eqnarray}
1 + u_i + u_iu_{i+1} + \cdots + u_{[i, i+n-2]} = 0 \label{eqn:USum},\\
u_{[n]} = 1. \label{eqn:UProd}
\end{eqnarray}
Any two equations \eqref{eqn:USum} for consecutive $i$ imply the rest.
Any equation \eqref{eqn:USum} together with equation \eqref{eqn:UProd} implies the rest.
\end{proposition}
\begin{proof}
The fact that the union of infinite fibers is the solution set of \eqref{eqn:USum}
follows from Theorem \ref{thm:InfAlpha} and Lemma \ref{lem:DijU}.
Equation \eqref{eqn:UProd} follows from them because it expresses the parabolicity of the monodromy $M(\alpha^{-1}c)$.

The statements about the mutual dependence of the equations are easily checked, and we leave it to the reader.
\end{proof}

\subsection{Poisson structures}
\label{sec:IntProof}
The following pair of Poisson brackets appears in the study of the Volterra lattice, see \cite{Suris97} and the book \cite{SurB}.
These brackets are compatible, that is, any linear combination of them is also a Poisson bracket:
\begin{equation}
\label{eqn:BiHamiltonian}
\begin{split}
\{c_i, c_{i+1}\}_1 &= c_ic_{i+1}\\
\{c_i, c_{i+1}\}_2 &= c_ic_{i+1}(c_i + c_{i+1})\\
\{c_i, c_{i+2}\}_2 &= c_ic_{i+1}c_{i+2}
\end{split}
\end{equation}
(As we mentioned earlier, the values that are not mentioned explicitly are either zero or follow by the skew-symmetry from those mentioned.)

\begin{theorem}
\label{thm:InvBracket}
The Poisson bracket
\begin{equation}
\label{eqn:PoissonC}
\{ \cdot, \cdot \}_{\alpha^{-1}} = \{ \cdot, \cdot\}_1 - \alpha^{-1} \{ \cdot, \cdot \}_2
\end{equation}
is invariant with respect to the correspondence $\T$.
\end{theorem}

To prove this theorem, consider the  commutative diagram that follows from Lemma \ref{lem:AuxVar}:
\begin{equation}
\label{eqn:TauDiagram}
\begin{tikzcd}
X \arrow[r, "\tau"] \arrow[d, "\pi_{\alpha}"] & X \arrow[d, "\pi_{\alpha}"]\\
C \arrow[r, "\T"] & C
\end{tikzcd}
\end{equation}
where $\tau(x) = y$ with $y_i = 1-x_i$. 
Our arguments are parallel to those in \cite{Evr}.

Consider the following Poisson bracket on the space $X$:
\begin{equation}
\label{eqn:PoissonX}
\{x_i, x_{i+1}\} = x_i(1-x_i)x_{i+1}(1-x_{i+1}).
\end{equation}

\begin{lemma}
\label{lem:PoissonMap}
The map $\pi_\alpha$ is Poisson with respect to the bracket \eqref{eqn:PoissonX} on $X$
and \eqref{eqn:PoissonC} on $C$.
\end{lemma}
\begin{proof}
The proof is a direct computation.
First compute the partial derivatives of $c$ with respect to $x$:
\[
\frac{\partial c_i}{\partial x_i} = \alpha (1 - x_{i+1}), \quad \frac{\partial c_i}{\partial x_{i+1}} = -\alpha x_i.
\]
It follows that
\begin{multline*}
\{c_{i-1} \circ \pi_\alpha, c_i \circ \pi_\alpha\} = \frac{\partial c_{i-1}}{\partial x_{i-1}} \frac{\partial c_i}{\partial x_i} \{x_{i-1}, x_i\}
+ \frac{\partial c_{i-1}}{\partial x_i} \frac{\partial c_i}{\partial x_{i+1}} \{x_i, x_{i+1}\}\\
= \alpha^2 (1-x_i)(1-x_{i+1}) x_{i-1}(1-x_{i-1}) x_i(1-x_i) + \alpha^2 x_{i-1} x_i x_i(1-x_i) x_{i+1}(1-x_{i+1})\\
= \alpha^2 x_{i-1}x_i(1-x_i)(1-x_{i+1}) (1 - x_{i-1} - x_i + x_{i-1}x_i + x_ix_{i+1})\\
= c_{i-1}c_i \left( 1 - \frac{c_{i-1}+c_i}{\alpha} \right) = \{c_{i-1}, c_i\}_{\alpha^{-1}},
\end{multline*}
\begin{multline*}
\{c_{i-1} \circ \pi_\alpha, c_{i+1} \circ \pi_\alpha\} = \frac{\partial c_{i-1}}{\partial x_i} \frac{\partial c_{i+1}}{\partial x_{i+1}} \{x_i, x_{i+1}\}\\
= -\alpha^2 x_{i-1}(1-x_{i+2}) x_i(1-x_i) x_{i+1}(1-x_{i+1})\\
= -\frac1{\alpha} c_{i-1}c_ic_{i+1} = \{c_{i-1}, c_{i+1}\}_{\alpha^{-1}},
\end{multline*}
and the lemma is proved.
\end{proof}

\begin{proof}[Proof of Theorem \ref{thm:InvBracket}]
It is easy to see that the bracket \eqref{eqn:PoissonX} is preserved by the involution $\tau \colon x_i \mapsto 1-x_i$.
Lemma \ref{lem:PoissonMap} and the commutative diagram \eqref{eqn:TauDiagram}  imply that the bracket $\{\cdot, \cdot\}_{\alpha^{-1}}$
is preserved by the relation $\T$.
\end{proof}

We now proceed to showing that the integrals found in Section \ref{intsubsect} are in involution with respect to the bracket $\{\cdot, \cdot\}_{\alpha^{-1}}$.
For this we will use the variables $u_i$ introduced in Section \ref{sec:AuxVar2}.
The Poisson bracket \eqref{eqn:PoissonX} takes in these variables the form
\begin{equation}
\label{eqn:PoissonU}
\{u_i, u_{i+1}\} = u_iu_{i+1},
\end{equation}
as can be easily checked.

\begin{lemma}
\label{lem:CasimirAlpha}
For every $\alpha \in \C \setminus \{0\}$, the function
\[
E_\alpha = \frac{1}{c_{[n]}} \left( \sum\limits_{k=0}^{\lfloor \frac{n}2\rfloor} (-1)^k \alpha^{\frac{n}2 - k} F_k(c) \right)^2
\]
is a Casimir of the bracket \eqref{eqn:PoissonC}.
\end{lemma}
\begin{proof}
Since the map $\pi_\alpha$ is Poisson, it suffices to show that
the pullback of this function to $X$ via $\pi_{\alpha}$ is a Casimir of the Poisson structure on $X$.

It is easily seen that the function $u_{[n]}$ on $X$ is a Casimir of the bracket \eqref{eqn:PoissonU}.
On the other hand, by Lemma \ref{lem:DijU} we have
\[
E_\alpha \circ \pi_\alpha = 2 + u_{[n]} + \frac{1}{u_{[n]}}
\]
and the lemma follows.
\end{proof}

\begin{theorem}
\label{thm:FCommute}
The functions $\frac{F_k^2}{c_{[n]}}$ are in involution with respect to each of the Poisson brackets \eqref{eqn:BiHamiltonian}.
\end{theorem}
\begin{proof}
Choose $c \in C$ and choose one of the branches of the function $\sqrt{E_\alpha}$ in its neighborhood.
By Lemma \ref{lem:CasimirAlpha}, this function annihilates the Poisson bracket $\{\cdot, \cdot\}_{\alpha^{-1}}$.
Denoting $E_k(c) = \frac{F_k(c)}{\sqrt{c_{[n]}}}$, we have, for all functions $f \colon C \to \C$
and for all $\alpha \in \C \setminus \{0\}$,
\begin{multline*}
0 = \left\{ f, \sum_{k=0}^{\lfloor\frac{n}{2}\rfloor} (-1)^k \alpha^{\frac{n}{2}-k} E_k(c) \right\}_{\alpha^{-1}}\\
= \left\{ f, \sum_{k=0}^{\lfloor\frac{n}{2}\rfloor} (-1)^k \alpha^{\frac{n}{2}-k} E_k(c) \right\}_1
- \alpha^{-1} \left\{ f, \sum_{k=0}^{\lfloor\frac{n}{2}\rfloor} (-1)^k \alpha^{\frac{n}{2}-k} E_k(c) \right\}_2\\
= \sum_{k=0}^{\lfloor\frac{n}{2}\rfloor+1} (-1)^k \alpha^{\frac{n}{2}-k} (\{f, E_k\}_1 + \{f, E_{k-1}\}_2),
\end{multline*}
where we put $E_{-1} = E_{\lfloor\frac{n}{2}\rfloor+1} = 0$.
It follows that
\[
\{f, E_k\}_1 = - \{f, E_{k-1}\}_2
\]
for all $f$ and $k$.
This allows us to apply the Lenard-Magri scheme:
\[
\{E_k, E_l\}_1 = -\{E_k, E_{l-1}\}_2 =  \{E_{l-1}, E_k\}_2 = -\{E_{l-1}, E_{k+1}\}_1 = \{E_{k+1}, E_{l-1}\}_1.
\]
This implies $\{E_k, E_l\}_1 = \{E_{k+i}, E_{l-i}\}$, which vanishes for $i$ sufficiently large.
It also follows that $\{E_k, E_l\}_2 = 0$.
Thus the (globally defined) functions $E_k^2 = \frac{F_k^2}{c_{[n]}}$ are in involution with respect to each of the two brackets,
and the theorem is proved.
\end{proof}

The space of the Hamiltonian vector fields obtained by the Lenard-Magri scheme does not depend on the choice of the Poisson bracket in the pencil, and one obtains the following corollary.

\begin{corollary}
\label{cor:indepstr}
The space of the Hamiltonian vector fields of the integrals is independent on the choice of the Poisson structure in the pencil $\{ \cdot, \cdot \}_{\alpha}$.
\end{corollary}

We  finish the proof of Main Theorem \ref{thm:TwistedModuli}: what remains to be done is to compute the rank of the Poisson structure and find all the Casimirs.
At the regular values of $\pi_\alpha$, the Poisson structure $\{\cdot, \cdot\}_{\alpha^{-1}}$
is equivalent to the Poisson structure \eqref{eqn:PoissonX}, which  has the form \eqref{eqn:PoissonU} in the coordinates $u_i$.
A logarithmic change of variables transforms this structure into a constant one with the matrix $a_{i,i+1} = 1, a_{i,i-1} = -1$.
This matrix has corank $1$ for $n$ odd and corank $2$ for $n$ even.

One Casimir function for $\{\cdot, \cdot\}_{\alpha^{-1}}$ is provided by Lemma \ref{lem:CasimirAlpha}.
For $n$ even, there is an additional Casimir.
Let
\[
c_{odd} = c_1 c_3 \cdots c_{n-1}, \quad c_{even} = c_2 c_4 \cdots c_n.
\]
Then we have
\[
\frac{c_{odd}}{c_{even}} \circ \pi_{\alpha} = \frac{x_{odd} (1-x)_{even}}{x_{even} (1-x)_{odd}} = \frac{u_{odd}}{u_{even}}.
\]
It is not hard to see that both $u_{odd}$ and $u_{even}$ are Casimirs of the bracket \eqref{eqn:PoissonU},
thus ${c_{odd}}/{c_{even}}$ is a Casimir of $\{\cdot, \cdot\}_{\alpha^{-1}}$.

Combined with Theorems \ref{thm:Integrals}, \ref{thm:IntIndependence}, \ref{thm:InvBracket}, and \ref{thm:FCommute}, this completes the proof of Main Theorem \ref{thm:TwistedModuli}.

\begin{remark}
{\rm Corollary \ref{cor:indepstr} provides another proof of Theorem \ref{Bianchi} (Bianchi permutability). According to this corollary, the affine structure on the Lagrangian leaves is independent of $\alpha$. Therefore, for all values of $\alpha$, the maps $T_\alpha$ are parallel translations in the same affine coordinates, and hence they commute.
}
\end{remark}

\section{Integrability on the moduli space of closed polygons}
\label{sec:IntClosed}

\subsection{Main Theorem 2}
\label{subsect:mainclosed}

The moduli space of closed ideal $n$-gons is a codimension 3 subspace of the moduli space of twisted ideal $n$-gons, invariant under the relations $\T$. Complete integrability of $\T$ on $\cT$ does not automatically imply integrability of its restriction to $\cP$. In this section we prove the following result. 

\begin{mainthm}
\label{thm:mainclosed}
Almost every point of $\cP$ lies on a ${\lfloor \frac{n-3}{2} \rfloor}$-dimensional  submanifold, invariant under $\T$ and carrying an invariant affine structure. The map $\T$ is a parallel translation with respect to this affine structure.  This foliation on invariant manifolds and the affine structures on its leaves is independent of $\alpha$.
\end{mainthm}

As before, ``almost every" means belonging to a dense open set.

\subsection{Another set of integrals} \label{altsect}

Let us introduce the functions on the space $\widetilde \cP$:
\begin{align}
G_0(P) &= 2,\nonumber\\
G_k(P) &= \sum_{i_1 < \cdots < i_k} \frac{(p_{i_1} - p_{i_k+1})(p_{i_2} - p_{i_1+1}) \cdots (p_{i_k} - p_{i_{k-1}+1})}
{(p_{i_1} - p_{i_1+1})(p_{i_2} - p_{i_2+1}) \cdots (p_{i_k} - p_{i_k+1})},
\label{eqn:Integrals2}
\end{align}
where $k=0,1,\ldots, \lfloor {n}/2 \rfloor$.
In particular, $G_1 = n$ and $G_2 = \sum_{i<j} [p_i, p_j, p_{j+1}, p_{i+1}]$.
Note that for a non-sparse multiindex $(i_1, \ldots, i_k)$ the corresponding summand in \eqref{eqn:Integrals2} vanishes.

\begin{lemma}
The functions $G_k$ are projectively invariant and hence descend on the moduli space $\cP$.
\end{lemma}

\begin{proof}
It is straightforward to check that $G_k$ is invariant under parallel translations $p \mapsto p+c$, dilations $p \mapsto cp$, and the inversion $p \mapsto 1/p$.
\end{proof} 

The next theorem implies that the functions $G_k$ are integrals of the relations $\T$.

\begin{theorem}
\label{thm:IntegralsClosed}
For a closed polygon $\pP$, we have
\[
\Tr A_\lambda(\pP) = \sum_{k=0}^{\lfloor \frac{n}2 \rfloor} G_k(\pP) (\lambda - 1)^k.
\]
The two sets of integrals for closed polygons are related by the equations
\[
F_l = \sqrt{c_{[n]}} \sum_{k=l}^{\lfloor \frac{n}2 \rfloor} (-1)^{k} \binom{k}{l} G_k, \quad G_l = \frac{1}{\sqrt{c_{[n]}}}\sum_{k=l}^{\lfloor \frac{n}2 \rfloor} (-1)^k \binom{k}{l} F_k
\]
for every $0 \le l \le \lfloor \frac{n}2 \rfloor$ and for one of the two choices of the square root value.
\end{theorem}

\begin{proof}
Observe that
\[
A_\lambda(p_i, p_{i+1}) = \Id + \frac{1-\lambda}{p_i - p_{i+1}} B_i,
\]
where
\[
B_i = \begin{pmatrix} p_{i+1} & -p_ip_{i+1}\\ 1 & -p_i \end{pmatrix}
= \begin{pmatrix} p_{i+1}\\ 1 \end{pmatrix} \begin{pmatrix} 1 & -p_i \end{pmatrix}.
\]
Thus we have
\[
A_\lambda(\pP) = \sum_{k=0}^n (1-\lambda)^k \sum_{i_1 < \cdots < i_k} \frac{B_{i_k} \cdots B_{i_1}}{(p_{i_k} - p_{i_k+1}) \cdots (p_{i_1} - p_{i_1+1})}.
\]
The $k=0$ term is $\Id$, therefore $G_0(\pP) = \Tr(\Id) = 2$.
For $k \ge 1$, one computes
\begin{multline*}
B_{i_k} \cdots B_{i_1} = \begin{pmatrix} p_{{i_k}+1}\\ 1 \end{pmatrix} \begin{pmatrix} 1 & -p_{i_k} \end{pmatrix}
\cdots \begin{pmatrix} p_{i_2+1}\\ 1 \end{pmatrix} \begin{pmatrix} 1 & -p_{i_2} \end{pmatrix}
\begin{pmatrix} p_{i_1+1}\\ 1 \end{pmatrix} \begin{pmatrix} 1 & -p_{i_1} \end{pmatrix}\\
= (p_{i_{k-1}+1} - p_{i_k}) \cdots (p_{{i_1}+1} - p_{i_2}) \begin{pmatrix} p_{i_k+1} & -p_{i_1}p_{i_k+1}\\ 1 & -p_{i_1} \end{pmatrix},
\end{multline*}
which implies
\[
\Tr (B_{i_k} \cdots B_{i_1}) = (p_{{i_1}+1} - p_{i_2}) \cdots (p_{i_k+1} - p_{i_1}).
\]
This allows us to compute $\Tr A_\lambda(\pP)$ and leads to the formula stated in the theorem.

Next, we prove the identities relating the two sets of integrals.
Theorem \ref{thm:Integrals} implies that
\[
\frac{1}{\sqrt{c_{[n]}}} \sum_{k=0}^{\lfloor \frac{n}2 \rfloor} (-1)^kF_k \lambda^k = \Tr(A_\lambda) = \sum_{k=0}^{\lfloor \frac{n}2 \rfloor} G_k (\lambda - 1)^k.
\]
Transform the right hand side:
\[
\sum_{k=0}^{\lfloor \frac{n}2 \rfloor} G_k (\lambda - 1)^k = \sum_{k=0}^{\lfloor \frac{n}2 \rfloor} G_k \sum_{l=0}^k \binom{k}{l} \lambda^l (-1)^{k-l}
= \sum_{l=0}^{\lfloor \frac{n}2 \rfloor} \lambda^l \sum_{k=l}^{\lfloor \frac{n}2 \rfloor} (-1)^{k-l} \binom{k}{l} G_k.
\]
This proves the first series of identities.
The second series is proved similarly after the substitution $\mu = 1-\lambda$.
\end{proof}

\subsection{Alternating perimeters}
\label{sec:AltPer}
The integrals $G_k$ can be interpreted in terms of the alternating perimeters of ideal polygons. We refer to \cite{Pen} for this material.

Consider an ideal even-gon. Its alternating perimeter is defined as follows. Choose a decoration of the polygon, that is, a horocycle at every vertex. Define the side length of the polygon as the signed  distance between the intersection points of the respective horocycles with this side; by convention, if the two consecutive horocycles are disjoint then the respective distance is positive. The alternating sum of the side lengths does not depend on the decoration, see Figure \ref{horo}.

\begin{figure}[hbtp]
\centering
\includegraphics[width=5in]{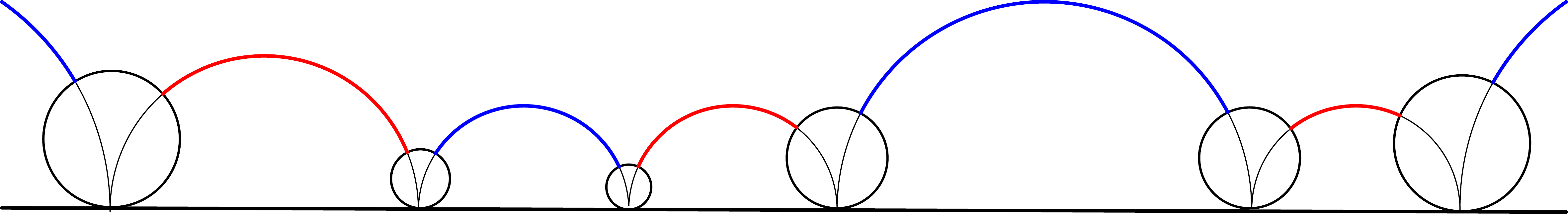}
\caption{Alternating perimeter of an ideal polygon.}
\label{horo}
\end{figure}

Let $\pP=(p_1,\ldots,p_{2m})$ be a decorated ideal $2m$-gon, and let $\delta_{i,j}$ be the above defined length of $p_i p_j$. The 
$\lambda$-length is defined as
$$
\lambda_{i,j} = e^{\delta_{i,j}/2}.
$$
 Set
$$
A(\pP) = \frac{\lambda_{1,2} \lambda_{3,4} \cdots \lambda_{2m-1,2m}}{\lambda_{2,3} \lambda_{4,5} \cdots \lambda_{2m,1}},
$$
 the exponent of the alternating semi-perimeter of $\pP$.
 
 \begin{lemma} \label{altper}
 One has
 $$
G_k(\pP) = \sum_{i_1 < \cdots < i_k} (-1)^k {A}^{-1} (p_{i_1}, p_{i_1+1}, \ldots, p_{i_k}, p_{i_k+1}).
$$
 \end{lemma}

\proof
For a decorated ideal $2k$-gon $\pP$, one has
$$
\frac{\prod_{j=1}^k (p_{2j-1} - p_{2j})}{\prod_{j=1}^k (p_{2j} - p_{2j+1})} = 
\frac{\prod_{j=1}^k \lambda_{2j-1,2j}}{\prod_{j=1}^k \lambda_{2j,2j+1}},
$$
 see \cite{Pen}. This implies the result.
\proofend

For example, if $n=2m$ then $c_{even}/c_{odd} = A^2(\pP)$ and $G_m (\pP) = A(\pP) + 1/A (\pP)$.
In fact, for even-gons, the next result holds.

\begin{lemma} 	\label{thm:AP}
If $\pP \T \pQ$, then $A(\pP) \cdot A(\pQ) =1$.
\end{lemma}

\proof
	We have
$$
	\left(1 - \frac1\alpha\right)^n = \prod_{i=1}^{n} [p_{2i-1}, q_{2i}; p_{2i}, 
	q_{2i-1}]
	= \prod_{i=1}^n 
	\frac{(p_{2i-1}-p_{2i})(q_{2i}-q_{2i-1})}{(p_{2i-1}-q_{2i-1})(q_{2i}-p_{2i})},
$$
and
$$
	\left(1 - \frac1\alpha\right)^n = \prod_{i=1}^{n} [p_{2i}, q_{2i+1}; 
	p_{2i+1}, 
	q_{2i}]
	= \prod_{i=1}^n 
	\frac{(p_{2i}-p_{2i+1})(q_{2i+1}-q_{2i})}{(p_{2i}-q_{2i})(q_{2i+1}-p_{2i+1})}.
$$
	Hence the ratio of the right hand sides of the equations is equal to $1$. 
\proofend

\subsection{Restriction of the integrals to the space of closed polygons}

The moduli space $\cP$ has codimension 3 in $\cT$. After restriction to $\cP$, the integrals $F_k$ and $c_{[n]}$ become functionally dependent. Likewise, their differentials along $\cP$ (considered as 1-forms in $\cT$) also become linearly dependent. The following theorem  makes this precise. 

\begin{theorem}
\label{thm:RestrictIntegrals}
The restrictions of the integrals $F_k$ to $\cP$ satisfy the relations
\begin{equation}
\label{intrel}
 \sum_{k=0}^{\lfloor \frac{n}2 \rfloor} (-1)^k {F_k} =  2 \sqrt{c_{[n]}}, \quad \sum_{k=0}^{\lfloor \frac{n}{2} \rfloor} (-1)^k (n-2k) F_k = 0.
\end{equation}
Also, along $\cP$ one has a linear relation between the differentials of the integrals
\begin{equation}
\label{difintrel}
d\left(\frac{1}{\sqrt{c_{[n]}}} \sum_{k=0}^{\lfloor \frac{n}2 \rfloor} (-1)^k  {F_k} \right)= 0.
\end{equation}
On an open dense subset of $\cP$, these are the only relations between the integrals and their differentials.
\end{theorem}

\begin{proof}
The second of equations \eqref{intrel} was proved in Lemma \ref{lem:LinRelClosed}.
The first one follows from $G_0 = 2$ and expression of $G_0$ in terms of $F_i$ obtained in Theorem \ref{thm:IntegralsClosed}.

Note that Theorem \ref{thm:IntegralsClosed} together with $G_1 = n$ implies
\[
\sum_{k=1}^{\lfloor \frac{n}2 \rfloor} (-1)^k kF_k = n\sqrt{c_{[n]}},
\]
which is a linear combination of equations \eqref{intrel}.

In the proof of Theorem \ref{thm:trmono} we established the following polynomial identity:
\[
\sum_{k=0}^{\lfloor \frac{n}2 \rfloor} (-1)^k  {F_k} = \sum_{I \circlearrowleft \text{sparse}} (-1)^{|I|} c_I = D_{i+1,i+n-1} - c_i D_{i+2,i+n-2}.
\]
It follows that
\begin{multline*}
\frac{\partial}{\partial c_i} \left(\frac{1}{\sqrt{c_{[n]}}} \sum_{k=0}^{\lfloor \frac{n}2 \rfloor} (-1)^k  {F_k} \right)
= -\frac{1}{2c_i\sqrt{c_{[n]}}} \sum_{k=0}^{\lfloor \frac{n}2 \rfloor} (-1)^k  {F_k} - \frac{1}{\sqrt{c_{[n]}}} D_{i+2,i+n-2}\\
= -\frac{1}{2c_i\sqrt{c_{[n]}}} (D_{i+1,i+n-1} + c_i D_{i+2,i+n-2})\\
= -\frac{1}{2c_i\sqrt{c_{[n]}}} (D_{i+1,i+n-1} - D_{i,i+n-2} + D_{i+1,i+n-2}).
\end{multline*}
The cross-ratios of a closed polygon satisfy $D_{i+1,i+n-2} = 0$ and $D_{i+1,i+n-1} = D_{i,i+n-2}$
(see Lemma \ref{lem:DRelations} and Lemma \ref{lem:MoreDRel}).
It follows that the above partial derivative vanishes at $\cP$, which proves equation \eqref{difintrel}.

Next, we show that (\ref{difintrel}) is the only relation between the differentials of the integrals, considered as covectors in $\cT$. It suffices to show that $dF_1,\ldots, dF_{\lfloor \frac{n}2 \rfloor}$ are linearly independent in an open dense set. The argument is similar to the one in the proof of Theorem \ref{thm:IntIndependence}, with the important difference that only $n-3$ of the variables $c_i$ are independent.

Set
\begin{equation}
\label{setc}
c_1=\eps, c_2=\eps^2,\ldots, c_{n-3} = \eps^{n-3}.
\end{equation}

Let us call the greatest power of $\eps$ that divides a polynomial in variables $c_j$ its {\it height}. The height of zero is infinite.

We claim that the heights of the remaining variables, $c_{n-2}, c_{n-1}$, and $c_n$ are, respectively, $1, (n-3)(n-2)/{2}$, and $1$.

Indeed, due to Lemma \ref{lem:DRelations} and formula (\ref{eqn:DRecurrence}), one has
$$
0=D_{1,n-2} = D_{1,n-3} - c_{n-2} D_{1,n-4},  
$$
hence $c_{n-2} = D_{1,n-3}/D_{1,n-4}$. By (\ref{setc}), both the numerator and denominator have height 1, and the claim about $c_{n-2}$ follows.

Likewise, 
$$
0=D_{0,n-3} = D_{1,n-3} -c_n D_{2,n-3}, 
$$
hence $c_n=D_{1,n-3}/D_{2,n-3}$, having height 1 as well.

Finally, 
$$
0=D_{2,n-1} = D_{2,n-2} - c_{n-1} D_{2,n-3},
$$
hence $c_{n-1} = D_{2,n-2}/D_{2,n-3}$. One has
$$
D_{2,n-2} = D_{2,n-3} - c_{n-2} D_{2,n-4} = \frac{D_{2,n-3} D_{1,n-4}-D_{1,n-3} D_{2,n-4}}{D_{1,n-4}}.
$$
The denominator has height 1; as to the numerator, Lemma \ref{lem:MonMat1} with $n$ replaced by $n-3$, along with Lemma \ref{lem:MonMat2}, imply that the numerator equals $c_1c_2\ldots c_{n-3}$, and this implies the claim about $c_{n-1}$. 

For the last step, alternatively, one can use the identity satisfied by continuants (see \cite{Muir}, item 554):
\begin{equation} \label{contident}
D_{1,m} D_{i,j} - D_{1,j} D_{i,m} = -c_{i-1} \ldots c_{j+1} D_{1,i-2} D_{j+2,m}.
\end{equation}

Next we consider the ${\lfloor \frac{n}{2} \rfloor} \times n$ matrix whose columns are the gradients of the integrals $F_i$ and whose rows are their first partial derivatives with respect to the variables $c_j$. This matrix is evaluated along $\cP$.

The cases of even and odd $n$ differ slightly. Let us consider the case of $n=2k$ in detail.

Reduce the matrix to a square $k \times k$ one by retaining the  rows $1,2,4,6,\ldots, n-2$. As before, we start with the last column and go leftward. In each consecutive column, we identify the terms with the smallest height, divide by the respective power of
$\eps$, and take limit as $\eps\to 0$. 

The entries of the last, $k$th, column are $\partial (c_{even})/\partial c_j$ with $j \in \{2,4,6,\ldots, n-2\}$ and $\partial (c_{odd})/\partial c_1$, and the smallest height is attained for $j=n-4$. Thus the last column will have 1 as $(n-4)$th entry and the rest zeros. Subtracting this column from the columns left of it, we may ignore their $(n-4)$th entries.  

Similarly, consider the next, $(k-1)$st, column. Ignoring its $(n-4)$th entry, it has 1 as $(n-6)$th entry and the rest zeros. Again,  subtracting the two rightmost columns from the columns left of them, we may ignore their $(n-4)$th and $(n-6)$th entries. Continuing in this way, we reach 3rd column that has 1 at its second position. 

It remains to consider the top, 1st, and bottom, $(n-2)$nd, positions of the first and the second columns. That is, we need to show that the $2\times 2$ matrix made of these entries is non-degenerate.

Since $F_1 = \sum c_j$, the first column  has ones as every entry. The top position of the second column is $\sum c_j - c_1-c_2-c_n$, and the bottom position of the second column is $\sum c_j - c_{n-3} - c_{n-2} - c_{n-1}$. Thus we need to show that their difference, $c_{n-3} + c_{n-2} + c_{n-1} - c_1-c_2-c_n$, has a finite height.

The terms $c_{n-3}$ and $c_{n-1}$ have height not less than $n-3$. We shall see that the height of $c_{n-2}-c_{n}-c_1$ equals 3, and this will imply that the height of $c_{n-3} + c_{n-2} + c_{n-1} - c_1-c_2-c_n$ equals that of $c_2$, that is, 2.

One has
\begin{equation*}
\begin{split}
c_{n-2}-c_{n} = \frac{D_{1,n-3}}{D_{1,n-4}} - \frac{D_{1,n-3}}{D_{2,n-3}} &= \frac{D_{1,n-3} (D_{2,n-3} - D_{1,n-4})}{D_{1,n-4}D_{2,n-3}}\\
&= \frac{D_{1,n-3} (c_1 D_{3,n-4} - c_{n-3} D_{2,n-5})}{D_{1,n-4}D_{2,n-3}},
\end{split}
\end{equation*}
where the last equality makes use of (\ref{eqn:DRecurrence}).

Since the height of $c_{n-3}$ is $n-3$, we ignore this term and continue: 
\begin{equation*}
\begin{split}
c_{n-2}-c_{n}-c_1&= c_1 \frac{D_{1,n-3} D_{3,n-4} - D_{1,n-4}D_{2,n-3}}{D_{1,n-4}D_{2,n-3}}\\
&= c_1 \frac{(D_{1,n-3} D_{3,n-4} - D_{1,n-4}D_{3,n-3}) + c_2 D_{1,n-4} D_{4,n-3}}{D_{1,n-4}D_{2,n-3}},
\end{split}
\end{equation*}
where the last equality is again due to (\ref{eqn:DRecurrence}).

Due to (\ref{contident}),  the expression in parentheses in the numerator is divisible by $c_2\ldots c_{n-3}$, hence the height of $c_{n-2}-c_{n}-c_1$ equals that of $c_1c_2$, that is, equal to 3. This completes the argument.

If $n$ is odd, one  argues similarly, and we not dwell on it. 

We shall show that (\ref{intrel}) are the only relations between the restrictions of the integrals on $\cP$ as part of the proof of complete integrability in the next section.
\end{proof}

The following lemma provides an alternative derivation of relations (\ref{intrel}) and a proof of (\ref{difintrel}).
Let $M(t) \in \GL(2)$ be a germ of a differentiable curve with $M(1) = \Id$, and let $\Phi(t) = {\Tr M(t)}/{\sqrt{\det M(t)}}$ be the normalized trace.

\begin{lemma}
\label{lm:3rel}
One has
$$
\Phi(1) = 2,\ \Phi' (1) = 0,\ d\Phi (1) =0.
$$
\end{lemma}

\begin{proof}
Let $e^{\lambda_1(t)}$ and $e^{\lambda_2(t)}$ be the eigenvalues of $M(t)$, and let $\mu(t)=\frac{\lambda_1(t)-\lambda_2(t)}{2}$. 
Note that $\mu(1)=0$.

One has $\Phi(t)=e^{\mu(t)} + e^{-\mu(t)}$. It follows that 
$$
\Phi(1) = e^{\mu(1)} + e^{-\mu(1)} =2, \ \Phi'(1) = (e^{\mu(1)} - e^{-\mu(1)}) \mu'(1) =0,
$$
and
$$
 d\Phi (1) = (e^{\mu(1)} - e^{-\mu(1)}) d\mu'(1) =0,
$$
as claimed.
\end{proof}


\begin{proof}[Proof of Theorem \ref{thm:RestrictIntegrals}]
Applied to the monodromy of a twisted polygon, Lemma \ref{lm:3rel}, along with formula (\ref{scaledmono})
imply relations (\ref{intrel}) and (\ref{difintrel}).
\end{proof}

\subsection{Proof of integrability}
\label{sec:IntegrTwist}

The next proposition is an analog of Proposition 3.1 in \cite{OST2}.

\begin{proposition} \label{tang}
Let $H$ be one of the integrals  described in Corollary \ref{cor:FIntegrals}, and let $X$ be 
the Hamiltonian vector field of $H$ with respect to the bracket $\{,\}_\alpha$. Then $X$ is tangent  to the submanifold  $\cP$.
\end{proposition}

\proof
Space $\cT$ is foliated by isomonodromic submanifolds that are generically of codimension one and that are the level hypersurfaces of the normalized trace of the monodromy $M$. Denote the normalized trace by $\Phi(M)$.

The isomonodromic foliation is singular, and $\cP$ is a singular leaf of codimension $3$. 
Note that the versal deformation of $\cP$ is locally isomorphic to $\SL(2)$ partitioned into the conjugacy equivalence classes.

One has $\{H,\Phi(M)\}_\alpha=0$, since the integrals Poisson commute with respect to all brackets $\{,\}_\alpha$ and $\Phi (M)$  is a function of the integrals, see Theorem \ref{thm:trmono}. 
Hence $X$ is tangent to the generic leaves of the isomonodromic foliation on $\cT$. 
Let us show that $X$ is tangent to $\cP$ as well.

In a nutshell,  the tangent space to $\cP$ at a smooth point $x_0$ is the intersection of the limiting positions of the tangent spaces to the isomonodromic leaves at points $x$, as $x$ tends to $x_0$. If  $X$ is transverse to $\cP$ at point $x_0\in \cP$, then $X$ is also transverse to an isomonodromic leaf at some point $x$ close to $x_0$,  which leads to a contradiction.

More precisely,  fix the first three vertices of a twisted polygon by a projective transformation, say, $0,1,\infty$.  This gives a local identification of $\cT$ with  the set of tuples $(p_4,\dots,p_n; M)$. The space of closed $n$-gons is characterized by the condition that $M=\Id$,  so we  locally identify $\cT$ with $\cP \times \SL(2)$. In particular, we have a projection $\cT \to \SL(2)$, and the preimage of the identity is $\cP$. The isomonodromic leaves project to the conjugacy equivalence classes in $\SL(2)$.

Thus the proof of the proposition reduces to the following lemma about the  $\SL(2)$.

\begin{lemma}
\label{limpos}
Consider the singular foliation of $\SL(2)$ by the conjugacy equivalence classes, and let $T_X$ be the tangent space to this foliation at $X\in \SL(2)$. Then the intersection, over all $X$, of the limiting positions of the spaces $T_X$, as $X$ tends to $\Id$, is trivial. 
\end{lemma}

\proof Let $B\in\SL(2)$, and let $B+\varepsilon C$ be an
infinitesimal deformation within the conjugacy equivalence class. Then
$$
{\Tr} (B+\eps C) = {\Tr} B,\ 1=\det (B+\eps C) = \det B \det(\Id+ \eps B^{-1} C) = 1+ \eps {\Tr} (B^{-1} C),
$$
hence 
${\Tr} C = 0$ and ${\Tr} (B^{-1} C) =0.
$

Let  $B=\Id+\eps A$ be a point in an infinitesimal neighborhood of the identity $\Id$; then ${\Tr} A=0$. 
Our conditions on $C$ imply that ${\Tr} (C)={\Tr} (AC)=0$. 
Since ${\Tr} (AC)$ is a non-degenerate quadratic form, 
a matrix $C\in \sl(2)$, satisfying ${\Tr} (AC)=0$ for all $A\in \sl(2)$, must be zero.
\proofend

The proposition follows.
\proofend

{\it Proof of Main Theorem \ref{thm:mainclosed}}. Fix $\alpha$ and consider the Hamiltonian vector fields with respect to $\{,\}_\alpha$.

Consider the case of odd $n$ and set $n=2k+1$. 
The space of integrals of $\T$ on $\cT$ is $k+1$-dimensional.
Since $\{,\}_\alpha$ has one Casimir function (Main Theorem \ref{thm:TwistedModuli}) and that, at a generic point $x\in \cP$, there is exactly one relation between the differentials of the integrals considered as covectors in $\cT$  (Theorem \ref{thm:RestrictIntegrals}), it follows that the vector space generated by the Hamiltonian vectors fields of the integrals has dimension $k-1$.
 One obtains a $k-1$-dimensional foliation on $\cP$ whose leaves carry an affine structure, given by these commuting Hamiltonian vector fields. 

The leaves are level surfaces of the restrictions of the integrals on $\cP$, hence the relations (\ref{intrel})  are generically the only relations between the integrals. The leaves and the affine structures therein are invariant under $\T$, hence the restriction of the map $\T$ to a leaf is a parallel translation. 

The case of even $n$ is similar. Let $n=2k$. The space of integrals of $\T$ on $\cT$ is again $k+1$-dimensional, but one has two Casimirs. Therefore, at a generic point $x\in \cP$, the vector space generated by the Hamiltonian vectors fields of the integrals has dimension $k-2$. The rest of the argument repeats that in the case when $n$ is odd.

\subsection{$\cP$ as a Poisson submanifold of $\cT$}

Consider the Poisson structure 
\begin{equation}
\label{eqn:PoissonC2}
\begin{split}
\{c_i, c_{i+1}\} &= c_ic_{i+1}(1 - c_i - c_{i+1}),\\
\{c_i, c_{i+2}\} &= -c_ic_{i+1}c_{i+2}
\end{split}
\end{equation}
on  the space of twisted polygons $\cT$. This is the bracket $\{\cdot, \cdot\}_1 - \{\cdot, \cdot\}_2$ from \eqref{eqn:BiHamiltonian}.

\begin{proposition}
The space of closed polygons $\cP$ is a Poisson submanifold of the space of twisted polygons $\cT$
with respect to the  bracket (\ref{eqn:PoissonC2}).
\end{proposition}

\begin{proof}
We have to show that for every function $H$ on $\cT$ its Hamiltonian vector field $X_H$ is tangent to $\cP$.
It suffices to show this for the Hamiltonian vector fields $X_i$ of the coordinate functions $c_i$.
Due to the cyclic symmetry it suffices to consider the case $i=1$. We have
\[
X_1 = c_1c_2(1-c_1-c_2) \frac{\partial}{\partial c_2} - c_1c_2c_3 \frac{\partial}{\partial c_3} + c_{n-1}c_nc_1 \frac{\partial}{\partial c_{n-1}}
- c_nc_1(1-c_n-c_1) \frac{\partial}{\partial c_n}.
\]

Lemma \ref{lem:MonMat1} and equation \eqref{eqn:TraceMon} imply that $\cP \subset \cT$ is the solution set of the system
\[
D_{i,n-3+i} = 0, \quad D_{i+1,n-2+i} = 0, \quad \Tr M = 0
\]
for any $i$.
Besides, at a generic point of $\cP$ the normal space to $\cP$ is spanned by the gradients of the functions $D_{i,n-3+i}$, $D_{i+1,n-2+i}$, $\Tr M$.
Thus it suffices to show that the functions $\Tr M$, $D_{2,n-1}$, and $D_{3,n}$ have zero derivatives in the direction of $X_1$.

By Lemma \ref{lem:CasimirAlpha}, the function $\frac{1}{\sqrt{c_{[n]}}} \Tr M$ is a Casimir of our Poisson bracket.
Since $\Tr M$ vanishes on $\cP$, it follows that $X_1 (\Tr M) = 0$ on $\cP$.

Let us compute $X_1(D_{2,n-1})$.
One has
\[
\frac{\partial D_{2,n-1}}{\partial c_2} = -D_{4,n-1}, \quad \frac{\partial D_{2,n-1}}{\partial c_3} = -D_{5,n-1}, \quad
\frac{\partial D_{2,n-1}}{\partial c_{n-1}} = -D_{2,n-3}, \quad \frac{\partial D_{2,n-1}}{\partial c_n} = 0.
\]
By substituting this into the formula for $X_1(\cdot)$, one obtains
\[
X_1(D_{2,n-1}) = c_1(-c_2(1-c_1-c_2)D_{4,n-1} + c_2c_3 D_{5,n-1} - c_{n-1}c_n D_{2,n-3}).
\]
Using the recurrence \eqref{eqn:DRecurrence} and the vanishing of $D_{i,n-3+i}$ on $\cP$,
one shows that the expression on the right hand side vanishes on $\cP$.

The computation of $X_1(D_{3,n})$ is similar.
\end{proof}

Let $U = (\C \setminus \{0,-1\})^{n-3}$ with coordinates $u_2, \ldots, u_{n-2}$.
Consider the following Poisson bracket on $U$:
\begin{equation}
\label{eqn:PoissonU2}
\{u_i, u_{i+1}\} = u_i u_{i+1}, \quad i = 2, \ldots, n-3
\end{equation}
(note that $\{u_2, u_{n-2}\} = 0$ unless $n=5$).

\begin{theorem}
\label{thm:PoissonEmbedding}
The map $\rho \colon U \to \cT$ defined as
\begin{equation}
\label{eqn:CFromU}
\begin{split}
c_1 &= \frac{u_2}{1+u_2},\\
c_i &= \frac{u_{i+1}}{(1+u_i)(1+u_{i+1})}, \quad i = 2, \ldots, n-3,\\
c_{n-2} &= \frac{1}{1+u_{n-2}},\\
c_{n-1} &= \frac{u_{[2,n-2]}}{1 + u_2 + u_2u_3 + \cdots + u_{[2,n-2]}},\\
c_n &= \frac{1}{1 + u_2 + u_2u_3 + \cdots + u_{[2,n-2]}}
\end{split}
\end{equation}
is a Poisson map with respect to the brackets \eqref{eqn:PoissonU2} and \eqref{eqn:PoissonC2}, and its image is an open dense subset of $\cP$.
\end{theorem}
\begin{proof}
Let us prove that $\rho(U)$ is an open dense subset of $\cP$.
The system of the first $n-3$ equations in \eqref{eqn:CFromU} can be solved for $u_2, \ldots, u_{n-2}$.
It follows that the composition of $\rho$ with the projection to the first $n-3$ coordinates $c_1, \ldots, c_{n-3}$ has a dense image.
On the other hand, the first $n-3$ coordinates determine a point in $\cP \subset \cT$ uniquely:
\begin{align*}
0 = D_{1,n-2} = D_{1,n-3} - c_{n-2}D_{1,n-4} &\Rightarrow c_{n-2} = \frac{D_{1,n-3}}{D_{1,n-4}},\\
0 = D_{0,n-3} = D_{1,n-3} - c_0 D_{2,n-3} &\Rightarrow c_n = c_0 = \frac{D_{1,n-3}}{D_{2,n-3}},\\
0 = D_{2,n-1} = (D_{2,n-3} - c_{n-2}D_{2,n-4}) - c_{n-1}D_{2,n-3} &\Rightarrow c_{n-1} = 1 - c_{n-2} \frac{D_{2,n-4}}{D_{2,n-3}}.
\end{align*}
Therefore it suffices to check that the last three equations in \eqref{eqn:CFromU} arise from the substitution of the first $n-3$ equations
into the above formulas.
For this we can use the first formula of Lemma \ref{lem:DijU} by formally setting $u_1 = 0$ (which brings the first equation in \eqref{eqn:CFromU}
into the form of the subsequent $n-4$ equations).
It follows that
\[
D_{1,j} = \frac{1}{(1+u_2) \cdots (1+u_{j+1})},
\]
while $D_{2,j}$ is given by the formula from Lemma \ref{lem:DijU}.
By substituting this into the formulas which express the last three coordinates in $\cP$ by means of the first $n-3$
we obtain the last three equations of \eqref{eqn:CFromU}.

The fact that the map $\rho$ is Poisson is proved by a tedious computation.
The computation can be facilitated by the following formulas for partial derivatives.
\begin{align*}
\frac{\partial c_i}{\partial u_i} &= -\frac{c_i}{1+u_i}, \quad i \ne 1,\\
\frac{\partial c_i}{\partial u_{i+1}} &= \frac{c_i}{u_{i+1}(1+u_{i+1})}, \quad i \ne n-2,\\
\frac{\partial c_{n-1}}{\partial u_i} &= \frac{c_{n-1}}{u_i} \frac{u_{\langle 2, i-1 \rangle}}{u_{\langle 2, n-2 \rangle}},\\
\frac{\partial c_n}{\partial u_i} &= \frac{c_n}{u_i} (u_{\langle 2, i-1 \rangle} c_n - 1).
\end{align*}
Here $u_{\langle i, j \rangle} = \sum_{k=i-1}^j u_{[k,j]}$.
\end{proof}

\begin{remark}
\label{rem:UCR}
The map $\rho$ can be viewed as an extension of the map $\pi \colon X \to C$ to a codimension $3$ space $u_1 = 0$, $u_{n-1} = \infty$, $u_n = 1$ lying in its closure.

Similarly to Lemma \ref{lem:CoordInfFiber} one computes the coordinates of a point in $\pi^{-1}(c)$ for $c \in \cP$ as
\[
u_i = -[p_{i+1}, p_{i-1}, p_i, z],
\]
where $\pP$ is any polygon with cross-ratios $c$, and $z$ is any point different from $p_i$ for all $i$.
The extension to $u_1 = 0$, $u_{n-1} = \infty$, $u_n = 1$ corresponds to setting $z = p_n$.
\end{remark}

\begin{corollary}
For $n$ odd, the Poisson bracket \eqref{eqn:PoissonC2} induces a symplectic structure on $\cP$.
\end{corollary}
\begin{proof}
Indeed, the Poisson bracket \eqref{eqn:PoissonU2} is non-degenerate for $n$ odd.
\end{proof}

In the next section we show that this structure coincides with the one coming from the theory of cluster algebras.

\subsection{Odd-gons: a symplectic structure}
\label{sec:cluster}

As we saw in Section \ref{fripat}, if $n$ is odd, the moduli space $\cP$ is identified with the space of frieze patterns of width $n-3$. 

Let $x_1,x_2,\dots,x_{n-3}$ be the diagonal frieze coordinates, extended by setting $x_{-1} = x_{n-1} = 0, x_0 = x_{n-2} = 1$. The space of friezes is a cluster variety, and one has the symplectic form 
\begin{equation} 
\label{sympfr}
\omega = \sum_{i=1}^{n-4} \frac{d x_i \wedge dx_{i+1}}{x_i x_{i+1}},
\end{equation}
known in the theory of cluster algebras, see \cite{GSV,MOT,Mor}.
The respective Poisson bracket is given by the following lemma whose proof is a direct calculation and is omitted.

\begin{lemma} \label{Pbrack}
For $1\le i<j\le n-3$, the  non-zero value of bracket is when $i$ is odd and $j$ is even:
$$
\{x_i,x_j\} =-x_ix_j
$$
(as always, one also has the opposite sign if $x_i$ and $x_j$ are swapped).
\end{lemma}

\begin{theorem} \label{Cluster}
The symplectic structure (\ref{sympfr}) on $\cP$ coincides with the one induced by the Poisson bracket \eqref{eqn:PoissonC2}.
\end{theorem}

\begin{proof}
The diagonal frieze coordinates and the coordinates $u_i$ introduced in \eqref{eqn:CFromU} are related by the equations
\begin{equation}
\label{eqn:UtoX}
u_i = \frac{x_{i-2}}{x_i}, \quad i = 2, \ldots, n-2.
\end{equation}
This follows from the formula indicated in the Remark \ref{rem:UCR}:
\[
u_i = -[p_{i+1}, p_{i-1}, p_i, p_n] = -\frac{[V_{i+1},V_i] [V_{i-1},V_n]}{[V_{i+1},V_n] [V_{i-1},V_i]} = \frac{[V_{i-1},V_n]}{[V_{i+1},V_n]} = \frac{x_{i-2}}{x_i}
\]
Alternatively, one can check this using the known expressions of the cross-ratio coordinates $c_i$
in terms of the diagonal frieze coordinates $x_i$, see \cite{Mor}:
\begin{equation} \label{afromx}
a_i = \frac{x_{i-2}+x_i}{x_{i-1}},\ {\rm for}\ i=1,\ldots,n-1,\ {\rm and}\ \ a_n = \sum_{k=0}^{n-3} \frac{1}{x_kx_{k+1}},
\end{equation}
and $c_i = \frac{1}{a_ia_{i+1}}$, see Lemma \ref{atoc}.

Thus it suffices to show that the map \eqref{eqn:UtoX} is a Poisson map.
This is done by a direct computation.
\end{proof}

\begin{remark}
In the coordinates $a_i$, the Poisson bracket \eqref{eqn:PoissonC2} on $\cP$ is as follows.
\begin{equation*}
\begin{split}
&\{a_i,a_j\}_1 = (-1)^{j-i} a_i a_j\ \ {\rm for}\  1\leq i<j \leq n, i+2\leq j, (i,j) \neq (1,n),\\
 &\{a_i,a_{i+1}\}_1 = 1-a_i a_{i+1},\ {\rm and} \ \{a_1,a_n\}_1 =   a_1 a_n -1.
\end{split}
\end{equation*}
\end{remark}



\subsection{Odd-gons: infinitesimal map} 
\label{infinitsection}
If $n$ is odd, the map $\T$ with infinitesimal $\alpha$ is interpreted as a vector field on the space of polygons $\widetilde{\cP}$.
In this section we study this vector field and its projection to the moduli space $\cP$. 

\begin{theorem}
\label{thm:InfMap}
Let $n$ be odd. Consider the vector field
\[
\widetilde{\xi} = \sum_{i=1}^n \frac{p_i^{(1)} p_{i+2}^{(1)} \cdots p_{i+n-1}^{(1)}}{p_{i+1}^{(1)} p_{i+3}^{(1)} \cdots p_{i+n-2}^{(1)}} \frac{\partial}{\partial p_i}
\]
on $\widetilde{\cP}$.
(As before, $p_j^{(1)} = p_j - p_{j+1}$ and the indices are taken modulo $n$.)
Then the following holds:
\begin{enumerate}
\item
If $\pP(t)$ is a smooth curve in $\widetilde{\cP}$ tangent to $\widetilde{\xi}$ at $t=0$, then one has
\[
[p_i(0), p_{i+1}(0), p_i(t), p_{i+1}(t)] = -t^2 + o(t^2).
\]
\item
The field $\widetilde{\xi}$ is $\PSL(2)$-invariant, and its projection to $\cP$ is given by
\[
\xi = \frac{1}{\sqrt{c_{[n]}}} \sum_{i=1}^n (c_i c_{i+2} \cdots c_{i+n-1} - c_{i+1}c_{i+3} \cdots c_{i+n}) \frac{\partial}{\partial c_i}.
\]
\item
The vector field $\xi$ on $\cP$ is the Hamiltonian vector field of the function
\[
E_{\lfloor \frac{n}2 \rfloor} = \frac{F_{\lfloor \frac{n}2 \rfloor}}{\sqrt{c_{[n]}}}
\]
with respect to the Poisson structure \eqref{eqn:PoissonC2} on $\cP$.
\end{enumerate}
\end{theorem}
\begin{proof}
Let $p_i(t) = p_i(0) + t v_i + o(t)$.
Then we have
\[
[p_i(0), p_{i+1}(0), p_i(t), p_{i+1}(t)] = -t^2 \frac{v_i v_{i+1}}{(p_i - p_{i+1})^2} + o(t^2).
\]
The system of equations $v_i v_{i+1} = (p_i - p_{i+1})^2$ has a unique solution for odd $n$,
given by the components of the vector field $\widetilde{\xi}$ above.
This proves the first part of the theorem.

We note  that for even $n$  a solution exists if and only if the alternating perimeter of $\pP$ is zero,
and one has a one-parameter family of solutions in this case.

Take a path $p_i(t)$ such that $\dot p_i(0) = v_i$ for all $i$ (with the components of $\widetilde\xi$ as $v_i$),
project it to $\cP$, and compute $\dot c_i(0)$:
\[
\dot c_i = \frac{\partial c_i}{\partial p_{i-1}} v_{i-1} + \frac{\partial c_i}{\partial p_i} v_i
+ \frac{\partial c_i}{\partial p_{i+1}} v_{i+1} + \frac{\partial c_i}{\partial p_{i+2}} v_{i+2}.
\]
The partial derivatives of the cross-ratio $c_i = [p_i, p_{i+1}, p_{i-1}, p_{i+2}]$ are as follows:
\begin{align*}
\frac{\partial c_i}{\partial p_{i-1}} &= c_i \frac{p_i^{(1)}}{p_{i-1}^{(1)}p_{i-1}^{(2)}}, &
\frac{\partial c_i}{\partial p_i} &= -c_i \frac{p_{i-1}^{(3)}}{p_{i-1}^{(1)}p_i^{(2)}},\\
\frac{\partial c_i}{\partial p_{i+1}} &= c_i \frac{p_{i-1}^{(3)}}{p_{i+1}^{(1)}p_{i-1}^{(2)}}, &
\frac{\partial c_i}{\partial p_{i+2}} &= -c_i \frac{p_i^{(1)}}{p_{i+1}^{(1)}p_i^{(2)}}.
\end{align*}
Besides, we have
\[
v_{i-1} = \left(\frac{p_{i-1}^{(1)}}{p_i^{(1)}}\right)^2 v_{i+1}, \quad v_{i+2} = \left( \frac{p_{i+1}^{(1)}}{p_i^{(1)}} \right)^2 v_i.
\]
It follows that
\begin{multline*}
\dot c_i = c_i \left( \frac{p_i^{(1)}}{p_{i-1}^{(1)}p_{i-1}^{(2)}} v_{i-1} - \frac{p_{i-1}^{(3)}}{p_{i-1}^{(1)}p_i^{(2)}} v_i
 + \frac{p_{i-1}^{(3)}}{p_{i+1}^{(1)}p_{i-1}^{(2)}} v_{i+1} - \frac{p_i^{(1)}}{p_{i+1}^{(1)}p_i^{(2)}} v_{i+2}\right)\\
= c_i \left( \left( \frac{p_{i-1}^{(3)}}{p_{i+1}^{(1)}p_{i-1}^{(2)}} + \frac{p_{i-1}^{(1)}}{p_i^{(1)}p_{i-1}^{(2)}} \right)v_{i+1}
 - \left( \frac{p_{i+1}^{(1)}}{p_i^{(1)}p_i^{(2)}} + \frac{p_{i-1}^{(3)}}{p_{i-1}^{(1)}p_i^{(2)}} \right)v_i \right)\\
= c_i \frac{p_{i-1}^{(3)}p_i^{(1)} + p_{i-1}^{(1)}p_{i+1}^{(1)}}{p_{i}^{(1)}}
\left( \frac{v_{i+1}}{p_{i+1}^{(1)}p_{i-1}^{(2)}} - \frac{v_i}{p_{i-1}^{(1)}p_i^{(2)}} \right).
\end{multline*}
One shows the following identity by a simple computation:
\[
p_{i-1}^{(3)}p_i^{(1)} + p_{i-1}^{(1)}p_{i+1}^{(1)} = p_{i-1}^{(2)}p_i^{(2)},
\]
and by substituting it into the formula above obtains
\[
\dot c_i = c_i \left( \frac{p_i^{(2)}}{p_i^{(1)}p_{i+1}^{(1)}} v_{i+1} - \frac{p_{i-1}^{(2)}}{p_{i-1}^{(1)}p_i^{(1)}} v_i \right).
\]
On the other hand, due to $c_i = \frac{p_{i-1}^{(1)}p_{i+1}^{(1)}}{p_{i-1}^{(2)}p_i^{(2)}}$, we have $\sqrt{c_{[n]}} = \frac{p_{[n]}^{(1)}}{p_{[n]}^{(2)}}$ and
\begin{multline*}
c_{i+1} c_{i+3} \cdots c_{i+n-2} = \frac{p_i^{(1)} (p_{i+2}^{(1)} \cdots p_{i+n-3}^{(1)})^2 p_{i+n-1}^{(1)}}{p_i^{(2)}p_{i+1}^{(2)} \cdots p_{i+n-2}^{(2)}}\\
= \frac{p_{[n]}^{(1)}}{p_{[n]}^{(2)}} \cdot p_{i-1}^{(2)} \cdot
\frac{p_{i+2}^{(1)} p_{i+4}^{(1)} \cdots p_{i+n-3}^{(1)}}{p_{i+1}^{(1)} p_{i+3}^{(1)} \cdots p_{i+n-2}^{(1)}}
= \sqrt{c_{[n]}} \frac{p_{i-1}^{(2)}}{p_{i-1}^{(1)}p_i^{(1)}} v_i.
\end{multline*}
This implies
\[
\dot c_i = \frac{c_i}{\sqrt{c_{[n]}}} (c_{i+2}c_{i+4} \cdots c_{i+n-1} - c_{i+1} c_{i+3} \cdots c_{i+n-2}),
\]
which proves the second part of the theorem.

From the proof of Theorem \ref{thm:FCommute} we know that
\[
\{E_{\lfloor \frac{n}{2} \rfloor}, c_i\}_2 = -\{E_{\lfloor \frac{n}{2} \rfloor + 1}, c_i\}_1 = 0.
\]
Thus we have
\[
\{E_{\lfloor \frac{n}{2} \rfloor}, c_i\} = \{E_{\lfloor \frac{n}{2} \rfloor}, c_i\}_1 =
\frac{1}{\sqrt{c_{[n]}}} \{F_{\lfloor \frac{n}{2} \rfloor}, c_i\}_1,
\]
because $c_{[n]}$ is a Casimir of the bracket $\{\cdot, \cdot\}_1$.
One computes
\[
\{F_{\lfloor \frac{n}{2} \rfloor}, c_i\}_1 = \frac{\partial F_{\lfloor \frac{n}{2} \rfloor}}{\partial c_{i-1}} \{c_{i-1}, c_i\}_1
+ \frac{\partial F_{\lfloor \frac{n}{2} \rfloor}}{\partial c_{i+1}} \{c_{i+1}, c_i\}_1
= c_i\left( c_{i-1}\frac{\partial F_{\lfloor \frac{n}{2} \rfloor}}{\partial c_{i-1}}
- c_{i+1}\frac{\partial F_{\lfloor \frac{n}{2} \rfloor}}{\partial c_{i+1}} \right).
\]
By splitting the sum $F_{\lfloor \frac{n}{2} \rfloor}$ into terms containing $c_j$ and those not containing $c_j$, one easily shows that
\[
c_j\frac{\partial F_{\lfloor \frac{n}{2} \rfloor}}{\partial c_j} = F_{\lfloor \frac{n}{2} \rfloor} - F_{\lfloor \frac{n}{2} \rfloor}^{\setminus j},
\]
where $F_k^{\setminus j}$ denotes the sum of all cyclically sparse monomials without factor $c_j$.
It follows that
\[
\{F_{\lfloor \frac{n}{2} \rfloor}, c_i\}_1
= c_i \left( F_{\lfloor \frac{n}{2} \rfloor}^{\setminus (i+1)} - F_{\lfloor \frac{n}{2} \rfloor}^{\setminus (i-1)} \right).
\]
The sums $F_{\lfloor \frac{n}{2} \rfloor}^{\setminus (i-1)}$ and $F_{\lfloor \frac{n}{2} \rfloor}^{\setminus (i+1)}$
have many common summands, and their difference is
\[
F_{\lfloor \frac{n}{2} \rfloor}^{\setminus (i+1)} - F_{\lfloor \frac{n}{2} \rfloor}^{\setminus (i-1)}
= c_{i+2} c_{i+4} \cdots c_{i+n-1} - c_{i+1} c_{i+3} \cdots c_{i+n-2}.
\]
Putting all things together we obtain
\begin{multline*}
\{E_{\lfloor \frac{n}{2} \rfloor}, c_i\} = \frac{1}{\sqrt{c_{[n]}}} \{F_{\lfloor \frac{n}{2} \rfloor}, c_i\}_1\\
= \frac{c_i}{\sqrt{c_{[n]}}} (c_{i+2}c_{i+4} \cdots c_{i+n-1} - c_{i+1} c_{i+3} \cdots c_{i+n-2}) = \dot c_i,
\end{multline*}
and the theorem is proved.
\end{proof}

\begin{remark}
In the $a$-coordinates (see Section \ref{fripat}) one has
$$
\xi = \sum_{i=1}^n a_i (a_{i+1} - a_{i+2} + \ldots - a_{n+i-1}) \frac{\partial}{\partial a_i}.
$$
This coincides with the dressing chain of Veselov-Shabat, see formula (12) in \cite{SV}.
\end{remark}

\begin{remark}
A  continuous limit $n\to\infty$ of the moduli space $\cP$ is the moduli space of projective equivalence classes of diffeomorphisms $\R/\pi\Z \to \RP^1$.
In this limit, the relations $\T$ can be interpreted as B\"acklund transformations of the KdV equation, see \cite{Tab1}.
\end{remark}

\section{(Pre)symplectic form and two additional integrals} \label{twoaddsection}

\subsection{(Pre)symplectic form on the space of closed polygons} \label{symplsect}

Define differential 1- and 2-forms on the space of ideal $n$-gons $\widetilde\cP$:
$$
\lambda=\frac{1}{2} \sum_{i=1}^n \frac{dp_i+dp_{i+1}}{p_{i+1}-p_i}, \ \ 
\Omega = d\lambda = \sum_{i=1}^n \frac{dp_i \wedge dp_{i+1}}{(p_i-p_{i+1})^2},
$$
where, as usual, the indices are understood cyclically mod $n$.

\begin{theorem} \label{2form}
1. The 2-form $\Omega$ is invariant under the maps $\T$. Furthermore, $\T$ is an exact (pre)symplectic correspondence: it changes the 1-form $\lambda$ by a differential of a function. \\
2. The forms $\lambda$ and $\Omega$ are $SL(2)$-invariant, but they are not basic: they do not descend to $\cP$.
\end{theorem}

\proof
Let $\pQ \T \pP$. A direct calculation shows that this is equivalent to each of the equalities
\begin{equation} \label{oneother}
\frac{\alpha}{q_i-p_i} -\frac{1}{q_{i+1}-p_i}=\frac{\alpha-1}{p_{i+1}-p_i},
\end{equation}
\begin{equation} \label{twoother}
\frac{\alpha}{q_{i+1}-p_{i+1}} -\frac{1}{q_{i+1}-p_i}=\frac{\alpha-1}{q_{i+1}-q_i}.
\end{equation}

Multiply (\ref{oneother}) by $dp_i$, multiply (\ref{twoother}) by $dq_{i+1}$, subtract the second equation from the first, and sum up over $i$ to obtain
\begin{equation} \label{interm}
\begin{split}
\alpha \sum_{i=1}^n  \left( \frac{dp_i}{q_i-p_i} - \frac{dq_{i+1}}{q_{i+1}-p_{i+1}} \right) + \sum_{i=1}^n \left( \frac{dq_{i+1}-dp_i}{q_{i+1}-p_i} \right)& \\
=(\alpha-1) \sum_{i=1}^n  \left( \frac{dp_i}{p_{i+1}-p_i} - \frac{dq_{i+1}}{q_{i+1}-q_i} \right)&.   
\end{split}
\end{equation}

The first sum on the left of (\ref{interm}) equals $-\sum d\ln (q_i-p_i)$, and the second sum equals $\sum d\ln(q_{i+1}-p_i)$, that is, the left hand side of (\ref{interm}) is a differential of a function. The sum on the right equals 
$$
\lambda(\pP) - \lambda(\pQ) - \frac{1}{2} \sum d\ln(q_{i+1}-p_i) - \frac{1}{2} \sum d\ln(q_{i+1}-q_i).
$$
 Assuming that $\alpha \neq 1$, we conclude that $\lambda(\pP) - \lambda(\pQ)$ is a differential of a function, as needed.
 
 To check the M\"obius invariance of $\lambda$  (and hence of $\Omega$), it suffices to consider the three transformations, translation, dilation, and inversion:
$$
p_i \mapsto p_i + c,\ p_i \mapsto c p_i,\ p_i \mapsto \frac{1}{p_i}.
$$
In the first two cases,  $\lambda$ clearly remains intact. 

Let $\psi (p) = 1/p$, and denote by $\Psi$ the diagonal action of $\psi$ on polygons. Then
\begin{equation*}
\begin{split}
&\Psi^*(2\lambda) - 2\lambda = \sum \frac{-\frac{dp_i}{p_i^2}-\frac{dp_{i+1}}{p_{i+1}^2}}{\frac{1}{p_{i+1}}-\frac{1}{p_i}} - \sum \frac{dp_i+dp_{i+1}}{p_{i+1}-p_i}\\
= &\sum \frac{\left(\frac{p_{i+1}}{p_i}-1\right) dp_i + \left(\frac{p_{i}}{p_{i+1}}-1\right) dp_{i+1}}{p_{i+1}-p_i}
 = \sum \frac{dp_i}{p_i} - \sum \frac{dp_{i+1}}{p_{i+1}} = 0,
\end{split}
\end{equation*}
as needed.

The forms $\lambda$ and $\Omega$ are not basic because they are not annulated by the vertical vector fields, the fields that are tangent to the fibers of the projection $\widetilde\cP \to \cP$, that is, by the  generators of the Lie algebra $\sl(2)$ (see proof of Theorem \ref{upint} below for the explicit formulas). 
\proofend

\subsubsection{Even $n$, Poisson structure} \label{evenn}
If $n=2k$, then the 2-form $\Omega$ is generically (that is, in an open dense set) non-degenerate. Let us compute the respective Poisson structure on the space of ideal $n$-gons $\widetilde\cP$. 

Let $f(p_1,\ldots,p_n)$ be a function, and let ${\bf v}=\sum v_i \partial/\partial p_i$ be its Hamiltonian vector field given by $i_{\bf v} \Omega = df$. Set
$$
O=\left(p_1^{(1)} p_3^{(1)} \cdots p_{2k-1}^{(1)})\right)^2, E= \left(p_2^{(1)} p_4^{(1)} \cdots p_{2k}^{(1)} \right)^2,
$$
where $p_i^{(1)} = p_i - p_{i+1}$.
The  cyclic permutation of the indices interchanges $O$ and $E$, and one has $O/E= A(\pP)^2.$

Let $1\le i,j \le 2k$ have opposite parity.  Define
$$
K(i,j) = \left(p_i^{(1)} p_{i+2}^{(1)} \cdots p_{j-1}^{(1)}\right)^2 \left(p_j^{(1)} p_{j+2}^{(1)} \cdots p_{i-1}^{(1)}\right)^2.
$$
One has $K(i,j)=K(j,i)$. Set $K(i,j)=0$ for  $i,j$ of the same parity.

\begin{proposition} \label{sgrad}
One has
$$
v_i = \frac{(-1)^i}{O-E} \sum_j K(i,j) \frac{\partial f}{\partial p_{j}},
$$
and
$$
\{f,g\} = \frac{1}{O-E} \sum_{i,j} (-1)^i K(i,j) \frac{\partial f}{\partial p_{j}} \frac{\partial g}{\partial p_{i}}.
$$
\end{proposition}

\proof To find ${\bf v}$, one needs to solve the system of linear equations
$$
  \frac{v_{i-1}}{(p_i-p_{i-1})^2} - \frac{v_{i+1}}{(p_i-p_{i+1})^2} = \frac{\partial f}{\partial p_{i}},\ i=1,\ldots,2k.
$$
The solution is given by the formula stated in the proposition, and then $\{f,g\}= dg({\bf v})$ implies the second formula.
\proofend

\begin{remark}
Although the  2-form $\Omega$ does not descend to $\cP$, the respective Poisson structure, invariant under $\SL(2)$, does.
Since $n$ is even, dim $\cP = n-3$ is odd, and the quotient Poisson structure has a kernel.

This Poisson structure is different from those introduced in the previous sections.
In particular, the cross-ratio coordinates $c_i$ and $c_j$ do not commute for $i$ and $j$ far apart.
\end{remark}

\subsubsection{Odd $n$, kernel of $\Omega$} \label{oddn}
If $n$ is odd, then $\Omega$ has a 1-dimensional kernel, and this field of directions is invariant under $\T$. 
It turns out that this field is obtained from the relation $\T$ when  $\alpha \to 0$, cf. Section \ref{infinitsection}.

\begin{proposition} \label{kernel}
The vector field $\widetilde\xi$ from Theorem \ref{thm:InfMap} spans ker $\Omega$.
\end{proposition}
\proof
Let $\widetilde\xi = \sum_{i=1}^n v_i \frac{\partial}{\partial p_i}$. One has
$$
i_{\widetilde\xi} \Omega = \sum_{i=1}^n \left[ \frac{v_{i-1}}{(p_{i-1}-p_i)^2} - \frac{v_{i+1}}{(p_i-p_{i+1})^2} \right] d p_i = 0,
$$
as claimed.
\proofend

\begin{remark} \label{charact} 
For even $n$, the form $\Omega$ is degenerate at $\pP$ if and only if the polygon $\pP$ has zero alternating perimeter.
Indeed, $A(\pP)=1$ is a necessary and sufficient condition for the system of equations
\[
\frac{v_{i-1}}{(p_{i-1}-p_i)^2} = \frac{v_{i+1}}{(p_i-p_{i+1})^2}
\]
to have a solution.
The space of polygons of zero alternating perimeter has codimension $1$.
Since the solution space of the above system has  dimension two for $A(\pP)=1$,
the restriction of $\Omega$ to the space of polygons of zero alternating perimeter has a one-dimensional kernel.
\end{remark}

\subsection{Additional integrals} \label{threesect}
We introduce three rational functions on $\widetilde\cP$:
$$
I(\pP)= \sum \frac{1}{p_{i}-p_{i+1}}, \  J(\pP)= \frac{1}{2} \sum \frac{p_i + p_{i+1}}{p_{i}-p_{i+1}},\ K(\pP)= \sum \frac{p_i p_{i+1}}{p_{i}-p_{i+1}},
$$
and three vector fields, the  infinitesimal generators of the diagonal action of the M\"obius group on the space of polygons:
$$
u = \sum \frac{\partial}{\partial p_i},\ v = \sum p_i\frac{\partial}{\partial p_i},\ w = \sum p_i^2\frac{\partial}{\partial p_i}.
$$

\begin{theorem} \label{upint}
1. The functions $I,J,K$ are invariant under the correspondences $\T$.\\
2. These functions are the Hamiltonians of the infinitesimal generators of the diagonal action of the M\"obius group on the space of polygons:
$$
i_u \Omega = dI,\ i_v \Omega = dJ,\ i_w \Omega = dK.
$$
\end{theorem}

\proof
Subtract equation (\ref{oneother}) from (\ref{twoother}) and sum over $i$. The left hand side vanishes, and the right hand side, after dividing by $\alpha-1$, gives $I(\pP)=I(\pQ)$. Thus $I$ is an integral.

Next, let $\psi (p) = 1/p$, and denote by $\Psi$ the diagonal action of $\psi$ on polygons. Then $\T \circ \Psi = \Psi \circ \T$. It follows that $I \circ \Psi$ is also an integral:
$$
(I \circ \Psi) \circ \T =  I \circ (\Psi \circ \T) = I \circ (\T \circ \Psi)  = (I \circ \T) \circ \Psi = I \circ \Psi.
$$
Since $K = - I \circ \Psi$, this implies that $K$ is an integral.

Finally, let $\phi_c (p) = p+c$, and let $\Phi_c$ be the respective diagonal action on polygons. Arguing as above, $K \circ \Phi_c$ is an integral for all $c$, and  the part linear in $c$ yields the integral $J$.

The second statement follows from the formulas that are verified by a direct calculation:
\begin{equation*}
\begin{split}
&\lambda(u) = -I,\ \lambda (v) = -J,\ \lambda(w) = -K,\\
&i_u \Omega = dI,\ i_v \Omega = dJ,\ i_w \Omega = dK.
\end{split}
\end{equation*}
This completes the proof.
\proofend

The next lemma, also verified by a calculation, describes the behavior of the integrals $I,J,K$ under infinitesimal M\"obius transformations.

\begin{lemma} \label{Hamf}
One has
\begin{equation} \label{action}
\begin{split}
&u(I)=0,\ u(J)=I,\ u(K)=2J;\\
&v(I)=-I,\ v(J)=0,\ v(K)=K;\\
&w(I)=-2J,\ w(J)=-K,\ w(K)=0.
\end{split}
\end{equation}
\end{lemma}

Equations (\ref{action}) mean that the action of the Lie algebra $\sl(2)$ on the space generated by $I,J,K$ is isomorphic to  the coadjoint representation. Lemma \ref{Hamf} has then the following interpretation. 

The moment map $\mu: \widetilde {\mathcal P}_n \to \sl(2)^*$ has $I,J,K$ as its components. 
Formulas (\ref{action}) imply that $\mu$ is $\sl(2)$-, and hence $\SL(2)$-equivariant, and this describes  the action of M\"obius transformations on these three functions. 

Formulas (\ref{action}) also imply that $IK - J^2$ is annihilated by $\sl(2)$, hence $IK - J^2$ is an $\SL(2)$-invariant function.
In fact, this is one of the integrals from formula (\ref{eqn:Integrals2}).

\begin{lemma}
\label{lem:IK-J^2}
One has
\[
IK - J^2 = \frac{1}{2} \sum_{i,j} [p_i, p_{j+1}, p_j, p_{i+1}] = \frac{n^2}{4} - G_2.
\]
\end{lemma}

Let $M^{n-3} \subset \widetilde {\mathcal P}_n$ be a generic level surface of the integrals $I,J,K$. The submanifold $M$ is not transverse to the orbits of the M\"obius group, but has 1-dimensional intersections with them, that is, the restriction of the projection $\pi: \widetilde \cP \to \cP$ on $M$ has 1-dimensional fibers.

\begin{lemma} \label{projfield}
These fibers are spanned by the vector field $\nu = Ku-2Jv+Iw$.
\end{lemma}

\proof 
Clearly, $\nu$ is vertical (it is the Hamiltonian vector field of the function $IK-J^2$). Formulas (\ref{action}) imply that 
$\nu(I)=\nu(J)=\nu(K)=0$ (the function $IK-J^2$ is a Casimir function),
hence $\nu$ is tangent to $M$.
\proofend

Thus $M$ carries two vector fields, $\xi$ and $\nu$.

\begin{lemma} \label{commute}
The fields $\xi$ and $\nu$ are invariant under the maps $\T$ and they commute.
\end{lemma}

\proof
That $\xi$ is invariant under $\T$ follows from Bianchi permutability and the fact the $\xi$ is a  limit of $\T$ as $\alpha \to 0$. On $M$, the functions $I,J,K$ are constant, so $\nu$ is an infinitesimal M\"obius transformation. Since $\T$ commute with M\"obius transformations, $\nu$ is invariant under $\T$. In the limit $\alpha \to 0$, one obtains $[\xi,\nu]=0$.
\proofend

\subsubsection{Geometric interpretation: axis of an ideal polygon} \label{axis}
Now we present a geometric interpretation of the two additional integrals of the maps $\T$ on $\cP$.

Consider the M\"obius transformation $A_t(\pP)$ for $t$ close to $1$.
We have
\[
A_{1+\epsilon}(u, v) =
\Id + \frac{\epsilon}{u-v}
\begin{pmatrix}
-v & uv\\
-1 & u
\end{pmatrix}
\]
Therefore, ignoring the terms of order two or higher in $\eps$,
\begin{multline*}
A_{1+\epsilon}(\pP) = \Id + \epsilon \sum_{i=1}^n \frac{1}{p_i - p_{i+1}}
\begin{pmatrix}
-p_{i+1} & p_ip_{i+1}\\
-1 & p_i
\end{pmatrix}
\\
= \Id + \epsilon
\begin{pmatrix}
-J + \frac{n}2 & K\\
-I & J + \frac{n}2
\end{pmatrix}.
\end{multline*}
It follows that, as $\epsilon \to 0$, the axis of the loxodromic transformation $A_{1+\epsilon}(\pP)$ converges to the line
through the eigenvectors of the matrix
\begin{equation}
\label{eqn:3Inv}
\begin{pmatrix}
-J + \frac{n}2 & K\\
-I & J + \frac{n}2
\end{pmatrix}
\end{equation}
(viewed as points of the circle or the sphere at infinity). 

Thus this axis is an invariant of $\T$ for all $\alpha$. The space of lines in $\HH^2$ is 2-dimensional, and the space of lines in $\HH^3$ is 4-dimensional. This correspond to two additional real or complex integrals. 

\begin{remark}
{\rm
If $\pP \T \pQ$, then the loxodromic transformations $A_t(\pP)$ and $A_t(\pQ)$ have the same parameters. Since their first-order behavior for $\epsilon \to 0$ is determined by the matrix \eqref{eqn:3Inv}, this matrix is the same for $\pP$ and $\pQ$.
This provides another way to prove Theorem \ref{upint}. Also the ``infinitesimal parameter'' of $A_t(\pP)$ is M\"obius invariant, and this 
corresponds to the M\"obius invariance of $IJ - K^2$.
}
\end{remark}

\begin{remark}
{\rm 
In the real case, we have three real invariants: the ``barycenter'' of an ideal polygon and its ``moment of inertia''.
The barycenter is the fixed point of the composition of infinitesimal translations
along the edges of the polygon (all translations by the same small distance).
Since the norm of the velocity field of a translation along a line is proportional to the hyperbolic cosine of the distance from this line,
the barycenter is the point that minimizes the sum of the hyperbolic sines of distances from the lines
(the gradient of the distance is perpendicular to the velocity field of the translation and has the same norm).

The moment of inertia is equal to the sum of hyperbolic cosines of distances to the sides of the polygon.
This is the angular velocity of the infinitesimal rotation about the barycenter, obtained by composing infinitesimal translations
along the sides.
}
\end{remark}

\section{Small-gons, exceptional polygons, and loxogons} \label{smallsection}

In this section we are interested in closed ideal polygons. First we describe the dynamics of $\T$ on ideal triangles, quadrilaterals, and pentagons, and then discuss polygons that are in the relation $\T$ with themselves.

\subsection{Triangles, quadrilaterals, and pentagons}
\label{345gons}

\subsubsection{Triangles} \label{trisect}
There is not much to say about triangles: all ideal triangles are isometric, so the moduli space consists of one point and there are no dynamics on $\mathcal{P}_3$.

As to $\widetilde {\mathcal{P}_3}$, in the Poincar\'e disk model of $\HH^2$, consider an ideal triangle $\pP$ whose vertices divide the boundary circle into three equal (in the Euclidean sense) arcs. If $\pP \T \pQ$, then $\pQ$ is obtained from $\pP$ by a rotation about the center of the disk, with the angle of rotation depending on $\alpha$. 

\subsubsection{Quadrilaterals} \label{quadsect} 
Space $\mathcal{P}_4$ is 1-dimensional, but there are still no dynamics on it. Let us discuss the situation in $\HH^3$.

An ideal quadrilateral $\pP$ in $\HH^3$ is represented by four cyclically ordered points in $\CP^1$.
The diagonals of $\pP$ have a unique common perpendicular $\ell$.
Let us call it the \emph{axis} of $\pP$.
(If the diagonals intersect, then we define the axis as the line through the intersection point perpendicular to both diagonals).
The rotation by $\pi$ about $\ell$ exchanges the endpoints of diagonals, thus maps $\pP$ to itself.

The common perpendicular of the lines $z_1z_3$ and $z_2z_4$ is the line $uv$ such that
\[
[u,v,z_1,z_3] = -1 = [u,v,z_2,z_4].
\]
This leads to the following system of equations:
$$
uv + \frac{z_1+z_3}2 (u+v) + z_1z_3 = 0,\ 
uv + \frac{z_2+z_4}2 (u+v) + z_2z_4 = 0.
$$

\begin{theorem} \label{thm:quad}
If $\pP \T \pQ$, then the ideal quadrilaterals $\pP$ and $\pQ$ are isometric, and they share the axis.
\end{theorem}

\proof
Send $\pP$ to a quadrilateral of the form $(z, w, -z, -w)$ by a M\"obius transformation.
Then the axis of $\pP$ is the line $0\infty$. We want to show that if $\pP \T \pQ$, then $\pQ$ also has axis $0\infty$.

In order to find the vertices of $\pQ$, we have to solve for $x, y, x', y'$ the system of equations
\[
[z,w, x,y] = \alpha, \ [w,-z, y,x'] = \alpha, \ [-z,-w, x',y'] = \alpha, \ [-w,z, y',x] = \alpha.
\]
Let us use the ansatz $x'=-x, y'=-y$ and show that we  have two solutions.
Indeed, the system becomes equivalent to
\[
[z,w, x,y] = \alpha, \quad [-z,w, -x,y] = \alpha,
\]
which, in an explicit form, is
\begin{equation}
\label{eqn:Quad3}
zx + yw = 0, \quad (1-\alpha)(xy + zw) = wx + zy.
\end{equation}
Expressing $y$ through $x$ from the first equation and substituting this in the second one we obtain a quadratic equation for $x$.

Thus both quadrilaterals $\pQ$ with $\pQ \T \pP=(z,w,-z,-w)$ have the form $(x,y,-x,-y)$, that is,  have axis $0\infty$.

Note that the first equation in \eqref{eqn:Quad3} implies $\frac{x}{y} = -\frac{w}{z}$.
We have
\[
[z,w,-z,-w] = \frac{4zw}{(z+w)^2} = \frac{4\lambda}{\lambda^2-1} \ {\rm with}\  \lambda = \frac{z}{w}.
\]
The value of $\frac{4\lambda}{\lambda^2-1}$ is invariant under $\lambda \mapsto -\frac{1}{\lambda}$.
This shows that the quadrilaterals $\pP$ and $\pQ$ have the same cross-ratio and hence are isometric. \proofend

\subsubsection{Pentagons} \label{pentasect}
We think of the moduli space ${\mathcal P}_5$ as the space of frieze patterns of width 2; denote the diagonal coordinates by  $x,y$. One has an area form
$$
\omega=\frac{dx\wedge dy}{xy},
$$
see formula (\ref{sympfr}).

\begin{theorem} \label{areapr}
The correspondences $\T$ are area-preserving.
\end{theorem}

\proof
According to Theorem \ref{thm:IntegralsClosed}, we have only one integral
$$
G_2 = \frac{F_2}{\sqrt{c_{[n]}}} = \sum_{i=1}^5 a_i = x+\frac{1+y}{x}+\frac{1+x}{y}+y+\frac{1+x+y}{xy},
$$
see Example \ref{ex:fr5}.

Making a choice of a direction, consider $\T$ as a map $T_\alpha$, and let $\Phi_t$ be the time-$t$ flow along the vector field $\xi$ corresponding to infinitesimal $\alpha$, see Section~\ref{infinitsection}. The maps $T_\alpha$ and $\Phi_t$ commute and preserve the level curves of the integral $I$. One has 
$$
 \Phi_t^* T_\alpha^* (\omega) = T_\alpha^*  \Phi_t^* (\omega) = T_\alpha^* (\omega),
 $$
 hence the 2-form $T_\alpha^* (\omega)$ is invariant under $\xi$ (see Section \ref{infinitsection}). It follows that $T_\alpha^* (\omega) = J \omega$, where $J$ is an integral (of the family of maps $T_\alpha$ and the field $\xi$).
We need to show that $J \equiv 1$. 

\begin{figure}[ht]
\centering
\includegraphics[height=2.3in]{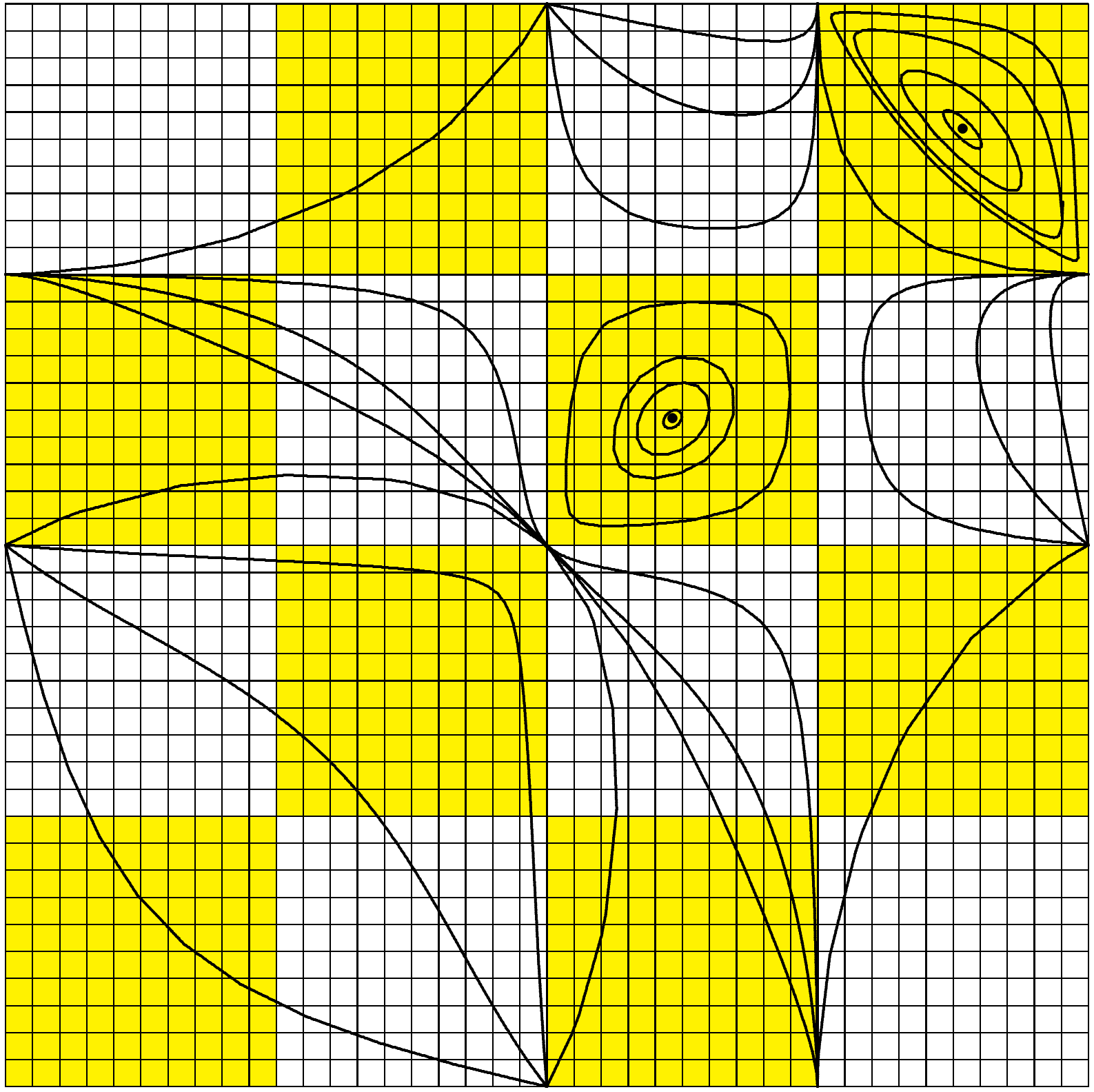}
\caption{Level curves of the integral $I$. The coordinates are $\arctan [p_1,p_2,p_3,p_4]$ and $\arctan [p_2,p_3,p_4,p_5]$.
}
\label{chessboard}
\end{figure}

Consider a closed level curve  $I=c$, see Figure \ref{chessboard}, and an infinitesimally close curve $I = c + \eps$. On the one hand, both curves are preserved by $T_\alpha$, so the area of the annulus bounded by them remains the same. On the other hand, this area is multiplied by the value of the integral $J$ on the curve $I=c$. Hence this value equals 1, as needed. 
\proofend

As a consequence of Theorem \ref{areapr}, we have

\begin{corollary}
Each map $T_\alpha$ equals the map $\Phi_t$ for some $t$.
\end{corollary}

\proof
By the  Arnold-Liouville theorem, the (Lagrangian) foliation on the level curves of the integral $I$  defines  a coordinate $x$ on each invariant curve $I=c$ such that the vector field is constant: $\xi = d/dx$, and the maps $T_\alpha$ are parallel translations $x \mapsto x + c$. This implies the result. 
\proofend

Next, consider the dynamics on $\widetilde {{\mathcal P}_5}$. We again have Liouville integrability.

\begin{theorem} \label{transl}
An open dense subspace of the space $\widetilde {{\mathcal P}_5}$ is foliated by surfaces that are invariant under the maps $\T$. These surfaces carry  flat structures in which the maps $\T$ are parallel translations.
\end{theorem}

\proof 
We have three integrals, $I,J,K$ of the maps $\T$, so $\widetilde {{\mathcal P}_5}$ is foliated by their level surfaces. These surfaces carry a flat structure given by the commuting vector fields $\xi$ and $\nu$, see Lemma \ref{commute}. The correspondences $\T$ preserve this flat structure, hence they are parallel translations.
\proofend

If the level surface is compact, it must be a torus.  
Theorem  \ref{transl} is illustrated in Figure \ref{penta}. 
 
\begin{figure}[ht]
\centering
\includegraphics[height=2.1in]{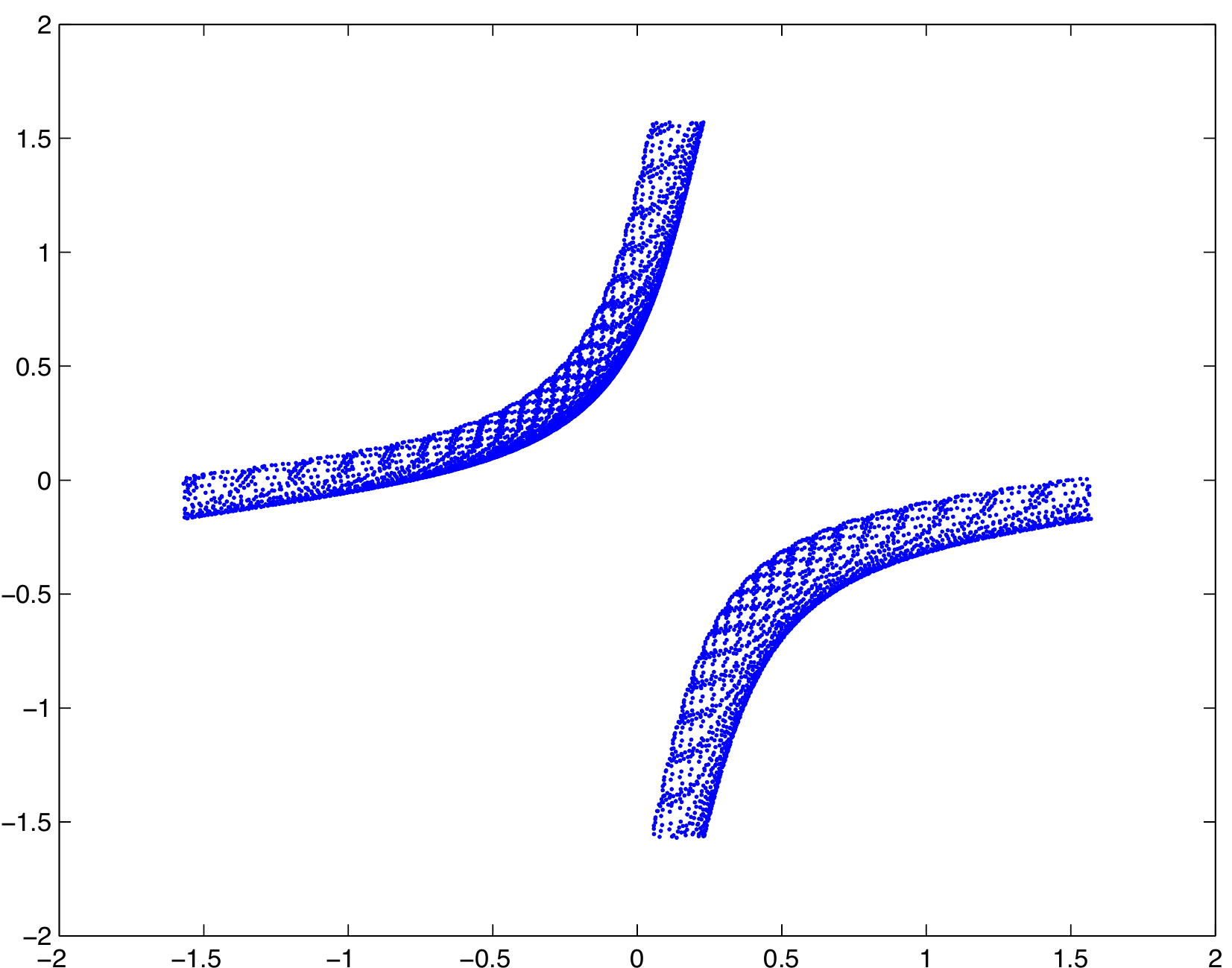}\ 
\includegraphics[height=2.1in]{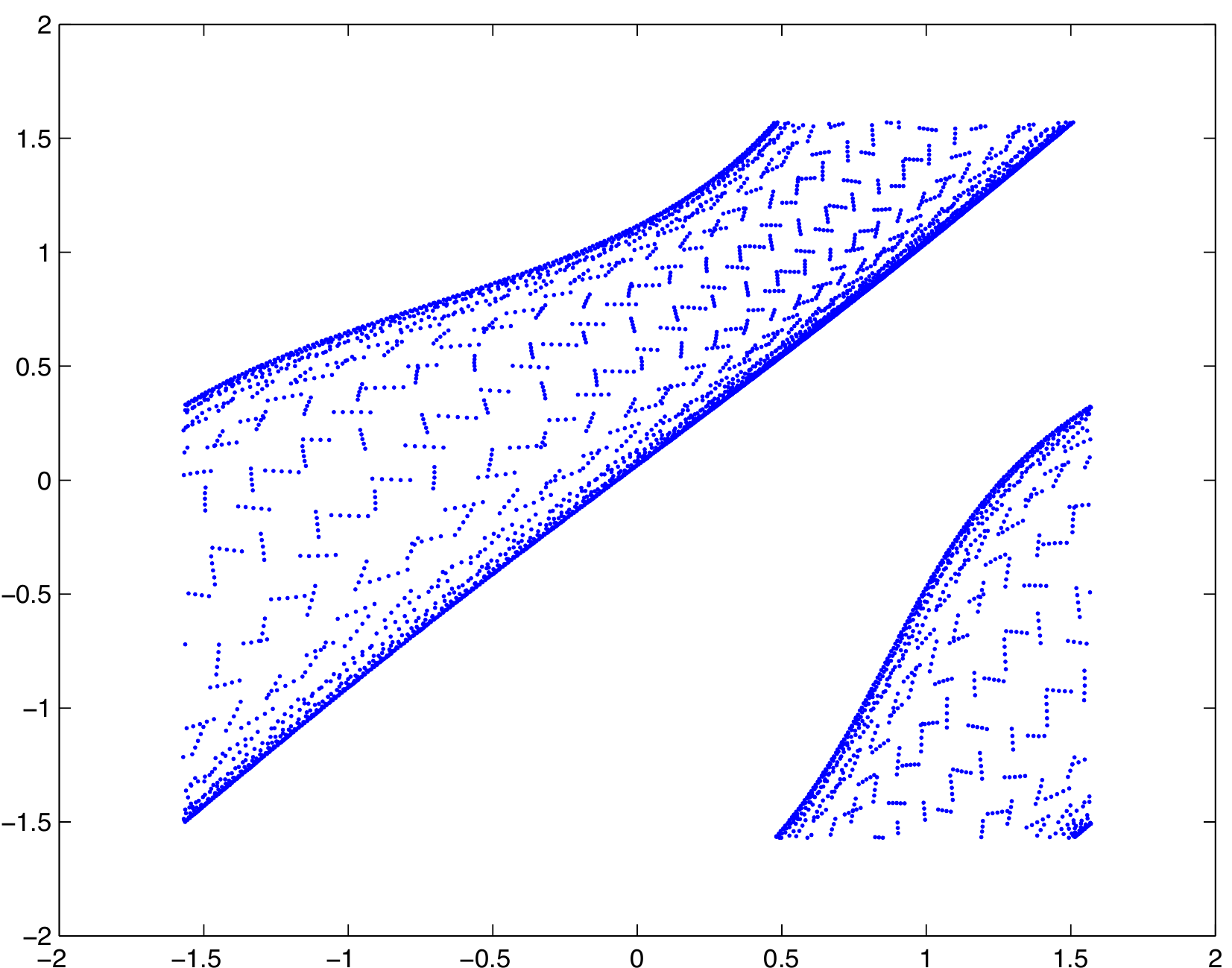}
\caption{An orbit of $\T$ in $\widetilde {{\mathcal P}_5}$ with $\alpha = -0.2$ and $\alpha = -1$.}
\label{penta}
\end{figure}

\subsection{Exceptional polygons} \label{subsect:except}
A closed ideal polygon is called \emph{$\alpha$-exceptional}, if there are infinitely many closed ideal polygons $\alpha$-related to it.

\begin{theorem}
\label{thm:PentaInf}
The only closed ideal $\alpha$-exceptional pentagons are
the regular ideal pentagon for $\alpha = \frac{3-\sqrt{5}}{3+\sqrt{5}}$ and the regular ideal pentagram for $\alpha = \frac{3+\sqrt{5}}{3-\sqrt{5}}$, see Figure \ref{excpenta}.
\end{theorem}

\begin{proof}
By Lemma \ref{lem:DRelations} and Theorem \ref{thm:InfAlpha}, a pentagon is closed and is $\alpha$-related to infinitely many pentagons if and only if
\begin{align*}
1 - (c_{i-1} + c_i + c_{i+1}) + c_{i-1} c_{i+1} &= 0,\\
\alpha^{-2} - \alpha^{-1}(c_{i-1} + c_i + c_{i+1}) + c_{i-1} c_{i+1} &= 0,
\end{align*}
for $i = 1, \ldots, 5$, where the indices are taken modulo $5$.
Subtracting the equations and dividing by $\alpha - 1$ one obtains
\[
c_{i-1} + c_i + c_{i+1} = \alpha + 1,
\]
which implies that all $c_i$ are equal.
It follows that there is a M\"obius transformation that cyclically permutes the vertices of the pentagon,
that is the pentagon is either regular or star-regular.

If $c_i = c$ for all $i$, then $c$ is found by solving the quadratic equation
\[
c^2 - 3c + 1 = 0.
\]
The solution $c = \frac{3-\sqrt{5}}2$ corresponds to the regular pentagon, the solution $c = \frac{3+\sqrt{5}}2$ corresponds to the regular pentagram.
\end{proof}

\begin{figure}[ht]
\centering
\includegraphics[height=2.5in]{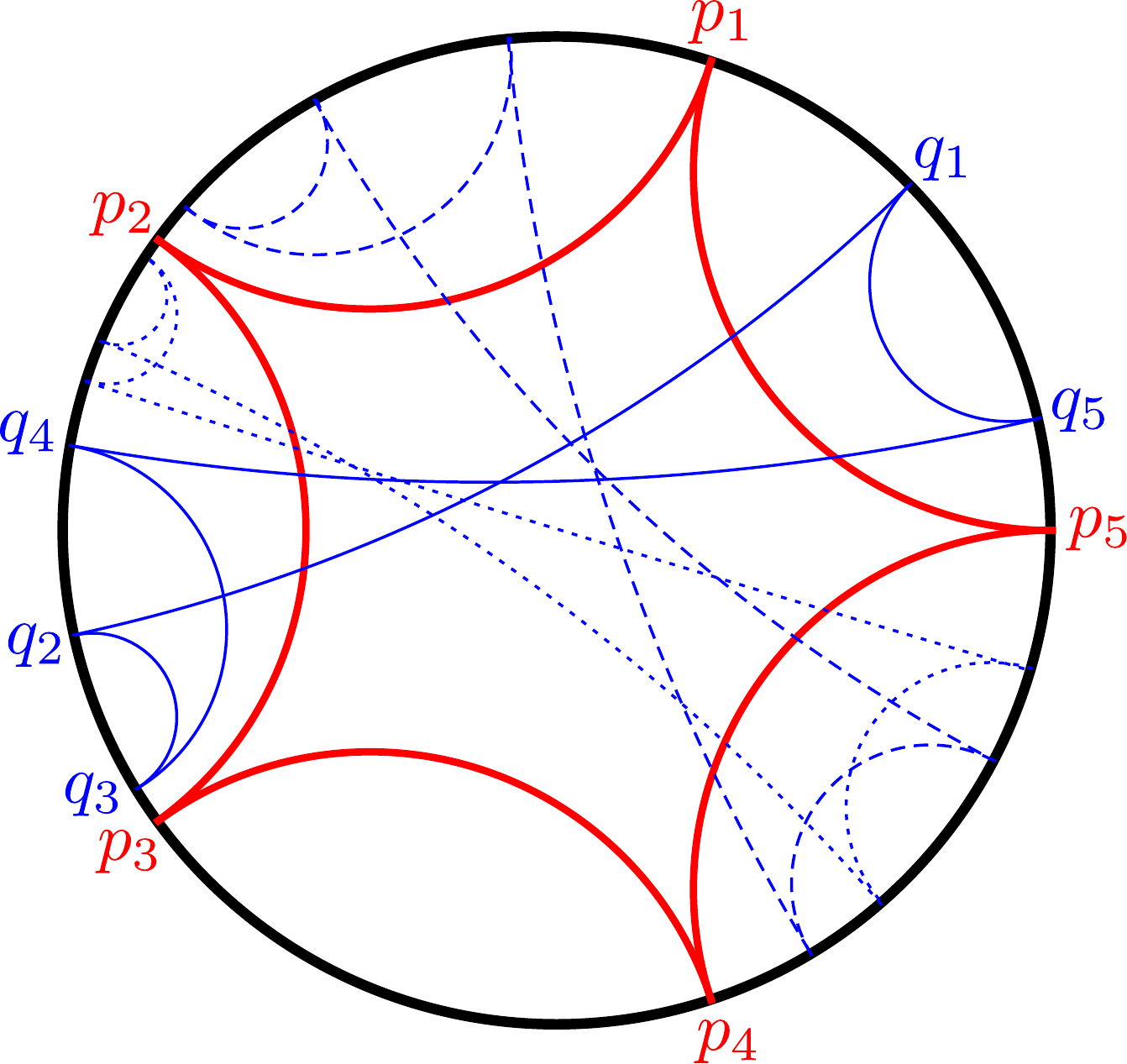}\quad
\includegraphics[height=2.5in]{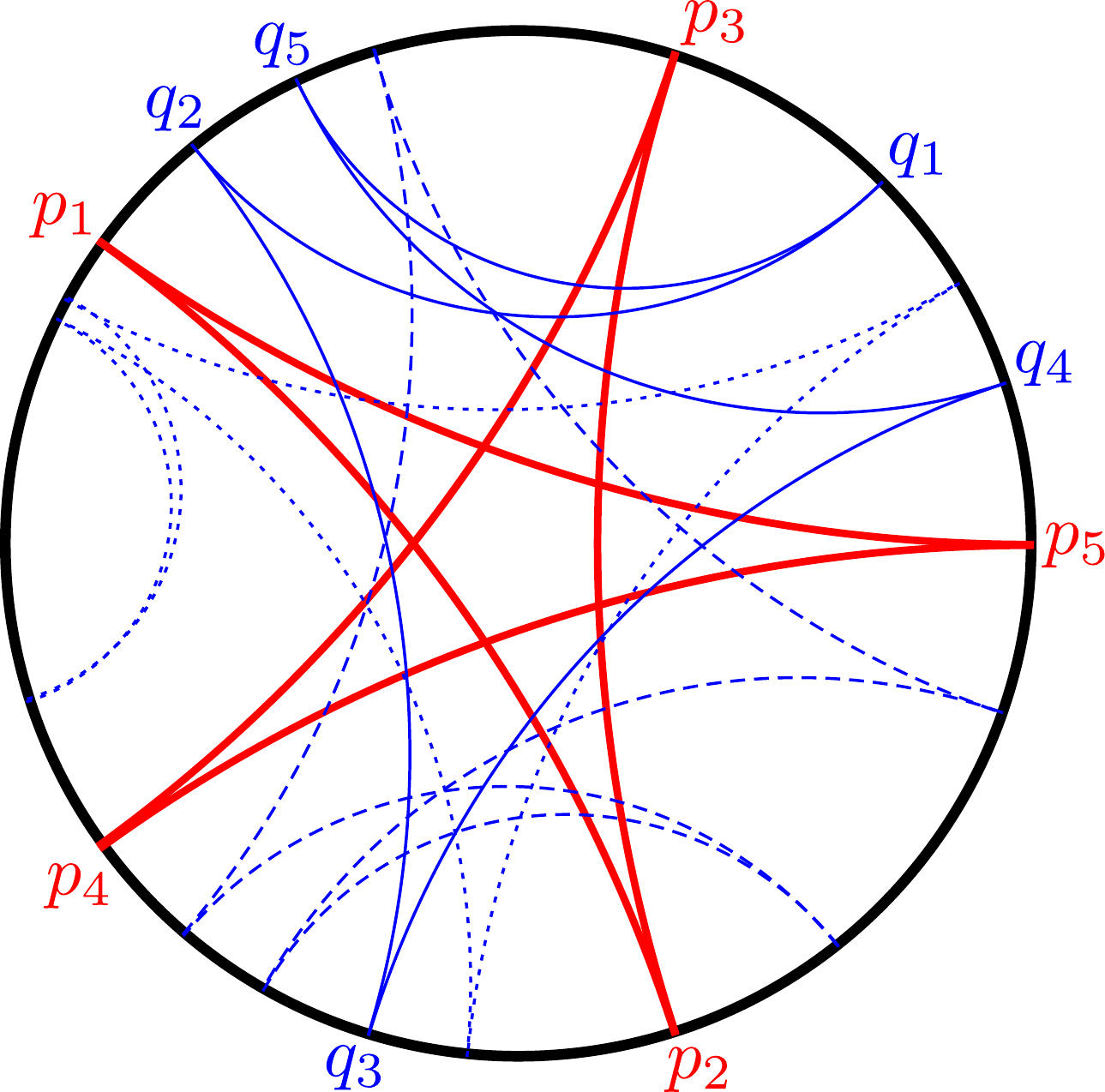}
\caption{A regular ideal pentagon and a regular ideal pentagram, along with
three $\alpha$-related ideal pentagons with $\alpha = \frac{3-\sqrt{5}}{3+\sqrt{5}}$ on the left, and 
$\alpha = \frac{3+\sqrt{5}}{3-\sqrt{5}}$ on the right.}
\label{excpenta}
\end{figure}

\begin{corollary}
\label{cor:PentaEquidist}
Let $p_1p_2p_3p_4p_5$ be a regular ideal pentagon in the hyperbolic plane.
Take any ideal point $q_1$.
Of the two lines through $q_1$ at distance $\frac12 \log 5$ from the line $p_1p_2$
choose the one that meets the absolute inside the arc $q_1p_2$ and denote by $q_2$
the other intersection point of this line with the absolute.
Repeat the construction until you get the point $q_6$: for all $i$ the line $q_iq_{i+1}$ is at the distance $\frac12 \log 5$ from the line $p_ip_{i+1}$
and the pairs $\{p_i, q_{i+1}\}$, $\{p_{i+1}, q_i\}$ are linked.
Then $q_6 = q_1$, independently of the choice of the point $q_1$.
\end{corollary}

\begin{proof}
Theorem \ref{thm:PentaInf} implies that $q_6 = q_1$ whenever we define $q_i$ recursively by
\[
[p_i, p_{i+1}, q_i, q_{i+1}] = \alpha, \text{ where } \alpha = \frac{3-\sqrt{5}}{3+\sqrt{5}}.
\]
According to equation \eqref{eqn:DistCR}, this means
\[
\tanh^2 \frac{\dist(p_ip_{i+1}, q_iq_{i+1})}2 = \alpha,
\]
and that the lines $p_ip_{i+1}$, $q_iq_{i+1}$ have the same direction.
With the help of the formula for $\artanh$, we obtain
\[
\dist(p_ip_{i+1}, q_iq_{i+1}) = \log \frac{1+\sqrt{\alpha}}{1-\sqrt{\alpha}}.
\]
By substituting the value of $\alpha$ we obtain $\frac12 \log 5$ on the right hand side.
\end{proof}

There is a similar corollary for pentagons equidistant from the regular ideal pentagram.
In this case, the distance $\dist(p_ip_{i+1}, q_iq_{i+1})$ is the same, but the lines $p_ip_{i+1}$ and $q_iq_{i+1}$ are oppositely directed.

Lines equidistant from a given line are tangent to a hypercycle about this line.
In the projective model of the hyperbolic plane, hypercycles are conics inscribed into the absolute
(the two points of tangency are the endpoints of the central line).
This leads to the following reformulation of Corollary \ref{cor:PentaEquidist}.

Inscribe into a circle five ellipses (of a specific size that can be computed from the formulas above)
tangent at the consecutive vertices of a regular pentagon.
From any point on the circle draw a tangent to one of the ellipses,
from the intersection point of this tangent with the circle draw a tangent to the next ellipse, and so on.
The choice between two possible tangents is explained above.
Then, independently of the choice of the initial point, one obtains a closed pentagon.


\begin{theorem} \label{Thm:exchexa}
The only $(-1)$-exceptional hexagon is formed by the vertices of a regular ideal octahedron
and a cycle of six edges remaining after removal of the boundaries of two opposite faces.
\end{theorem}

\begin{proof}
By Lemma \ref{lem:DRelations} and Theorem \ref{thm:InfAlpha}, we are looking for all sextuples $(c_1, \ldots, c_6)$ satisfying the equations
\begin{align*}
1 - (c_i + c_{i+1} + c_{i+2} + c_{i+3}) + (c_ic_{i+2} + c_ic_{i+3} + c_{i+1}c_{i+3}) &= 0,\\
1 + (c_i + c_{i+1} + c_{i+2} + c_{i+3}) + (c_ic_{i+2} + c_ic_{i+3} + c_{i+1}c_{i+3}) &= 0.
\end{align*}
The system is easily solved: by subtracting the two equations we obtain $c_i + c_{i+1} + c_{i+2} + c_{i+3} = 0$, which implies
\[
c_1 = c_3 = c_5 = c_7 =: c, \quad c_2 = c_4 = c_6 = c_8 = -c.
\]
By summing the same pair of equations we obtain $c^2 = -1$.

Acting by a M\"obius transformation we can assume $p_1 = 0$, $p_2 = 1$, $p_3 = i$.
Then for $c=i$ the above hexagon has vertices $(0, 1, i, \infty, -1, -i)$.
These are vertices of a regular ideal octahedron.
\end{proof}

The cross-ratio coordinates of an octagon orthogonal to infinitely many other octagons must satisfy
\begin{gather*}
1 + \sum_{\substack{i,j=1\\ i<j+1}}^6 c_i c_j = 0,\\
c_1 + \cdots + c_6 + c_1c_3c_5 + c_1c_3c_6 + c_1c_4c_6 + c_2c_4c_6 = 0,
\end{gather*}
and 14 more equations obtained by cyclic shifts.
Among these, only six equations are independent (closure conditions for $c_i$ and closure conditions for $-c_i$),
therefore one expects a two-parameter family of octagons with this property.

\subsection{Loxogons} \label{loxosection}

Here we are interested in closed ideal polygons that are in the relation $\T$ with themselves. We call such polygons {\it loxogons}.

\begin{definition}
An $(n,k)$-loxogon is an ideal $n$-gon $\pP$ such that 
$$
[p_i,p_{i+1},p_{i+k},p_{i+k+1}]=\alpha
$$ 	
for all $i=1,\ldots,n$ and some constant $\alpha$; here $2\leq k \leq n-2$. A loxogon is called trivial if it is projectively regular (possibly, star-shaped).
\end{definition}

Clearly, a regular $n$-gon is a $(n,k)$-loxogon for all $k$. The problem is to describe the non-trivial loxogons, if any.

\begin{remark}
{\rm
Bianchi permutability implies that the maps $\T$ act on the space of $(n,k)$-loxogons.
}
\end{remark}

\begin{theorem} \label{triv}
In the following cases every $(n,k)$-loxogon is trivial:\\
1) $n$ arbitrary, $k=2$;\\
2) $n$ odd, $k=3$;\\
3) $k$ arbitrary, $n=2k+1$.\\
On the other hand, for every even $n$ and odd $k$ there exist a 1-parameter family of (projective equivalence classes of) non-trivial $(n,k)$-loxogons, see Figure \ref{loxogon}.
\end{theorem}

\begin{figure}[ht]
\centering
\includegraphics[height=2in]{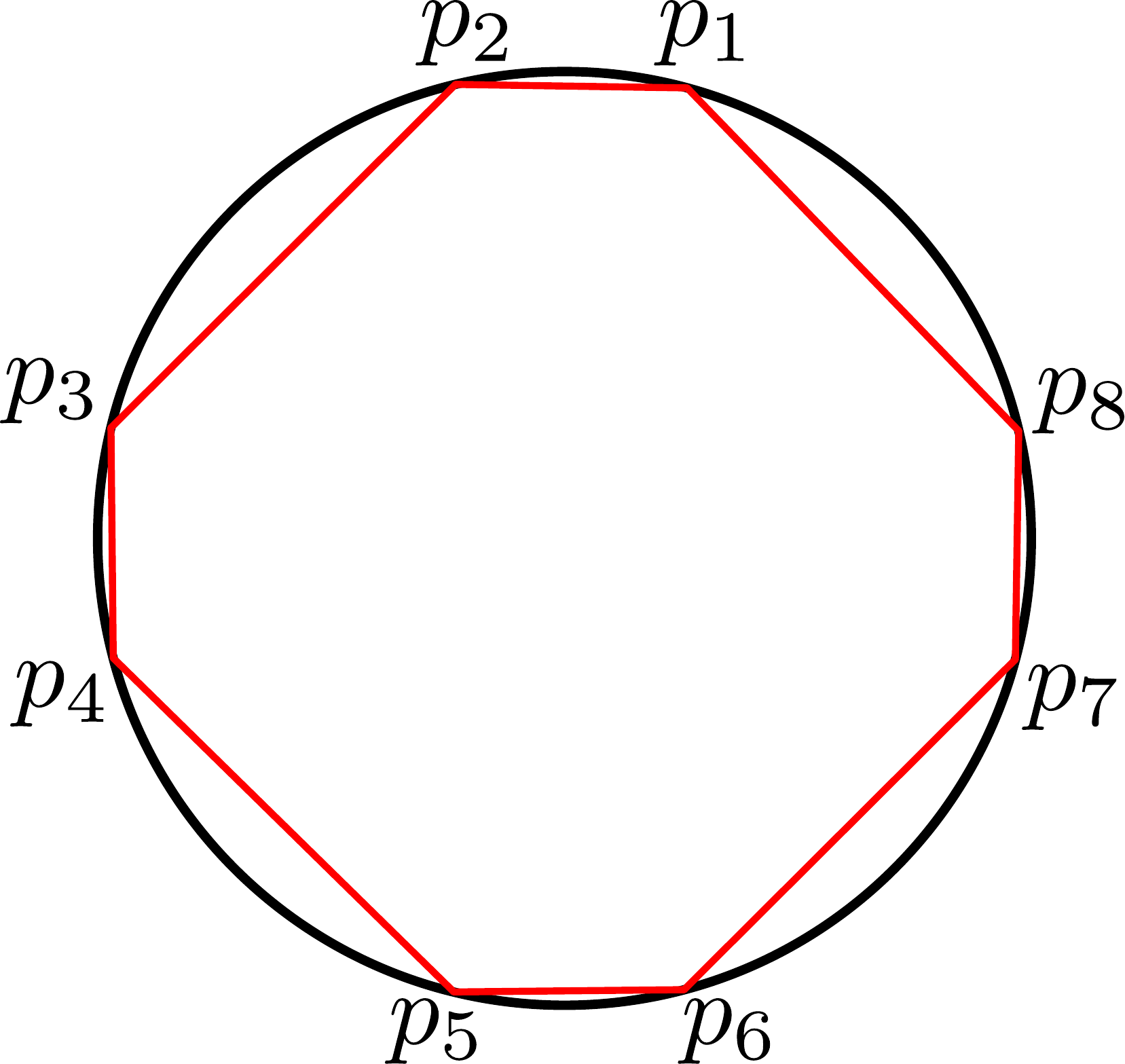}
\caption{An $(8,3)$-loxogon: one has $[p_1,p_2,p_4,p_5]=[p_2,p_3,p_5,p_6]$, and  the dihedral symmetry.}
\label{loxogon}
\end{figure}

\begin{remark}
{\rm 
This result is similar to Theorems 9 and 10 of \cite{Tab}, where a class of plane polygons, called bicycle $(n,k)$-gons, is studied. 
}
\end{remark}

\proof 
Let $k=2$. Consider the projective transformation $\varphi$ that takes $p_1,p_2,p_3$ to $p_2,p_3,p_4$. Since $[p_1,p_2,p_3,p_4]= [p_2,p_3,p_4,p_5]$, this transformation takes $p_4$ to $p_5$, and so on. Hence $\varphi$ takes $\pP$ to itself, cyclically permuting its vertices. Therefore $P$ is projectively regular.

Let $n$ be odd. Consider the respective lifted $n$-gon in the affine plane with the vertices $P_i$. The loxogon condition implies that $[P_i,P_{i+k}] [P_{i+1},P_{i+k+1}] = \alpha$, hence $[P_i,P_{i+k}] = [P_{i+2},P_{i+k+2}]$, and since $n$ is odd, the determinants $[P_i,P_{i+k}]$ are equal for all $i$. 

The rows of the corresponding frieze pattern are made of the determinants $[P_i,P_{i+k}], i=1,\ldots,n$. If  $k=3$, then the second non-trivial row is constant. Then the frieze relation  implies that the first non-trivial row is also constant (using that $n$ is odd).

If $k=(n-1)/2$, then the two middle rows of the frieze pattern are constant, and then, applying the frieze relation consecutively, we conclude that all rows are constant.

In both cases, the first row is constant, and this is the case of $k=2$, already considered.

Now we construct non-trivial $(n,k)$-loxogons for even $n$ and odd $k$. Start with a regular $n$-gon $\pP$ in $\RP^1 = \R \cup \infty$ with $p_j=\tan (\pi j/n) , j=1,\ldots,n$. Consider the $2n$-gon $\pQ$ with the vertices $p^-_1,p^+_1,p^-_2,p^+_2,\ldots$, where
$$
p^-_j = \tan \left(\frac{\pi j}{n} - \beta\right),\ p^+_j = \tan \left(\frac{\pi j}{n} + \beta\right),\ j=1,\ldots,n.
$$
Here $\beta$ is a parameter of the construction.

The polygons $\pQ$ has an $n$-fold rotational symmetry $(p^-_j,p^+_j) \mapsto (p^-_{j+1},p^+_{j+1})$. 
One has
$$
[p^-_j,p^+_j, p^+_{j+k}, p^-_{j+k+1}] = [p^+_j,p^-_{j+1}, p^-_{j+k+1},p^+_{j+k+1}], 
$$
which follows from the existence of the projective involution (dihedral symmetry)
$$
(p^-_j,p^+_j, p^+_{j+k}, p^-_{j+k+1}) \mapsto (p^+_{j+k+1},p^-_{j+k+1},p^-_{j+1},p^+_j)
$$
and the symmetry of cross-ratio. This, along with the rotational symmetry, implies that $\pQ$ is $(n,2k+1)$-loxogon. 
\proofend

\begin{conjecture}
There are no non-trivial $(n,k)$-loxogons for $n$ odd. Equivalently, if a frieze pattern of even width has a constant row then all rows are constant.
\end{conjecture}

This  fails for frieze patterns of odd width: for example (pointed out by S. Morier-Genoud), 
$$
 \begin{array}{ccccccccccccccccccccccc}
&&1&&1&&1&&1&&1&&1&&1&&1
 \\[4pt]
&1&&3&&1&&3&&1&&3&&1&&3&
 \\[4pt]
&&2&&2&&2&&2&&2&&2&&2&&2&
 \\[4pt]
&3&&1&&3&&1&&3&&1&&3&&1&
\\[4pt]
&&1&&1&&1&&1&&1&&1&&1&&1
\end{array}
$$
The next result provides an evidence toward the above conjecture.

\begin{theorem} \label{infinites}
For odd $n$ and every $k$, there do not exist non-trivial deformations of a regular ideal $n$-gon in the class of $(n,k)$-loxogons.
\end{theorem}

\proof
We  work with polygons in the affine plane. Let
$$
P_j = \left( \cos\left( \frac{\pi j}{n} \right), \sin\left( \frac{\pi j}{n} \right)\right),\ j=1,\ldots,n,
$$
be the vertices of a regular $n$-gon. We fix the normalization $[P_j,P_{j+1}]=\sin(\pi/n)$, and the respective second-order linear recurrence is 
\begin{equation} \label{rec}
P_{j+1}=2\cos\left( \frac{\pi}{n} \right) P_j - P_{j-1}.
\end{equation}

Consider an infinitesimal deformation $P_j + \eps V_j$, where $V_j$ is an $n$-anti-periodic sequence of vectors, and assume that the resulting polygon is an $(n,k)$-loxogon. Calculating modulo $\eps^2$, we obtain two systems of equations
\begin{equation} \label{norm}
[P_j,V_{j+1}] + [V_j,P_{j+1}]=0,\ j=1,\ldots,n,
\end{equation} 
and
\begin{equation} \label{lox}
[P_j,V_{j+k}] + [V_j, P_{j+k}] = C,\ j=1,\ldots,n.
\end{equation}
The first is the normalization, and the second is the loxogon condition, with $C$ a constant.

Consider  system (\ref{norm}). Let
$$
V_j=a_jP_j+b_jP_{j+1}=c_jP_j+d_jP_{j-1}.
$$
Then  recurrence (\ref{rec}) implies that
$$
\frac{c_j-a_j}{b_j} = 2\cos\left( \frac{\pi}{n} \right), \ \frac{d_j}{b_j} =-1.
$$
Substitute vectors $V_j$ into (\ref{norm}) to obtain
\begin{equation} \label{abd}
a_j=-c_{j+1},\ b_j = \frac{c_j+c_{j+1}}{2\cos(\pi/n)},\ d_j = - \frac{c_j+c_{j+1}}{2\cos(\pi/n)},
\end{equation}
where $c_j$ is an $n$-periodic sequence to be determined.

Now consider  system (\ref{lox}). Substituting vectors $V_j$, using (\ref{abd}), and collecting terms yields the linear system
\begin{equation} \label{linc}
\mu_{k-1} c_j - \mu_{k+1} c_{j+1} + \mu_{k+1} c_{j+k} - \mu_{k-1} c_{j+k+1} = C,\ j=1,\ldots,n,
\end{equation}
where $\mu_k = \sin(\pi k/n)$.

First, we note that $C$ must be zero. Indeed, add  equations (\ref{linc}): the left hand side vanishes, and so must the right hand side.

Second, the system (\ref{linc}) has a 3-dimensional space of trivial solutions that correspond to the action of the Lie algebra $\sl(2)$. These solutions are given by the formulas
$$
c_j = 1;\ c_j = \cos\left( \frac{\pi (2j-1)}{n} \right);\ c_j = \sin\left( \frac{\pi (2j-1)}{n} \right).
$$
We wish to prove that there are no other solutions.

To this end, consider the eigenvalues of the matrix defining  system (\ref{linc}). This is a circulant matrix, and its eigenvalues are given by the formula
$$
\lambda_j = \mu_{k-1} c_j - \mu_{k+1} \omega_j + \mu_{k+1} \omega_j^k - \mu_{k-1} \omega_j^{k+1},\ j=0,\ldots,n-1,
$$
where 
$\omega_j = e^{i\frac{2\pi j}{n}}$ is the $j$th rood of unity, 
see, e.g., \cite{Dav}. 

We are interested in zero eigenvalues. One has $\lambda_j=0$ if and only if
$$
\omega_j^{k+1} = \frac{\mu_{k-1} - \mu_{k+1} \omega_j}{\mu_{k-1} - \mu_{k+1} \overline \omega_j}.
$$
Let $2\alpha$ be the argument of the unit complex number on the right. A direct calculation yields
$$
\tan \alpha = - \frac{\sin\left( \frac{\pi (k+1)}{n} \right) \sin\left( \frac{2\pi j}{n} \right)}{\sin \left( \frac{\pi (k-1)}{n}\right) - \sin\left( \frac{\pi (k+1)}{n} \right) \cos\left( \frac{2\pi j}{n} \right)}.
$$
The argument of $\omega_j^{k+1}$ is $2\pi j(k+1)/n$, hence (after cleaning up the formulas)
$$
\sin\left( \frac{\pi j(k+1)}{n} \right) \sin\left( \frac{\pi (k-1)}{n} \right) = \sin\left( \frac{\pi j(k-1)}{n} \right) \sin\left( \frac{\pi (k+1)}{n} \right),
$$
or, equivalently,
\begin{equation}  \label{tangs}
\tan \left( \frac{\pi j}{n} \right) \tan \left( \frac{\pi k}{n} \right) = \tan \left( \frac{\pi jk}{n} \right) \tan \left( \frac{\pi}{n} \right).
\end{equation}
Note the trivial solutions $j=0,\pm1$.

This equation appeared in \cite{Tab} and in \cite{ASTW}, and it was solved in \cite{CC}. This equation has non-trivial solutions if and only if $n=2(j+k)$ and $n$ divides $(j-1)(k-1)$. In particular, there are no non-trivial solutions for odd $n$, and this concludes the proof.
\proofend

\appendix
\section{Loxodromic transformations along the edges of an ideal tetrahedron}
\subsection{Consistent labelings} \label{consist}
Consider an ideal hyperbolic tetrahedron with vertices $u_0, u_1, u_2, u_3 \in \CP^1$.
The tetrahedron may degenerate (this happens when $u_0, u_1, u_2, u_3$ lie on a circle or, equivalently, when their cross-ratio is real).
Consider a labeling $\{c_{ij}\}$ of the edges of the tetrahedron, where $c_{ij} = c_{ji}$.
Denote by $L_{ij}$ the loxodromic transformation with the axis $u_iu_j$ and parameter $c_{ij}$ (see Section \ref{sec:Loxodromics} for a definition).
We have $L_{ij} = L^{-1}_{ji}$.

\begin{definition}
A labeling $\{c_{ij}\}$ is called \emph{consistent} if, for every $i \in \{0, 1, 2, 3\}$, we have
\begin{equation}
\label{eqn:Consist}
L_{ij}  L_{ik}  L_{il} = \Id,
\end{equation}
where $(i,j,k,l)$ is any even permutation of $(0, 1, 2, 3)$.
\end{definition}

\begin{theorem}
\label{thm:ConsistLabel}
A labeling is consistent if and only if the labels satisfy the following system of equations:
\begin{subequations}
\begin{align}
&c_{ij} = c_{kl},\label{eqn:LoxoConsist0}\\
&c_{ij} c_{ik} c_{il} = 1,\label{eqn:LoxoConsist1}\\
&[u_i, u_j, u_k, u_l] c_{ij} + [u_i, u_k, u_j, u_l] c_{ik}^{-1} = 1,\label{eqn:LoxoConsist2}
\end{align}
\end{subequations}
where in the last equation $(i, j, k, l)$ is any even permutation of $(0, 1, 2, 3)$.

Besides, every label $c_{ij}$ can take any non-zero value except $[u_i, u_j, u_k, u_l]^{-1}$ (where $ijkl$ is an even permutation)
and determines all other labels uniquely.
\end{theorem}

\begin{remark}
\label{rem:ConsistLabel}
There are in total $12$ equations of the form \eqref{eqn:LoxoConsist2}: one equation for every pair of adjacent edges.
Equation for the pair $(ji, jl)$ has the form
\[
[u_j, u_i, u_l, u_k] c_{ij} + [u_j, u_l, u_i, u_k] c_{jl}^{-1} = 1.
\]
Due to the symmetries of the cross-ratio, the coefficients here are the same as in \eqref{eqn:LoxoConsist2}.
It follows that $c_{ik} = c_{jl}$: the system \eqref{eqn:LoxoConsist2} implies the system \eqref{eqn:LoxoConsist0}.

On the other hand, \eqref{eqn:LoxoConsist0} implies that from any of the equations \eqref{eqn:LoxoConsist1} the other ones follow,
and that the $12$ equations \eqref{eqn:LoxoConsist2} can be reduced to just three.
Thus the system \eqref{eqn:LoxoConsist0} -- \eqref{eqn:LoxoConsist2} is equivalent to the system
\begin{subequations}
\begin{align}
&c_{23} = c_{01}, \quad c_{13} = c_{02}, \quad c_{12} = c_{03},\label{eqn:LoxoConsist0Prime}\\
&c_{01} c_{02} c_{03} = 1,\label{eqn:LoxoConsist1Prime}\\
&[u_0, u_i, u_{i+1}, u_{i+2}] c_{0i} + [u_0, u_{i+1}, u_i, u_{i+2}] c_{0,i+1}^{-1} = 1, \quad i = 1, 2, 3.\label{eqn:LoxoConsist2Prime}
\end{align}
\end{subequations}
(The indices in the last line are taken modulo $3$.)

These are seven equations on six variables but, as we will show, they have a one-parameter solution set.
\end{remark}

\begin{lemma}
\label{lem:Loxo2Meaning}
Equation \eqref{eqn:LoxoConsist2} is equivalent to $L_{ij}L_{ik}L_{il}(u_l) = u_l$.
\end{lemma}
\begin{proof}
Since $u_l$ is a fixed point of the transformation $L_{il}$, and because of $L_{ij}^{-1} = L_{ji}$, we have
\[
L_{ij}L_{ik}L_{il}(u_l) = u_l \Leftrightarrow L_{ij}L_{ik}(u_l) = u_l \Leftrightarrow L_{ji}(u_l) = L_{ik}(u_l).
\]
On the other hand, Lemma \ref{loxodr} implies
\[
[u_j, u_i, u_l, L_{ji}(u_l)] = c_{ij}^{-1}, \quad [u_i, u_k, u_l, L_{ik}(u_l)] = c_{ik}^{-1}.
\]
From these formulas and the properties of cross-ratio, we infer
\[
[u_i, u_j, u_k, u_l] c_{ij} = [u_i, u_j, u_k, u_l] [u_i, u_j, u_l, L_{ji}(u_l)] = [u_i, u_j, u_k, L_{ji}(u_l)],
\]
and similarly,
\begin{multline*}
[u_i, u_k, u_j, u_l] c_{ik}^{-1} = [u_i, u_k, u_j, u_l] [u_i, u_k, u_l, L_{ik}(u_l)]\\
= [u_i, u_k, u_j, L_{ik}(u_l)] = 1 - [u_i, u_j, u_k, L_{ik}(u_l)].
\end{multline*}
Thus $L_{ji}(u_l) = L_{ik}(u_l)$ if and only if \eqref{eqn:LoxoConsist2} holds.
\end{proof}

\begin{proof}[Proof of Theorem \ref{thm:ConsistLabel}]
By Lemma \ref{lem:Loxo2Meaning}, for every consistent labeling,  equations \eqref{eqn:LoxoConsist2} hold.
Equations \eqref{eqn:LoxoConsist2} imply equations \eqref{eqn:LoxoConsist0}, see Remark \ref{rem:ConsistLabel}.
To prove the validity of \eqref{eqn:LoxoConsist1}, note that consistent labelings are M\"obius-invariant,
so  we may assume that $u_i = \infty$.
Since $L_c(\infty, u)$ is a complex affine map with the coefficient $c$,
the composition $L_{ij}L_{ik}L_{il}$ is a complex affine map with the coefficient $c_{ij}c_{ik}c_{il}$.
Thus $L_{ij}L_{ik}L_{il} = \Id$ implies $c_{ij}c_{ik}c_{il} = 1$,
and for every consistent labeling, all the equations \eqref{eqn:LoxoConsist0} -- \eqref{eqn:LoxoConsist2} hold.

Vice versa, equations \eqref{eqn:LoxoConsist2} imply that, for every $i$,
the composition $L_{ij}L_{ik}L_{il}$ has at least two fixed points: $u_i$ and $u_l$,
and equations \eqref{eqn:LoxoConsist1} imply that $L_{ij}L_{ik}L_{il}$ is of parabolic type
(affine with coefficient $1$ under assumption $u_i = \infty$).
Together this implies $L_{ij}L_{ik}L_{il} = \Id$.

We now proceed to the second part of Theorem \ref{thm:ConsistLabel}, the parametrization of the space of admissible labelings.
Due to Remark \ref{rem:ConsistLabel}, it suffices to solve the system of four equations \eqref{eqn:LoxoConsist1Prime} -- \eqref{eqn:LoxoConsist2Prime}
on three variables.
We claim that any two of the equations \eqref{eqn:LoxoConsist2Prime} imply the third one, as well as \eqref{eqn:LoxoConsist1Prime}.
Using the interpretation of these equations given in Lemma \ref{lem:Loxo2Meaning},
we assume, without loss of generality, that $L_{01}L_{02}L_{03}(u_3) = u_3$ and $L_{02}L_{03}L_{01}(u_1) = u_1$.
From $L_{01}(u_1) = u_1$, it follows that
\[
L_{01}L_{02}L_{03}(u_1) = L_{01}L_{02}L_{03}L_{01}(u_1) = L_{01}(L_{02}L_{03}L_{01}(u_1)) = L_{01}(u_1) = u_1.
\]
Thus the M\"obius tranformation $L_{01}L_{02}L_{03}$ has three distinct fixed points $u_0$, $u_1$, and $u_3$, and is therefore the identity.
It follows that the third of equations \eqref{eqn:LoxoConsist2Prime} and equation \eqref{eqn:LoxoConsist1Prime} are satisfied.

From \eqref{eqn:LoxoConsist2} one can express $c_{ik}$ as a function of $c_{ij}$:
\[
c_{ik} = \frac{1-[u_i, u_j, u_k, u_l]}{1-[u_i, u_j, u_k, u_l]c_{ij}}, \quad ijkl \text{ even}.
\]
The values $c_{ij} = 0, \infty, [u_i, u_j, u_k, u_l]^{-1}$ correspond to the values $c_{ik} = [u_i, u_k, u_l, u_j]^{-1}$, $0$, $\infty$, respectively.
This describes the admissible values of each label, as $c_{ij} \in \C \setminus \{0, [u_i, u_j, u_k, u_l]^{-1}\}$ for an even permutation $ijkl$.
\end{proof}

The transformations $L_{ij}$ are elements of the group $\PGL(2, \C)$.
They can be represented by matrices from $\GL(2, \C)$ as described at the end of Section \ref{sec:Loxodromics}.
If we denote $A_{ij} = A_{c_{ij}}(u_i, u_j)$, then Theorem \ref{thm:ConsistLabel} says that, for any even permutation $ijkl$ of $0123$, the matrix $A_{ij}A_{ik}A_{il}$ is a scalar multiple of the identity matrix.
In fact, more is true.

\begin{lemma}
\label{lem:ConsistA}
If $\{c_{ij}\}$ is a consistent labeling of the edges of an ideal tetrahedron, then for every even permutation $ijkl$ of $0123$ we have
\[
A_{ij}A_{ik}A_{il} = \Id, \quad A_{ij}A_{ik} = A_{jl}A_{kl}.
\]
\end{lemma}
\begin{proof}
We have $\det (A_{ij}A_{ik}A_{il}) = c_{ij}c_{ik}c_{il} = 1$.
This implies that the product in question is $\pm \Id$.
In order to check that the sign is positive, consider the case $u_i = \infty$.
Then it is easy to see that the top left entry of the product is $1$, thus the product equals $\Id$.
The general case holds by continuity (or by observing that matrices $A_t(u,v)$ behave well under conjugation).

The second equation follows from the first and the relation $(A_\lambda(p,q))^{-1} = \lambda^{-1}A_\lambda(q,p)$.
From $A_{li}A_{lk}A_{lj} = \Id$ and $A_{ij}A_{ik}A_{il} = \Id$ it follows that
\begin{gather*}
A_{li} = A_{lj}^{-1}A_{lk}^{-1} = c_{jl}^{-1}c_{kl}^{-1}A_{jl}A_{kl} = c_{il} A_{jl}A_{kl}\\
c_{il}^{-1}A_{li} = A^{-1}_{il} = A_{ij}A_{ik},
\end{gather*}
hence $A_{ij}A_{ik} = A_{jl}A_{kl}$ as claimed.
\end{proof}

\begin{proof}[Proof of Lemma \ref{lem:LaxConjug}]
The second of the equations of Lemma \ref{lem:ConsistA} implies $A_{kl} = A^{-1}_{jl} A_{ij} A_{ik}$.
Substituting $(u_i, u_j, u_k, u_l) = (p_i, p_{i+1}, q_i, q_{i+1})$ we obtain
\[
A_\lambda(q_i, q_{i+1}) = A_\mu^{-1}(p_{i+1}, q_{i+1}) A_\lambda(p_i, p_{i+1}) A_\mu(p_i, q_i)
\]
with $\lambda = c_{kl} = c_{ij}$ and $\mu = c_{ik}$.
The consistency equation \eqref{eqn:LoxoConsist2} implies that $\lambda$ and $\mu$ are related by
\[
\alpha \lambda + (1-\alpha)\mu^{-1} = 1.
\]
This is equivalent to $\mu = \frac{1-\alpha}{1-\alpha\lambda}$, and the lemma is proved.
\end{proof}

\subsection{From tetrahedron to cube}
\label{sec:TetraToCube}
In the previous section we explained how an ideal tetrahedron gives rise to a Lax representation for the cross-ratio dynamics.
In \cite{BS, BSb} a different geometric picture, based on a cube, was used to derive a different Lax representation.
Let us explain a connection between both pictures.

Let $(c_{ij})$ be any consistent labeling of an ideal tetrahedron $u_0u_1u_2u_3$.
Take any $v_0 \in \CP^1$, different from $u_1, u_2, u_3$, and put
\[
v_1 = L_{32}(v_0), \quad v_2 = L_{13}(v_0), \quad v_3 = L_{21}(v_0).
\]
The consistency condition \eqref{eqn:Consist} implies
$
v_i = L_{jk}(v_l)
$
for every even permutation $ijkl$.
Due to Lemma \ref{loxodr}, this can be rewritten as
\begin{equation}
\label{eqn:CubeFaces}
[u_i, u_j, v_k, v_l] = c_{ij}.
\end{equation}

The points $u_i, v_i$, $i = 0, 1, 2, 3$, can be viewed as the vertices of a (combinatorial) cube, see Figure \ref{fig:Cube}.
Each of the equations \eqref{eqn:CubeFaces} corresponds to a face of the cube.
The faces are planar if and only if the numbers $c_{ij}$ are real.

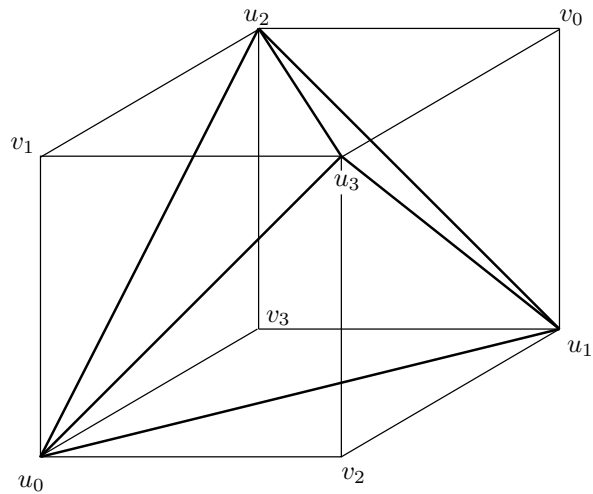
\begin{figure}[ht]
\begin{center}
\begin{picture}(0,0)%
\includegraphics{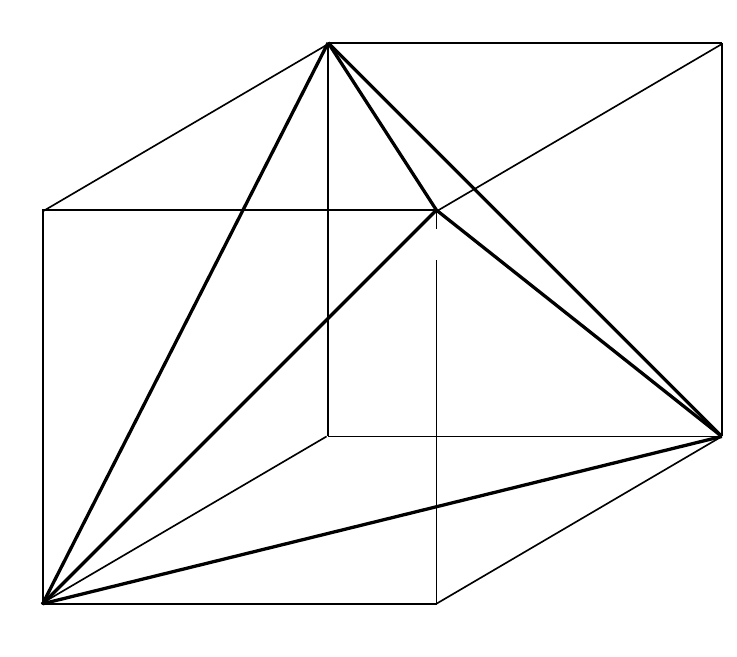}%
\end{picture}%
\setlength{\unitlength}{4144sp}%
\begingroup\makeatletter\ifx\SetFigFont\undefined%
\gdef\SetFigFont#1#2#3#4#5{%
  \reset@font\fontsize{#1}{#2pt}%
  \fontfamily{#3}\fontseries{#4}\fontshape{#5}%
  \selectfont}%
\fi\endgroup%
\begin{picture}(3360,2985)(-194,-1201)
\put(-134,-1141){\makebox(0,0)[lb]{\smash{{\SetFigFont{10}{12.0}{\rmdefault}{\mddefault}{\updefault}{\color[rgb]{0,0,0}$u_0$}%
}}}}
\put(1801,-1096){\makebox(0,0)[lb]{\smash{{\SetFigFont{10}{12.0}{\rmdefault}{\mddefault}{\updefault}{\color[rgb]{0,0,0}$v_2$}%
}}}}
\put(3151,-331){\makebox(0,0)[lb]{\smash{{\SetFigFont{10}{12.0}{\rmdefault}{\mddefault}{\updefault}{\color[rgb]{0,0,0}$u_1$}%
}}}}
\put(3106,1649){\makebox(0,0)[lb]{\smash{{\SetFigFont{10}{12.0}{\rmdefault}{\mddefault}{\updefault}{\color[rgb]{0,0,0}$v_0$}%
}}}}
\put(1216,1649){\makebox(0,0)[lb]{\smash{{\SetFigFont{10}{12.0}{\rmdefault}{\mddefault}{\updefault}{\color[rgb]{0,0,0}$u_2$}%
}}}}
\put(-179,884){\makebox(0,0)[lb]{\smash{{\SetFigFont{10}{12.0}{\rmdefault}{\mddefault}{\updefault}{\color[rgb]{0,0,0}$v_1$}%
}}}}
\put(1351,-151){\makebox(0,0)[lb]{\smash{{\SetFigFont{10}{12.0}{\rmdefault}{\mddefault}{\updefault}{\color[rgb]{0,0,0}$v_3$}%
}}}}
\put(1756,659){\makebox(0,0)[lb]{\smash{{\SetFigFont{10}{12.0}{\rmdefault}{\mddefault}{\updefault}{\color[rgb]{0,0,0}$u_3$}%
}}}}
\end{picture}%
\end{center}
\caption{A cube associated with a consistent labeling of an ideal tetrahedron.}
\label{fig:Cube}
\end{figure}

\begin{lemma}
For any choice of $v_0$, we have $[u_0, u_1, u_2, u_3] = [v_0, v_1, v_2, v_3]$.
\end{lemma}
\begin{proof}
Condition \eqref{eqn:CubeFaces} is invariant under exchanging $u_0, u_1, u_2, u_3$ with $v_0, v_1, v_2, v_3$.
It follows that $(c_{ij})$ is also a consistent labeling for the tetrahedron $v_0v_1v_2v_3$.
But the labels determine the cross-ratio: equation \eqref{eqn:LoxoConsist2} can be rewritten as
\[
[u_0, u_1, u_2, u_3] = \frac{1 - c_{ik}}{1 - c_{ij}c_{ik}}.
\]
It follows that the two cross-ratios coincide.
\end{proof}

In \cite{BS}, \cite[Section 6.6]{BSb} the following property of the cross-ratio system,
called \emph{three-dimensional consistency}, is formulated
(we state it modulo some variable changes).

Let $(c_{ij})$ be given such that $c_{ij} = c_{kl}$ and $c_{ij}c_{ik}c_{il} = 1$.
Choose any four points $u_0, v_1, v_2, v_3 \in \CP^1$.
Define $u_1, u_2, u_3$ through those of the equations \eqref{eqn:CubeFaces} that contain $u_0$.
Then there is $v_0 \in \CP^1$ such that all three equations \eqref{eqn:CubeFaces} that contain $v_0$ are simultaneously satisfied.

While in \cite{BS, BSb} this is proved by a direct computation, the arguments from the beginning of this section lead to a more geometric proof.
The numbers $(c_{ij})$ form a consistent labeling of the tetrahedron:
$L_{01}L_{02}L_{03} = \Id$ holds by construction, the compositions at the other vertices are the identities due to the symmetries of the ideal tetrahedron.
It follows that the point
\[
L_{23}(v_1) = L_{31}(v_2) = L_{12}(v_3) =: v_0
\]
satisfies all three equations.

\subsection{Another proof of Bianchi permutability}
\label{sec:BianchiAgain}
\begin{proof}[Proof of Theorem \ref{Bianchi}]
The proof is based on the consistency of the system \eqref{eqn:CubeFaces}.
Consider the combinatorial cube shown on Figure \ref{fig:Bianchi}.
Six of its vertices are given; our goal is to define $s_i$ as a function of $p_i, q_i, r_i$ (this function being the same for all $i$)
in such a way that the equations of Theorem \ref{Bianchi}
are satisfied.

\begin{figure}[ht]
\begin{center}
\begin{picture}(0,0)%
\includegraphics{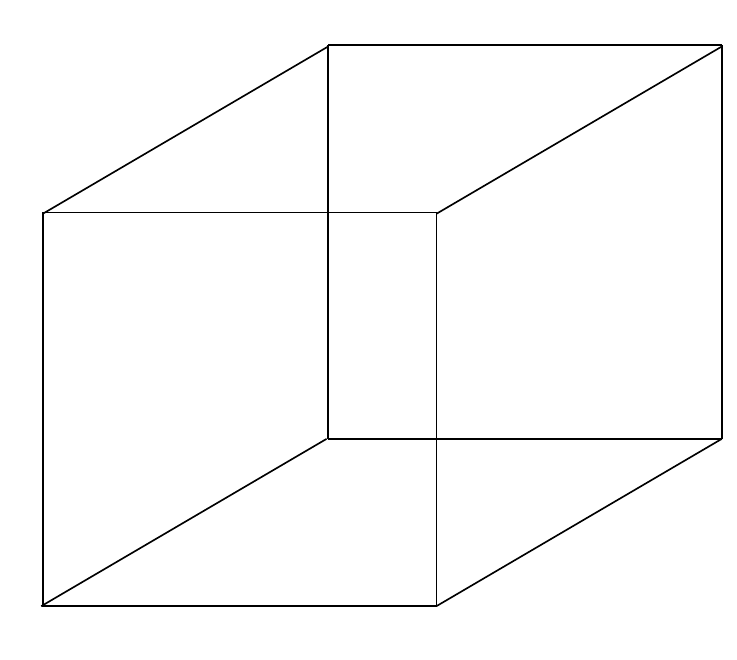}%
\end{picture}%
\setlength{\unitlength}{4144sp}%
\begingroup\makeatletter\ifx\SetFigFont\undefined%
\gdef\SetFigFont#1#2#3#4#5{%
  \reset@font\fontsize{#1}{#2pt}%
  \fontfamily{#3}\fontseries{#4}\fontshape{#5}%
  \selectfont}%
\fi\endgroup%
\begin{picture}(3360,3001)(-194,-1205)
\put(-134,-1141){\makebox(0,0)[lb]{\smash{{\SetFigFont{10}{12.0}{\rmdefault}{\mddefault}{\updefault}{\color[rgb]{0,0,0}$p_i$}%
}}}}
\put(1801,-1096){\makebox(0,0)[lb]{\smash{{\SetFigFont{10}{12.0}{\rmdefault}{\mddefault}{\updefault}{\color[rgb]{0,0,0}$p_{i+1}$}%
}}}}
\put(3151,-331){\makebox(0,0)[lb]{\smash{{\SetFigFont{10}{12.0}{\rmdefault}{\mddefault}{\updefault}{\color[rgb]{0,0,0}$q_{i+1}$}%
}}}}
\put(3106,1649){\makebox(0,0)[lb]{\smash{{\SetFigFont{10}{12.0}{\rmdefault}{\mddefault}{\updefault}{\color[rgb]{0,0,0}$s_{i+1}$}%
}}}}
\put(1216,1649){\makebox(0,0)[lb]{\smash{{\SetFigFont{10}{12.0}{\rmdefault}{\mddefault}{\updefault}{\color[rgb]{0,0,0}$s_i$}%
}}}}
\put(-179,884){\makebox(0,0)[lb]{\smash{{\SetFigFont{10}{12.0}{\rmdefault}{\mddefault}{\updefault}{\color[rgb]{0,0,0}$r_i$}%
}}}}
\put(1351,-151){\makebox(0,0)[lb]{\smash{{\SetFigFont{10}{12.0}{\rmdefault}{\mddefault}{\updefault}{\color[rgb]{0,0,0}$q_i$}%
}}}}
\put(1846,749){\makebox(0,0)[lb]{\smash{{\SetFigFont{10}{12.0}{\rmdefault}{\mddefault}{\updefault}{\color[rgb]{0,0,0}$r_{i+1}$}%
}}}}
\end{picture}%
\end{center}
\caption{Three-dimensional consistency and Bianchi permutability.}
\label{fig:Bianchi}
\end{figure}
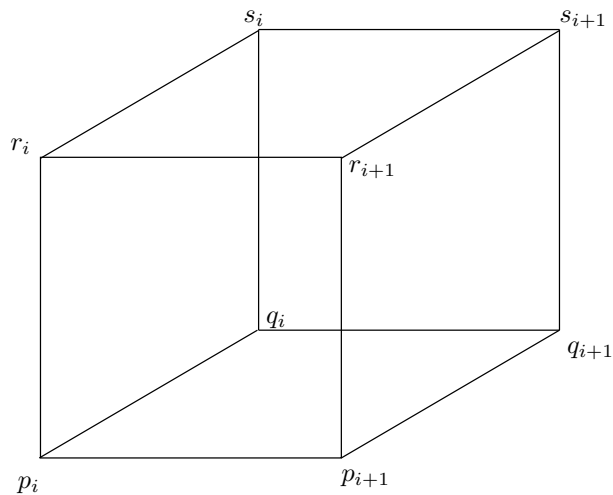

If we superpose Figure \ref{fig:Bianchi} with Figure \ref{fig:Cube}, then we have
\begin{gather*}
[u_0, u_1, v_2, v_3] = [p_i, q_{i+1}, p_{i+1}, q_i] = \frac{\alpha - 1}{\alpha},\\
[u_0, u_3, v_1, v_2] = [p_i, r_{i+1}, r_i, p_{i+1}] = \frac{\beta}{\beta - 1}.
\end{gather*}
Put $c_{01} = \frac{\alpha - 1}{\alpha}$, $c_{03} = \frac{\beta}{\beta - 1}$, and define $c_{02} = \frac{1}{c_{01}c_{03}}$.
The three-dimensional consistency of the cross-ratio system (as formulated at the end of Section \ref{sec:TetraToCube})
implies that for $s_i$ and $s_{i+1}$, defined by
\[
[p_i, s_i, q_i, r_i] = [p_{i+1}, s_{i+1}, q_{i+1}, r_{i+1}] = c_{02},
\]
the equations in Theorem \ref{Bianchi} will be satisfied.
\end{proof}


\begin{thebibliography}{99}

\bibitem{ABS} V. Adler, A. Bobenko, Yu. Suris. 
{\it Classification of integrable equations on quad-graphs. The consistency approach.}
 Comm. Math. Phys. {\bf 233} (2003), 513--543. 

\bibitem{ASTW}
T. Aougab, X. Sun, S. Tabachnikov, Y. Wang.
{\it On curves and polygons with the equiangular chord property.}
 Pacific J. Math. {\bf 274} (2015), 305--324.
 
 \bibitem{AFITT}
 M. Arnold, D. Fuchs, I. Izmestiev, S. Tabachnikov, E. Tsukerman. {\it Iterating evolutes and involutes.}
  Discrete Comput. Geom. {\bf 58} (2017), 80--143. 
  
  \bibitem{Arn} V. Arnold. {\em Mathematical methods of classical mechanics}. Springer-Verlag, New York, 1989.


\bibitem{Berger} M. Berger. {\em Geometry}. Springer-Verlag, Berlin, 2009.

\bibitem{BP} A. Bobenko, U. Pinkall.
{\it Discrete isothermic surfaces.}
J. Reine Angew. Math. {\bf 475} (1996), 187--208. 

\bibitem{BS} A. Bobenko, Yu. Suris.
{\it Integrable systems on quad-graphs.}
Int. Math. Res. Notices {\bf 11} (2002), 573--611.

\bibitem{BSb} A. Bobenko,  Yu. Suris.
{\it Discrete differential geometry. Integrable structure.}
Amer. Math. Soc., Providence, RI, 2008.

\bibitem{ConOvs17}
C. H. Conley, V. Ovsienko.
{\it Rotundus: triangulations, Chebyshev polynomials, and Pfaffians.}
Math. Intelligencer {\bf 40} (2018), no. 3, 45--50. 

\bibitem{CC}
R. Connelly, B. Csik\'os.
{\it Classification of first-order flexible regular bicycle polygons.}
Studia Sci. Math. Hungar. {\bf 46} (2009), 37--46.

\bibitem{Cox71}
H. S. M. Coxeter.
{\it Frieze patterns.}
Acta Arith. {\bf 18} (1971), 297--310.

\bibitem{Dav}
P. Davis. {\em Circulant matrices}.
John   Wiley \& Sons, New York-Chichester-Brisbane, 1979.

\bibitem{Evr}
C. Evripidou, P. van~der Kamp,  C. Zhang.
{\it Dressing the dressing chain.}
SIGMA Symmetry Integrability Geom. Methods Appl. {\bf 14} (2018), Paper No. 059, 14 pp.

\bibitem{FT}
D. Fuchs, S. Tabachnikov. {\it Iterating evolutes of spacial polygons and of spacial curves.}
 Mosc. Math. J. {\bf 17} (2017),  667--689.
 
 \bibitem{Ga} C. F. Gauss. {\it Pentagramma Mirificum}, Werke, Bd. III, 481 - 490; Bd VIII, 106--111.
 
\bibitem{GSV} 
 M. Gekhtman, M.  Shapiro, A. Vainshtein.
{\it Cluster algebras and Poisson geometry.} 
Amer. Math. Soc., Providence, RI, 2010. 

\bibitem{Gl} M. Glick. {\it The pentagram map and $Y$-patterns.} Adv. Math. {\bf 227} (2011), 1019--1045.

\bibitem{HMNP} U. Hertrich-Jeromin, I. McIntosh, P. Norman, F. Pedit. {\it Periodic discrete conformal maps.} J. Reine Angew. Math. {\bf 534} (2001), 129--153.

\bibitem{Izm}
I. Izmestiev.
{\it A porism for cyclic quadrilaterals, butterfly theorems, and hyperbolic geometry.}
Amer. Math. Monthly {\bf 122} (2015),  467--475.


\bibitem{Ma}
A. Marden. {\em Outer circles. An introduction to hyperbolic 3-manifolds}.
Cambridge Univ. Press, Cambridge, 2007.

\bibitem{Mor}
S. Morier-Genoud.
{\it Coxeter's frieze patterns at the crossroads of algebra, geometry and   combinatorics.}
Bull. Lond. Math. Soc. {\bf 47} (2015), 895--938.

\bibitem{MOST}
S. Morier-Genoud, V. Ovsienko, R. Schwartz, S.   Tabachnikov.
{\it Linear difference equations, frieze patterns, and the combinatorial Gale transform.}
Forum Math. Sigma {\bf 2} (2014), e22, 45pp.

\bibitem{MOT}
S. Morier-Genoud, V. Ovsienko,  S. Tabachnikov.
{\it $2$-frieze patterns and the cluster structure of the space of   polygons.}
Ann. Inst. Fourier  {\bf 62} (2012), 937--987.

\bibitem{Muir}
T. Muir.  {\em A treatise on the theory of determinants}.
Dover Publications, Inc., New York, 1960.

\bibitem{NC} F. Nijhoff, H. Capel.
{\it The discrete Korteweg-de Vries equation.}  Acta Appl. Math. {\bf 39} (1995), 133--158. 

\bibitem{OST1}
V. Ovsienko, R. Schwartz, S. Tabachnikov.
{\it The pentagram map: a discrete integrable system.}
Comm. Math. Phys. {\bf 299} (2010),  409--446.

\bibitem{OST2}
V. Ovsienko, R. Schwartz, S. Tabachnikov.
{\it Liouville-Arnold integrability of the pentagram map on closed   polygons.}
Duke Math. J. {\bf 162} (2013), 2149--2196.


\bibitem{OT2} V. Ovsienko, S. Tabachnikov. {\it Coxeter's frieze patterns and discretization of the Virasoro orbit.} J. Geom. Phys. {\bf 87} (2015), 373--381.

\bibitem{Pen}
R. Penner. {\em Decorated Teichm\"uller theory}.
European Math. Soc., Z\"urich,   2012.

\bibitem{Sch} R. Schwartz. {\it Discrete monodromy, pentagrams, and the method of condensation.}
 J. Fixed Point Theory Appl. {\bf 3} (2008),  379--409.
 
\bibitem{Sol} F. Soloviev. {\it Integrability of the pentagram map.} Duke Math. J. {\bf 162} (2013),  2815--2853. 

\bibitem{Suris97}
Yu. Suris.
{\it Integrable discretizations for lattice system: local equations of motion and their Hamiltonian properties.}
Rev. Math. Phys. {\bf 11} (1999), no. 6, 727--822.

\bibitem{SurB} Yu. Suris. {\it The problem of integrable discretization: Hamiltonian approach.}  Birkh\"auser Verlag, Basel, 2003.

\bibitem{Tab}
S. Tabachnikov. {\it Tire track geometry: variations on a theme.}
Israel J. Math. {\bf 151} (2006), 1--28.

\bibitem{Tab1}
S. Tabachnikov. {\it On centro-affine curves and B\"acklund transformations of the KdV equation.}
arXiv:1808.08454

\bibitem{SV}
A. Veselov, A. Shabat.
{\it A dressing chain and the spectral theory of the Schr\"odinger   operator.}
Funct. Anal. Appl. {\bf 27} (1993), 81--96.

\bibitem{Wanner}
G. Wanner.
{\it The Cramer-Castillon problem and Urquhart's ``most elementary'' theorem.}
Elem. Math. {\bf 61} (2006), 58--64. 

\end{thebibliography}
\end{document}